\documentclass[11pt]{article}
\usepackage[a4paper,margin=1in]{geometry}
\usepackage{graphicx}
\usepackage{amsmath,amssymb,amsfonts,amsthm,mathtools,bm,mathrsfs}
\usepackage{enumitem,booktabs,xcolor}
\usepackage[numbers,sort&compress]{natbib}
\usepackage[hidelinks]{hyperref}
\theoremstyle{plain}
\newtheorem{theorem}{Theorem}[section]
\newtheorem{proposition}[theorem]{Proposition}
\newtheorem{lemma}[theorem]{Lemma}
\newtheorem{corollary}[theorem]{Corollary}
\theoremstyle{definition}
\newtheorem{definition}[theorem]{Definition}
\newtheorem{assumption}[theorem]{Assumption}
\newtheorem{example}[theorem]{Example}
\theoremstyle{remark}
\newtheorem{remark}[theorem]{Remark}
\newcommand{\R}{\mathbb{R}}
\newcommand{\Rd}{\mathbb{R}^d}
\newcommand{\dd}{\,\mathrm{d}}
\newcommand{\dist}{\operatorname{dist}}
\newcommand{\proj}{\operatorname{proj}}
\newcommand{\osc}{\operatorname{osc}}

\newcommand{\MMD}{\mathrm{MMD}}
\newcommand{\mcP}{\mathcal{P}}

\newcommand{\mcL}{\mathcal{L}}
\newcommand{\mcS}{\mathcal{S}}
\newcommand{\mcT}{\mathcal{T}}
\newcommand{\one}{\mathbf{1}}
\newcommand{\eps}{\varepsilon}
\newcommand{\Tr}{\operatorname{Tr}}
\newcommand{\divg}{\nabla\!\cdot\!}

\newcommand{\wl}{w_\lambda}
\newcommand{\muh}{\widehat\mu_\lambda}
\newcommand{\hh}{\widehat h}
\newcommand{\Hm}{\mathcal{H}^{d-1}}
\newcommand{\TV}{\mathrm{TV}}
\newcommand{\dBL}{d_{\mathrm{BL}}}
\newcommand{\dob}{\mathfrak{d}}
\newcommand{\rli}{\rho_{\lambda,\infty}}
\newcommand{\ri}{\rho_{\infty}}

\begin{document}
\title{Moreau--Yosida Penalization of McKean--Vlasov Dynamics: Full-Flux Limits and Boundary Reaction Measures}
\author{Geuntaek Seo\textsuperscript{1}\thanks{Geuntaek Seo and Beomjun Choi contributed equally and share first authorship.}\quad
Beomjun Choi\textsuperscript{2}\footnotemark[1]\quad
Yeoneung Kim\textsuperscript{3}\thanks{Corresponding author.}}
\date{}
\maketitle
\begin{center}
\small
\textsuperscript{1}Department of Mathematics,\\
Pohang University of Science and Technology, Pohang, Republic of Korea\\
\texttt{gtseo@postech.ac.kr}\\[0.5em]
\textsuperscript{2}Department of Mathematical Sciences,\\
Korea Advanced Institute of Science and Technology, Daejeon, Republic of Korea\\
\texttt{bchoi@kaist.ac.kr}\\[0.5em]
\textsuperscript{3}Department of Industrial Engineering,\\
Yonsei University, Seoul, Republic of Korea\\
\texttt{yeoneung@yonsei.ac.kr}
\end{center}
\begin{abstract}
We study McKean--Vlasov Fokker--Planck dynamics on a compact convex set
$K\subset\Rd$ through the whole-space Moreau--Yosida penalization
$U_\lambda=\dist(\cdot,K)^2/(2\lambda)$, keeping the external and interaction
potentials fixed. At finite times, the penalized densities converge to the
unique variational solution of the reflected no-flux equation, strongly in
space--time $L^2$ on $K$ and in whole-space $L^1$, and narrowly, uniformly
in time. We prove that the singular
penalty-force measures converge weakly-$*$ to the purely normal boundary
reaction $\Tr(\rho)n\,\Hm\!\restriction_{\partial K}\,\dd t$.
This measure cancels exactly the surface singularity in the distributional
gradient of the zero-extended density, explaining the emergent no-flux
condition at the flux level. At equilibrium, we establish sharp
$\sqrt\lambda$ rates in $L^1$ and Wasserstein distances for kernel-free Gibbs
states, and in $L^1$ and quadratic Wasserstein distance for interacting
fixed points under Dobrushin smallness or positive semidefiniteness,
including maximum mean discrepancy interactions. The proof combines a
tilted Gaussian supersolution, nonautonomous weighted energy estimates,
and normal factorization with a Green-formula identification. We also
identify the limit with a nonlinear reflecting diffusion, obtain path-space
convergence and $\lambda$-uniform propagation of chaos through monotone
synchronous coupling, and derive all-time Gibbs domination in the reversible
case. Numerical experiments illustrate the boundary reaction and stationary
scalings on a constrained maximum mean discrepancy flow.
\end{abstract}

\noindent\textbf{Keywords:} Moreau--Yosida penalization, McKean--Vlasov equation, full-flux convergence, no-flux boundary condition, boundary reaction measure, reflected diffusion\par
\noindent\textbf{MSC Classification:} 35K61, 35B25, 35Q84, 60H30, 65M15\par

\nocite{AlasioBrunaCarrillo2020,AlexanderKim2016,AGS,ArbelKorbaSalimGretton2019,BKRS,BrosseDurmusMoulinesPereyra2018,BubeckEldanLehec2018,CarrilloMcCannVillani2003,GurbuzbalabanHuZhu2024,Hairer2021,HindsSharmaTretyakov2025,JKO,LaayounMissaoui2026,Lamperski2021,LionsMagenes,LionsSznitman1984,ManeBoukardagha2025,Menaldi1983,NingWu2026,Peyre2018,Santambrogio,Slominski1994,Slominski2013,Sznitman1984,Sznitman1991,Tanaka1979,YuZhangYu2026}

\section{Introduction}\label{sec:intro}

We consider interacting diffusions whose states must remain in a compact convex
set $K\subset\Rd$ with nonempty interior and $C^2$ boundary. Such constraints arise, for example, from trust
regions, admissible physical parameters, or smooth surrogates of boxes and
simplices in constrained sampling and optimization. One standard way to
enforce this constraint is to add an
inward normal correction that acts only when a trajectory lies on the
boundary and keeps it inside $K$. This construction is called
\emph{normal reflection} \cite{Tanaka1979,LionsSznitman1984}. At the density
level, it gives a no-flux boundary condition: no probability flows across
$\partial K$. Soft confinement approximates this boundary correction by a
restoring drift outside $K$, allowing the approximating trajectories to move
in the whole space. We study the canonical quadratic Moreau--Yosida wall
as its penalty parameter $\lambda$ tends to zero, keeping the external and
interaction potentials fixed. Although the mass outside $K$ vanishes, the penalty-force measures converge to a normal reaction measure supported on $\partial K$.

Before imposing the constraint, a pairwise-gradient
particle model is given by the stochastic differential equations (SDEs)
\begin{equation}\label{eq:particles-free-intro}
  \dd X_t^{i,N}
  =-\left[
      \nabla\Phi(X_t^{i,N})
      +\frac1N\sum_{j=1}^N
       \nabla_xW(X_t^{i,N},X_t^{j,N})
    \right]\dd t
  +\sqrt2\,\dd B_t^i,
\end{equation}
for $i=1,\ldots,N$, with independent Brownian motions $B^1,\ldots,B^N$.
Assume that the initial
positions are i.i.d.\ with density $\rho_0$ and are independent of these
Brownian motions. Under the bounded Lipschitz drift assumptions stated in
Assumption~\ref{ass:coefficients}, the empirical measures
$\rho_t^N=N^{-1}\sum_i\delta_{X_t^{i,N}}$ converge narrowly in probability,
for each fixed time as $N\to\infty$, to the probability solution of the McKean--Vlasov Fokker--Planck equation
\cite{Sznitman1991}
\begin{equation}\label{eq:mv-free}
  \partial_t\rho=\divg\bigl(\nabla\rho+\rho\,b[\rho_t]\bigr),
  \qquad
  b[\mu](x):=\nabla\Phi(x)+\int_{\Rd}\nabla_xW(x,y)\,\dd\mu(y).
\end{equation}
Here $\Phi$ collects external and data-attraction potentials and $W$ is a
symmetric interaction kernel.

For the normally reflected process in $K$, the density solves
\eqref{eq:mv-free} in $K$, with
\[
  \bigl(\nabla\rho+\rho\,b[\rho_t]\bigr)\cdot n=0
  \qquad\text{on }\partial K,
\]
where $n$ is the outward unit normal.
The boundary correction in the particle equation is represented by a
nondecreasing regulator, often called boundary local time, whose increments
are supported on boundary-contact times; see \eqref{eq:rmv-sde} for its
normalization. The whole-space approximation replaces this correction by
$-\nabla U_\lambda$, where
\[
  U_\lambda(x):=\frac{\dist(x,K)^2}{2\lambda},
  \qquad
  \nabla U_\lambda(x)=\frac1\lambda\bigl(x-\proj_Kx\bigr).
\]
The resulting continuous-time particle system is
\begin{equation}\label{eq:particles}
\begin{aligned}
  \dd X_t^{i,N,\lambda}
  ={}&-\Bigg[
      \nabla\Phi(X_t^{i,N,\lambda})
      +\frac1N\sum_{j=1}^N
       \nabla_xW\bigl(X_t^{i,N,\lambda},X_t^{j,N,\lambda}\bigr)\\[-0.2em]
  &\hspace{5.2em}
      +\frac1\lambda
       \bigl(X_t^{i,N,\lambda}-\proj_KX_t^{i,N,\lambda}\bigr)
    \Bigg]\dd t
  +\sqrt2\,\dd B_t^i .
\end{aligned}
\end{equation}

For each fixed $\lambda>0$, the same initial independence assumptions and
propagation-of-chaos theory give the mean-field equation
\begin{equation}\label{eq:mv-pen-intro}
  \partial_t\rho^\lambda
  =\divg\bigl(
       \nabla\rho^\lambda
       +\rho^\lambda\nabla U_\lambda
       +\rho^\lambda b[\rho_t^\lambda]
       \bigr)
  \qquad\text{in }(0,T)\times\Rd.
\end{equation}

Writing $\rho^{N,\lambda}$ for the empirical measure of these particles,
the approximation has the two-step form
\[
  \rho^{N,\lambda}
  \xrightarrow[N\to\infty]{\text{mean field}}
  \rho^\lambda
  \xrightarrow[\lambda\downarrow0]{\text{soft to hard}}
  \rho\one_K.
\]
Here the limiting density on $K$ is extended by zero to $\Rd$.
The first arrow is classical for fixed $\lambda$.

The second arrow is singular. At distances of order $\sqrt\lambda$ from
$K$ the restoring drift has magnitude of order $\lambda^{-1/2}$, so the
penalty part of the \emph{full flux}
\[
  J^\lambda:=\nabla\rho^\lambda+\rho^\lambda\nabla U_\lambda
                  +\rho^\lambda b[\rho_t^\lambda],
  \qquad \partial_t\rho^\lambda=\divg J^\lambda,
\]
develops a singular boundary-measure limit, even though the mass outside
$K$ vanishes. The difficulty is therefore not merely to recover a weak no-flux
formulation: one must determine in which topologies $\rho^\lambda$
converges, which flux is stable, and what this boundary measure is. Our
answer is that the three components of the full flux
converge separately in distributions on $(0,T)\times\Rd$:
\begin{equation}\label{eq:cancellation-intro}
  \begin{aligned}
    \nabla\rho^\lambda
      &\longrightarrow\one_K\nabla\rho
        -\Tr(\rho)n\,\Hm\!\restriction_{\partial K},\\
    \rho^\lambda\nabla U_\lambda
      &\longrightarrow\Tr(\rho)n\,\Hm\!\restriction_{\partial K},\\
    \rho^\lambda b[\rho^\lambda_t]
      &\longrightarrow\one_K\rho\,b[\rho_t],
  \end{aligned}
  \qquad\text{hence}\qquad
  J^\lambda\longrightarrow\one_KJ,
\end{equation}
with Lebesgue measure in time understood; here
$J:=\nabla\rho+\rho\,b[\rho_t]$ and $\Tr(\rho)$ is the interior $H^1(K)$
trace. The two surface measures cancel.

For symmetric $W$, the equations have the classical Wasserstein
gradient-flow structure for entropy, potential, and interaction energies
\cite{JKO,AGS,CarrilloMcCannVillani2003,Santambrogio}. One example in our
class is obtained by adding entropy to the maximum mean discrepancy (MMD)
interaction energy studied by Arbel et al.~\cite{ArbelKorbaSalimGretton2019}; see
Example~\ref{ex:mmd}. Their basic MMD flow is a transport equation, whereas
the entropy term produces the diffusion considered here. The diffusion
coefficient does not depend on $\lambda$. In the kernel-free case $W\equiv0$ and
$\Phi=f$, the equation reduces to penalized overdamped Langevin dynamics,
with explicit Gibbs density proportional to $e^{-f-U_\lambda}$.

\paragraph{Contributions.}

The results concern finite-time dynamical limits, stationary limits, and
their probabilistic interpretation. Further hypotheses are stated with
each result.

\begin{enumerate}[label=(\roman*),leftmargin=2.2em]

\item \emph{Dynamical full-flux limit and boundary reaction.}

Under Assumptions~\ref{ass:geometry}, \ref{ass:coefficients}, and
\ref{ass:data}, the penalty-force measures have uniformly bounded total
variation and converge weakly-$*$ in $\mathcal M([0,T]\times\Rd;\Rd)$:
\[
  Q^\lambda:=\rho^\lambda\nabla U_\lambda\,\dd x\,\dd t
  =\frac1\lambda\rho^\lambda(x-\proj_Kx)\,\dd x\,\dd t
  \rightharpoonup^*
  \Tr(\rho)n\,\Hm\!\restriction_{\partial K}\,\dd t,
\]
and the three components of $J^\lambda$ converge as in
\eqref{eq:cancellation-intro}; see Theorem~\ref{thm:main}. Along the way,
the penalized densities converge to the unique variational solution of
the reflected McKean--Vlasov equation, strongly in $L^2((0,T)\times K)$
and in $L^1((0,T)\times\Rd)$ after zero extension of the limit, and the
time marginals converge narrowly, uniformly on $[0,T]$. The exterior mass
is $O(\sqrt\lambda)$ uniformly in time, the full flux converges weakly in
$L^2((0,T)\times K)$ to $J$, and its $L^1((0,T)\times K^c)$ norm is
$O(\lambda^{1/4})$. 

\item \emph{Quantitative stationary soft-to-hard limits.}

Theorem~\ref{thm:stationary-main} establishes existence and uniqueness of
the interacting Gibbs fixed points and two-sided $\sqrt\lambda$ bounds in
$L^1$ and $W_2$, under the geometry and coefficient hypotheses, with $\Phi$
bounded and either the gauge-invariant Dobrushin condition $\kappa<1$ or
positive semidefiniteness of $W$. The positive semidefinite case includes
the Gaussian MMD kernels of Example~\ref{ex:mmd}, with no smallness
restriction on their amplitude. In the kernel-free case $W\equiv0$,
Proposition~\ref{prop:kernel--free} assumes only that $\Phi=f\in C^2(\Rd)$
is bounded below and gives an exact $L^1$ identity together with two-sided
$\sqrt\lambda$ bounds in total variation, $W_1$, and $W_2$.

\item \emph{Probabilistic interpretation.}

Corollary~\ref{cor:probabilistic} identifies the limiting variational
solution with the time marginals of the nonlinear normally reflected
diffusion~\eqref{eq:rmv-sde}. The laws of the penalized McKean--Vlasov
diffusions converge narrowly to these marginals, uniformly in time, and
their averaged penalty-force occupation measures converge weakly-$*$ to
the boundary reaction measure of~(i). At the level of paths,
Proposition~\ref{prop:pathwise} combines monotone synchronous coupling
with a classical prescribed-drift penalization estimate to obtain
$\lambda$-uniform propagation of chaos and the following joint bound,
when all processes share the same i.i.d.\ initial data and Brownian
motions:
\[
  \mathbb E\sup_{t\le T}|X_t^{i,N,\lambda}-X_t^i|
  \le C\left[N^{-1/2}
       +\bigl(\lambda\log(1/\lambda)\bigr)^{1/4}\right],
  \qquad 0<\lambda\le\tfrac12,
\]
with $C$ independent of $N$ and $\lambda$, so the particle error and the
soft-wall bias simply add. The same proposition shows that the accumulated
penalty of the penalized McKean--Vlasov process converges to the
reflection term $\int_0^\cdot n(X_s)\dd L_s$ at the rate
$(\lambda\log(1/\lambda))^{1/4}$ in expected uniform norm on $[0,T]$.

\end{enumerate}

\paragraph{Main ideas.}
The dynamical proof combines a Gaussian barrier uniform over the
admissible drifts with weighted energy estimates in the spirit of
\cite{AlasioBrunaCarrillo2020} and full-flux compactness.
Normal factorization and the Gauss--Green formula then identify the
boundary reaction.
The stationary $W_2$ estimate combines projection of the exterior layer
onto $\partial K$ with control of the interior redistribution in
$\dot H^{-1}(K)$, using a Benamou--Brenier argument in the spirit of
\cite{Peyre2018}.


\paragraph{Related literature.}

Classical constructions of reflected diffusions and the associated
Skorokhod problem include Tanaka~\cite{Tanaka1979} and
Lions--Sznitman~\cite{LionsSznitman1984}. For bounded convex domains,
Menaldi~\cite[Thm.~3.1 and (3.19)--(3.20)]{Menaldi1983} proves convergence
of the penalized state and accumulated penalty in $L^p$ of the uniform path
norm. S{\l}omi\'nski~\cite{Slominski1994} studies Euler and Euler--Peano
approximations of reflected SDEs; quantitative continuous-time penalization
estimates are given in \cite[Thm.~4.1]{Slominski2013}. The prescribed-drift
estimate used in Proposition~\ref{prop:pathwise} is specifically
\cite[Thm.~4.1(ii)]{Slominski2013}.
Sznitman~\cite{Sznitman1984} constructed nonlinear reflecting diffusions
and proved propagation of chaos; \cite{Sznitman1991} gives a general account
of propagation of chaos. 
Hinds, Sharma, and Tretyakov~\cite{HindsSharmaTretyakov2025}
establish well-posedness and propagation of chaos for reflected
McKean--Vlasov SDEs on bounded, possibly nonconvex domains
with $C^3$ boundary.
Ning and Wu~\cite{NingWu2026} establish well-posedness and propagation of
chaos for one-dimensional McKean--Vlasov stochastic variational
inequalities under local regularity and growth assumptions.
Laayoun and Missaoui~\cite{LaayounMissaoui2026} use penalization and relaxed
controls to approximate reflected mean field game equilibria.

A particularly close precursor is Alasio, Bruna, and
Carrillo~\cite{AlasioBrunaCarrillo2020}. Their Theorem~1 includes linear
diffusion with a nonzero interaction and obtains a weak no-flux equation
from whole-space equations with increasingly strong exterior confinement.
Their weighted $L^2$ estimates and compactness argument already control
both density and flux: Section~3.4 passes to the interior flux limit and
shows that exterior full-flux terms vanish. Their estimates imply
full-flux convergence against smooth compactly supported vector fields.
Thus density convergence and distributional full-flux convergence have
close precedents in that work. Their confining potentials interpolate
across a shrinking collar and are not the fixed-potential additive
quadratic wall considered here. Alexander and
Kim~\cite{AlexanderKim2016} study a related whole-space approximation of
local parabolic PDEs with homogeneous oblique boundary data, using a cubic
distance penalty and obtaining quantitative uniform error estimates.

In the present model, $\Phi$ and $W$ are fixed on all of $\Rd$ and only
$U_\lambda=\dist^2/(2\lambda)$ varies. Our main dynamical contribution is a
measure-level convergence theorem for the explicit penalty force under a
law-dependent drift, namely the weak-$*$ convergence of $Q^\lambda$
and the decomposition \eqref{eq:cancellation-intro}. The trace
identification itself rests on full-flux convergence and Gauss--Green; the
new input is the control of the force measures supplied by the quadratic
wall.

For constrained sampling, Brosse et
al.~\cite[Props.~4--5]{BrosseDurmusMoulinesPereyra2018} use the same
quadratic Moreau--Yosida wall and obtain stationary total-variation and
$W_1$ bias bounds of order $\sqrt\lambda$ for compactly supported
log-concave targets. G{\"u}rb{\"u}zbalaban, Hu, and
Zhu~\cite[Thm.~2.7]{GurbuzbalabanHuZhu2024} also give a stationary $W_2$
bias bound without convexity of the target potential, at the weaker order
$O((\lambda\log(1/\lambda))^{1/8})$, up to a constant rescaling of the
penalty parameter. Their paper separately treats sampling-algorithm
errors in total variation and $W_2$. Mane and
Boukardagha~\cite[Thm.~3.12]{ManeBoukardagha2025} study polynomial
stationary $W_2$ convergence under strong convexity. 
For fixed potentials, stationary Wasserstein bounds for quadratic
penalization are also studied by Yu, Zhang, and Yu~\cite{YuZhangYu2026}.
Our stationary result gives two-sided $W_2\asymp\sqrt\lambda$ bounds under the hypotheses above,
without log-concavity, and extends them to interacting Gibbs fixed points.
Projection-based constrained Langevin algorithms are studied by Bubeck,
Eldan, and Lehec~\cite{BubeckEldanLehec2018} and
Lamperski~\cite{Lamperski2021}.

\paragraph{Organization.}
Section~\ref{sec:setting} states the assumptions and main results, and
Section~\ref{sec:wp} gives the well-posedness preliminaries.
Sections~\ref{sec:dynamical-proof} and \ref{sec:stationary-proof} prove
the dynamical and stationary limits. Section~\ref{sec:probabilistic}
treats the kernel-free and stochastic consequences, and
Section~\ref{sec:numerics} gives the numerical illustration. The
appendices collect geometry, well-posedness, and comparison arguments.

\paragraph{Notation.}
$\mcP(\Rd)$ is the set of Borel probability measures; measures with Lebesgue
densities are identified with their densities.  The bounded-Lipschitz metric
is $\dBL(\mu,\nu):=\sup\{\int\varphi\dd(\mu-\nu):\|\varphi\|_\infty
+\|\nabla\varphi\|_\infty\le1\}$; on $\mcP(\Rd)$ it metrizes narrow
convergence along tight sequences.  For $x\in\Rd$,
$\dist(x)=\dist(x,K)$ and $\proj_Kx$ is the (unique, by convexity) closest
point of $K$.  We write $\Gamma:=\partial K$, $n(y)$ for the outer unit
normal at $y\in\Gamma$, and $n(x):=\nabla\dist(x)=(x-\proj_Kx)/\dist(x)$ for
$x\notin K$.  The penalty potential is $U_\lambda:=\dist^2/2\lambda$, the
penalty weight $\wl:=e^{-U_\lambda}$, and $\muh:=\wl\dd x$ (not normalized);
note $\wl\equiv1$ on $K$.  Sobolev spaces on $K$ are those on
$\operatorname{int}K$, with traces taken from the interior, and fluxes
follow the sign convention $\partial_t\rho=\divg J$, so that the
probability current is $-J$.  For a bounded function $f$ on a set $S$,
$\osc_S f:=\sup_Sf-\inf_Sf$.  {For a finite signed
measure $\sigma$, $\|\sigma\|_{\TV}:=|\sigma|(\Rd)$; hence for densities
this is the $L^1$ norm, while the probabilistic total-variation distance is
$\tfrac12\|\sigma\|_{\TV}$.}
Constants $C$ may change from line to line. In the dynamical estimates they
may depend on $d$, $K$, $T$, the coefficient bounds in
Assumption~\ref{ass:coefficients}, and the initial-data bounds $M,M_2$ in
Assumption~\ref{ass:data}, but never on $\lambda$; other dependencies are
specified where they arise.
Throughout, $\lambda\in(0,\lambda_0]$ for a fixed
$\lambda_0\le1$ determined by $K$ and the data.

\section{Setting and main results}\label{sec:setting}

\subsection{Penalized and reflected interacting dynamics}

Fix a time horizon $T>0$.  Given a potential pair $(\Phi,W)$ as in
Assumption~\ref{ass:coefficients} below, define for $\mu\in\mcP(\Rd)$ the
drift field
\begin{equation}\label{eq:drift-def}
  b[\mu](x):=\nabla\Phi(x)+\int_{\Rd}\nabla_xW(x,y)\dd\mu(y).
\end{equation}
The \emph{penalized interacting dynamics} is the McKean--Vlasov
Fokker--Planck equation
\begin{equation}\label{eq:pen-pde}
  \partial_t\rho^\lambda=\divg J^\lambda,
  \qquad
  J^\lambda:=\nabla\rho^\lambda+\rho^\lambda\nabla U_\lambda
  +\rho^\lambda\,b[\rho^\lambda_t],
  \qquad\text{in }(0,T)\times\Rd,
\end{equation}
with initial datum $\rho^\lambda(0,\cdot)=\rho_0$.  The \emph{reflected
interacting dynamics} is
\begin{equation}\label{eq:ref-pde}
\begin{cases}
  \partial_t\rho=\divg J, & \text{in }(0,T)\times K,\\[0.2em]
  J\cdot n=0, & \text{on }(0,T)\times\Gamma,\\[0.2em]
  \rho(0,\cdot)=\rho_0, & \text{in }K,
\end{cases}
\qquad
  J:=\nabla\rho+\rho\,b[\rho_t].
\end{equation}
We shall systematically distinguish the \emph{full flux} $J^\lambda$ of
\eqref{eq:pen-pde} from the \emph{interior (unpenalized) flux}
\begin{equation}\label{eq:JK-def}
  J^\lambda_K:=\nabla\rho^\lambda+\rho^\lambda\,b[\rho^\lambda_t],
\end{equation}
which coincides with $J^\lambda$ on $K$ (where $\nabla U_\lambda=0$) but not
outside.  The no-flux condition in \eqref{eq:ref-pde} concerns the normal
trace of the limiting interior flux $J$, whereas the penalty component
$\rho^\lambda\nabla U_\lambda$ carries an $O(1)$ boundary reaction
(Theorem~\ref{thm:main}); see
Remark~\ref{rem:flux-terminology}.

For symmetric $W$, these equations have the usual Wasserstein
gradient-flow interpretation for entropy--potential--interaction energies
\cite{JKO,AGS,Santambrogio}. Here the Moreau--Yosida regularization is applied
to the Euclidean indicator $I_K$, through $U_\lambda=(I_K)_\lambda$.
The force-level analysis below uses domination and weighted $L^2$ estimates
and requires no convexity of the dynamical energy. Positive semidefiniteness
of $W$ is used only as an alternative to smallness in the stationary theory.

\begin{example}[Constrained MMD flow]\label{ex:mmd}
Let $\nu\in\mcP(K)$ be a target measure and
$k_\sigma(x,y)=\alpha\exp(-|x-y|^2/2\sigma^2)$ the Gaussian kernel with
amplitude $\alpha>0$.  The $\tfrac12\MMD_{k_\sigma}^2(\cdot,\nu)$ gradient
flow with entropic regularization corresponds to
\[
  W=k_\sigma,\qquad \Phi=-k_\sigma*\nu,
\]
so that $\Phi\in C^\infty_b(\Rd)$ and $W\in
C^\infty_b(\Rd\times\Rd)$.  Every hypothesis of the dynamical theory of this
paper is satisfied; {since $k_\sigma$ is positive definite, the
static theory (Theorem~\ref{thm:stationary-main}) applies for every amplitude
$\alpha>0$.  The Dobrushin coefficient of \eqref{eq:dob-def} satisfies
$\kappa\le\alpha/2$, so $\alpha<2$ also suffices for the contraction
alternative.}  The
constrained MMD particle algorithm with quadratic penalty is exactly the
particle discretization of \eqref{eq:pen-pde}
(see \eqref{eq:particles}).
\end{example}

\begin{example}[Penalized Langevin dynamics]\label{ex:langevin}
With $W\equiv0$ and $\Phi=f$, \eqref{eq:pen-pde} is the Fokker--Planck
equation of penalized overdamped Langevin dynamics with Gibbs density
$\rli\propto e^{-f-U_\lambda}$, while \eqref{eq:ref-pde} describes
reflected Langevin sampling on $K$. The reversible specialization and
quantitative Gibbs estimates are treated in Section~\ref{sec:probabilistic}.
\end{example}

\subsection{Assumptions}

\begin{assumption}[Geometry]\label{ass:geometry}
$K\subset\Rd$ is compact and convex with nonempty interior and $C^2$ boundary.
(The strong convergence theory of the companion paper~\cite{SeoKimChoiStrong} additionally requires $\partial K\in C^3$; nothing
in the present paper does.)
\end{assumption}

Under Assumption~\ref{ass:geometry} the projection $\proj_K$ is single-valued
and $1$-Lipschitz on $\Rd$, $\dist^2$ is $C^{1,1}(\Rd)$, $\dist$ is $C^2$ on
an exterior collar $\{0<\dist<\delta_0\}$ and $C^{1,1}_{\rm loc}$ on $K^c$,
with
\begin{equation}\label{eq:lap-dist-bound}
  0\le\Delta\dist(x)\le C_K^*,\qquad
  |D^2\dist(x)|\le C_K^*\quad\text{for a.e. }x\in K^c,
\end{equation}
where $C_K^*$ depends only on $K$ (curvature bounds in the collar and the
bound $\Delta\dist\le(d-1)/\dist$ outside it; see
Appendix~\ref{app:geometry}).  We also use the Gaussian layer estimates of
Lemma~\ref{lem:layer}.

\begin{assumption}[Coefficients]\label{ass:coefficients}
$\Phi\in C^2(\Rd)$ with $\nabla\Phi\in C^1_b(\Rd;\Rd)$, and
$W\in C^2_b(\Rd\times\Rd)$ with $W(x,y)=W(y,x)$.  Set
\[
  B_0:=\|\nabla\Phi\|_\infty+\|\nabla_xW\|_\infty,\qquad
  B_1:=\|\nabla^2\Phi\|_\infty+\|\nabla_x^2W\|_\infty,\qquad
  L_W:=\|\nabla_xW\|_\infty .
\]
\end{assumption}

Assumption~\ref{ass:coefficients} yields, for all $\mu,\nu\in\mcP(\Rd)$
with densities $\rho_\mu,\rho_\nu\in L^1$,
\begin{equation}\label{eq:drift-bounds}
  \|b[\mu]\|_\infty\le B_0,\qquad
  \|\nabla b[\mu]\|_\infty\le B_1,\qquad
  \|b[\mu]-b[\nu]\|_\infty\le L_W\,\|\rho_\mu-\rho_\nu\|_{L^1(\Rd)} .
\end{equation}
For measures without densities the same bounds hold with the total
variation in place of the $L^1$ norm and, more usefully for empirical
measures, Kantorovich--Rubinstein duality gives the metric estimate
\begin{equation}\label{eq:drift-W1}
  \|b[\mu]-b[\nu]\|_\infty
  \ \le\ \|\nabla_y\nabla_xW\|_\infty\,W_1(\mu,\nu),
  \qquad \mu,\nu\in\mcP_1(\Rd),
\end{equation}
since for each fixed $x$ the function $y\mapsto\nabla_xW(x,y)$ is
Lipschitz with constant $\|\nabla_y\nabla_xW\|_\infty\le\|W\|_{C^2}$.
In particular the drift class
\begin{equation}\label{eq:drift-class}
  \mathcal B:=\bigl\{\,b\in C^1(\Rd;\Rd):\ \|b\|_\infty\le B_0,\
  \|\nabla b\|_\infty\le B_1\,\bigr\}
\end{equation}
contains $b[\mu]$ for every probability measure $\mu$, uniformly in
$\lambda$ and $t$.  This uniformity is the structural fact on which the whole
paper rests.

\begin{remark}
Boundedness of $\nabla\Phi$ on all of $\Rd$ is a convenience: since the
penalized mass at distance $\ge1$ from $K$ is $O(e^{-c/\lambda})$
(Corollary~\ref{cor:moments}), only the values of $\Phi$ and $W$ in a fixed
neighborhood of $K$ should matter.  We emphasize, however, that all
statements are proved under the global bounds as stated; replacing a
globally confining $\Phi$ by a truncation is a localization question at the
level of the flows, which we do not pursue (see
Remark~\ref{rem:linear-truncation}).  The MMD data of
Example~\ref{ex:mmd} satisfies the assumption as stated.
\end{remark}

\begin{assumption}[Initial data]\label{ass:data}
$\rho_0\in L^\infty(K)$, $\rho_0\ge0$, $\int_K\rho_0\dd x=1$, extended by
zero to $\Rd$ (\emph{hard initialization}); set $M:=\|\rho_0\|_{L^\infty(K)}$
and $M_2:=\|\rho_0\|_{L^2(K)}$.
\end{assumption}

\subsection{Solution concepts}

\begin{definition}[Penalized solutions]\label{def:pen-sol}
A \emph{penalized solution} on $[0,T]$ is a nonnegative
$\rho^\lambda\in C([0,T];L^1(\Rd))$ with $\rho^\lambda_t\in\mcP(\Rd)$,
such that $\hh^\lambda:=\rho^\lambda/\wl\in
L^\infty(0,T;L^2(\muh))\cap L^2(0,T;H^1(\muh))$ and, for all
$\varphi\in C_c^\infty([0,T)\times\Rd)$,
\begin{equation}\label{eq:pen-weak}
  \int_{\Rd}\rho_0\varphi(0)\dd x
  +\int_0^T\!\!\int_{\Rd}\rho^\lambda\,\partial_t\varphi\dd x\dd t
  =\int_0^T\!\!\int_{\Rd}
  \bigl(\wl\nabla\hh^\lambda+\rho^\lambda\,b[\rho^\lambda_t]\bigr)
  \cdot\nabla\varphi\dd x\dd t,
\end{equation}
where $H^1(\muh):=\{v:\int(v^2+|\nabla v|^2)\dd\muh<\infty\}$.  Note that
$\wl\nabla\hh^\lambda=\nabla\rho^\lambda+\rho^\lambda\nabla U_\lambda$, so
\eqref{eq:pen-weak} is the weak form of \eqref{eq:pen-pde} with
$J^\lambda=\wl\nabla\hh^\lambda+\rho^\lambda b[\rho^\lambda_t]$.
\end{definition}

\begin{definition}[Variational solutions of the reflected equation]
\label{def:var-sol}
A \emph{variational solution} of \eqref{eq:ref-pde} is a nonnegative
\[
  \rho\in L^\infty(0,T;L^2(K))\cap L^2(0,T;H^1(K)),\qquad
  \partial_t\rho\in L^2(0,T;H^1(K)'),
\]
with $\rho(0)=\rho_0$, such that for a.e.\ $t$ and all $\varphi\in H^1(K)$,
\begin{equation}\label{eq:var-form}
  \langle\partial_t\rho(t),\varphi\rangle
  +\int_K\bigl(\nabla\rho+\rho\,b[\rho_t]\bigr)\cdot\nabla\varphi\dd x=0 .
\end{equation}
No boundary integral appears in \eqref{eq:var-form}; this encodes
$J\cdot n=0$ in the normal-trace sense.  Variational solutions belong to
$C([0,T];L^2(K))$ and conserve mass.
\end{definition}

\subsection{Main results}

Existence and uniqueness of penalized solutions is standard for fixed
$\lambda$ (Proposition~\ref{prop:pen-wp}); all constants below are uniform in
$\lambda\in(0,\lambda_0]$.  We state first the finite-time dynamical limit,
then the quantitative stationary limit, and finally a probabilistic
corollary identifying the limiting dynamics.

\begin{theorem}[Full-flux soft-to-hard limit and boundary reaction]
\label{thm:main}
Suppose that Assumptions~\ref{ass:geometry}, \ref{ass:coefficients}, and
\ref{ass:data} hold.  
Then the reflected problem
\eqref{eq:ref-pde} admits a unique variational solution $\rho$.
Moreover, if $\rho^\lambda$ is a penalized solution of
\eqref{eq:pen-pde} for each $\lambda>0$, then the following
statements hold for the entire family $(\rho^\lambda)$
as $\lambda\downarrow0$.
\begin{enumerate}[label=(\roman*),leftmargin=2.2em]
\item \emph{Density and time-marginal convergence.}
\[
  \rho^\lambda|_K\longrightarrow\rho
  \quad\text{strongly in }L^2((0,T)\times K),
  \qquad
  \rho^\lambda\longrightarrow\rho\one_K
  \quad\text{in }L^1((0,T)\times\Rd),
\]
\[
  \sup_{t\in[0,T]}
  \dBL\bigl(\rho_t^\lambda,\rho_t\one_K\bigr)\longrightarrow0. 
\]
Moreover,
$\sup_{t\le T}\int_{K^c}\rho_t^\lambda\dd x\le C\sqrt\lambda$.

\item \emph{Full-flux convergence.}
Let $J=\nabla\rho+\rho b[\rho_t]$. Then
\[
  J^\lambda\one_K\rightharpoonup J
  \quad\text{weakly in }L^2((0,T)\times K),
  \qquad
  \int_0^T\!\int_{K^c}|J^\lambda|\dd x\dd t
  \le C\lambda^{1/4}.
\]
{Consequently,
$J^\lambda\to\one_KJ$ in
$\mathcal D'((0,T)\times\Rd;\Rd)$.}

\item \emph{Boundary reaction.}
The vector measures
\[
  Q^\lambda:=\frac1\lambda\rho^\lambda(t,x)
  \bigl(x-\proj_Kx\bigr)\,\dd x\,\dd t
\]
have uniformly bounded total variation and converge weakly-$*$ in
$\mathcal M([0,T]\times\Rd;\Rd)$ to
\[
  Q:=\Tr(\rho)\,n\,\Hm\!\restriction_\Gamma\,\dd t.
\]
Likewise,
$q^\lambda:=\lambda^{-1}\dist(x)\rho^\lambda\dd x\dd t$
converges weakly-$*$ in $\mathcal M([0,T]\times\Rd)$ to
$\Tr(\rho)\Hm\!\restriction_\Gamma\dd t$.

\item \emph{Exact cancellation.}
{The three components of
$J^\lambda=\nabla\rho^\lambda+\rho^\lambda\nabla U_\lambda+\rho^\lambda
b[\rho^\lambda_t]$ converge separately in
$\mathcal D'((0,T)\times\Rd;\Rd)$:
\[
\begin{gathered}
  \nabla\rho^\lambda\to\nabla(\rho\one_K)
  =\one_K\nabla\rho-\Tr(\rho)n\,\Hm\!\restriction_\Gamma\dd t,\\
  \rho^\lambda\nabla U_\lambda\to\Tr(\rho)n\,\Hm\!\restriction_\Gamma\dd t,
  \qquad
  \rho^\lambda b[\rho^\lambda_t]\to\one_K\rho\,b[\rho_t].
\end{gathered}
\]
Thus the diffusive and penalty components carry singular parts on $\Gamma$
of equal size and opposite sign, which cancel in the limit
$J^\lambda\to\one_KJ$ of (ii),} while the remaining
interior flux has zero normal trace in the variational sense of
\eqref{eq:var-form}.  Here all gradients are spatial derivatives, with
Lebesgue measure in time understood.
\end{enumerate}
\end{theorem}

\medskip
For the stationary theory, set
$(W*\rho)(x):=\int_{\Rd}W(x,y)\rho(y)\dd y$ and define the penalized and
hard Gibbs maps
\begin{equation}\label{eq:fixed-point}
\begin{aligned}
  \mcT_\lambda(\rho)(x)
  &:=
  \frac{e^{-U_\lambda(x)-\Phi(x)-(W*\rho)(x)}}
  {\int_{\Rd}e^{-U_\lambda-\Phi-W*\rho}\dd x},
  &&\lambda>0,\\
  \mcT_0(\rho)(x)
  &:=
  \frac{\one_K(x)e^{-\Phi(x)-(W*\rho)(x)}}
  {\int_K e^{-\Phi-W*\rho}\dd x}.
\end{aligned}
\end{equation}
The gauge-invariant interaction coefficient used below is
\begin{equation}\label{eq:dob-def}
  \dob(W):=\sup_{x,x'\in\Rd}\frac12\,
  \osc_{y\in\Rd}\Bigl(W(x,y)-W(x',y)\Bigr),
  \qquad
  \kappa:=\tfrac12\dob(W).
\end{equation}

\begin{theorem}[Quantitative stationary soft-to-hard limits]
\label{thm:stationary-main}
Let Assumptions~\ref{ass:geometry} and \ref{ass:coefficients} hold.
Assume in addition that $\Phi\in C_b(\mathbb R^d)$ and
either (a) $\kappa<1$, or (b) $W$ is positive semidefinite:
$\iint W(x,y)g(x)g(y)\dd x\dd y\ge0$ for every $g\in L^1(\Rd)$.
Then
\begin{enumerate}[label=(\roman*),leftmargin=2.2em]

\item \emph{Fixed points and sharp total-variation rate.}
For every $\lambda>0$, $\mathcal T_\lambda$ and $\mathcal T_0$
have unique fixed points $\pi_\lambda$ and $\pi$ among probability
densities. Moreover, there exist constants
$c_1=c_1(K,d,\|\Phi\|_\infty,\|W\|_\infty)>0$ and
$C_1=C_1(K,d,\|\Phi\|_\infty,\|W\|_\infty)>0$,
such that for every $\lambda\in(0,1]$,
\begin{equation}
\label{pi_lambda--pi--L1}
c_1\sqrt{\lambda}
\le
\|\pi_\lambda-\pi\|_{L^1(\mathbb R^d)}
\le
{
\begin{cases}
C_1(1-\kappa)^{-1}\sqrt\lambda, & \text{in case (a)},\\
C_1\sqrt\lambda, & \text{in case (b)}.
\end{cases}}
\end{equation}
In particular, $\pi$ is the unique stationary variational solution
of \eqref{eq:ref-pde}.

\item \emph{Sharp quadratic-Wasserstein rate.}
There exists
${C_2=C_2(K,d,\|\Phi\|_\infty,\|W\|_\infty)>0}$
such that
\[
  W_2(\pi_\lambda,\pi)
  \le
  {
  \begin{cases}
    C_2(1-\kappa)^{-1}\sqrt\lambda, & \text{in case (a)},\\
    C_2\sqrt\lambda, & \text{in case (b)},
  \end{cases}}
  \qquad 0<\lambda\le1.
\]
Moreover, there exist $c_2>0$ and $\lambda_*\in(0,1]$
{depending only on $K,d,\|\Phi\|_\infty,\|W\|_\infty$, and,
in case (a), $(1-\kappa)^{-1}$,} such that
\[
  W_2(\pi_\lambda,\pi)
  \ge c_2\sqrt{\lambda},
  \qquad 0<\lambda\le\lambda_*.
\]

\end{enumerate}
\end{theorem}

\medskip
For comparison, in the kernel-free case (i.e., $W\equiv0$), we obtain the following complementary result.
\begin{proposition}[Kernel-free Gibbs measures]\label{prop:kernel--free}
Let Assumption~\ref{ass:geometry} hold.
Set $W\equiv0$, $\Phi=f$, with $f\in C^2(\Rd)$ and
$\inf_{\Rd}f>-\infty$.  Let
\[
  Z_\lambda:=\int_{\Rd}e^{-f-U_\lambda}\dd x,\qquad
  Z:=\int_K e^{-f}\dd x,\qquad
  T_\lambda:=\int_{K^c}e^{-f-U_\lambda}\dd x,
\]
and
\[
\begin{aligned}
\mu_{\lambda,\infty}(\dd x)&=\rli(x)\dd x,
&\rli&:=Z_\lambda^{-1}e^{-f-U_\lambda},\\
\mu_\infty(\dd x)&=\ri(x)\dd x,
&\ri&:=Z^{-1}e^{-f}\one_K.
\end{aligned}
\]

Then there exist constants $c,C>0$ such that, for every
$\lambda\in(0,1]$,
\begin{equation}
\label{eq:linear-partition-sharp}
Z_\lambda=Z+T_\lambda,
\qquad
c\sqrt{\lambda}
\leq
T_\lambda
\leq
C\sqrt{\lambda}.
\end{equation}
Moreover,
\begin{equation}
\label{eq:linear-tv-exact}
\|\rli-\ri\|_{L^1(\mathbb R^d)}
=
\frac{2T_\lambda}{Z_\lambda}.
\end{equation}
As a result,
\[
c\sqrt{\lambda}
\leq
\|\rli-\ri\|_{L^1(\mathbb R^d)}
\leq
C\sqrt{\lambda},
\]
and the same two-sided rate holds for the total-variation distance.
Finally,
\begin{equation}
\label{eq:linear-Wasserstein-sharp}
c\sqrt{\lambda}
\leq
W_1(\mu_{\lambda,\infty},\mu_\infty)
\leq
W_2(\mu_{\lambda,\infty},\mu_\infty)
\leq
C\sqrt{\lambda}.
\end{equation}   
\end{proposition}

\medskip
The probabilistic counterpart of \eqref{eq:ref-pde} is the reflected
McKean--Vlasov SDE
\begin{equation}\label{eq:rmv-sde}
  \dd X_t=-b[\mu_t](X_t)\dd t+\sqrt2\dd B_t-n(X_t)\dd L_t,
  \qquad
  \mu_t=\mathrm{Law}(X_t),\qquad X_t\in K,
\end{equation}
Here $L$ is the continuous nondecreasing regulator with $L_0=0$ and
$\int_0^T\one_{\{X_t\notin\Gamma\}}\dd L_t=0$; its normalization is fixed
by \eqref{eq:rmv-sde}.  We refer to $L$ as the boundary local time.
For each $\lambda>0$, let
$X^\lambda$ denote the penalized McKean--Vlasov diffusion whose marginal
densities are $\rho^\lambda$, as in Proposition~\ref{prop:pen-wp}.

\begin{corollary}[Probabilistic soft-to-hard limit and averaged boundary reaction]
\label{cor:probabilistic}\label{prop:marginal}
Under the same assumptions as in Theorem \ref{thm:main}, the reflected SDE
\eqref{eq:rmv-sde} has a unique solution with
$\mathrm{Law}(X_0)=\rho_0$, and the following statements hold.
\begin{enumerate}[label=(\roman*),leftmargin=2.2em]
\item \emph{Identification of the limiting law.}
For every $t\in[0,T]$,
\[
  \mathrm{Law}(X_t)=\rho_t\,\dd x,
\]
where $\rho$ is the unique variational solution in
Theorem~\ref{thm:main}.

\item \emph{Convergence of time marginals.}
Regarding $\mathrm{Law}(X_t)$ as a probability measure on $\Rd$ supported
in $K$,
\[
  \sup_{t\in[0,T]}
  \dBL\bigl(\mathrm{Law}(X_t^\lambda),\mathrm{Law}(X_t)\bigr)
  \longrightarrow0.
\]
Moreover, the marginal densities satisfy
\[
  \rho^\lambda|_K\longrightarrow\rho
  \quad\text{strongly in }L^2((0,T)\times K),
  \qquad
  \rho^\lambda\longrightarrow\rho\one_K
  \quad\text{in }L^1((0,T)\times\Rd).
\]

\item \emph{Averaged penalty-force occupation measures.}
The vector measures
\[
  \mathsf Q^\lambda(\dd t,\dd x)
  :=\frac{x-\proj_Kx}{\lambda}\,
  \mathrm{Law}(X_t^\lambda)(\dd x)\,\dd t
\]
converge weakly-$*$ to
\[
  \Tr(\rho)\,n\,\Hm\!\restriction_\Gamma\,\dd t.
\]
\end{enumerate}
\end{corollary}

The corollary concerns time marginals and averaged force-occupation
measures.  {Path-space convergence, convergence of the
accumulated penalty process to $\int_0^\cdot n(X_s)\dd L_s$, and a
propagation-of-chaos estimate uniform in $\lambda$ are recorded in
Proposition~\ref{prop:pathwise}.}


\begin{remark}[Two boundary fluxes; terminology]\label{rem:flux-terminology}
The no-flux condition of \eqref{eq:ref-pde} concerns the limiting interior
flux $J=\nabla\rho+\rho\,b[\rho]$, whose normal trace vanishes in the
variational sense.  Quantitative estimates for the normal traces of
$J_K^\lambda=\nabla\rho^\lambda+\rho^\lambda b[\rho^\lambda]$ are treated in
the companion paper~\cite{SeoKimChoiStrong} under its additional
assumptions.  In contrast, the penalty component
$\rho^\lambda\nabla U_\lambda$ is supported outside $K$ and converges to the
$O(1)$ reaction measure in Theorem~\ref{thm:main}(iii).  Its density is the
boundary trace of $\rho$, not a normal flux.  Throughout the paper,
``boundary flux'' without qualification refers to the interior flux trace
$J_K^\lambda\cdot n$.
\end{remark}

\section{Well-posedness and uniqueness preliminaries}\label{sec:wp}

We first establish fixed-$\lambda$ well-posedness and the weighted energy
identity, then uniqueness of the reflected solution. These are the inputs
for the limiting argument in Section~\ref{sec:dynamical-proof}.

\begin{proposition}
\label{prop:pen-wp}
Suppose that Assumptions~\ref{ass:geometry}, \ref{ass:coefficients}, and
\ref{ass:data} hold.  Fix
$\lambda>0$.  There exists a unique penalized solution $\rho^\lambda$ in the
sense of Definition~\ref{def:pen-sol}.  It is the density of the time
marginals of the McKean--Vlasov SDE
\[
  \dd X_t=-\nabla U_\lambda(X_t)\dd t-b[\mathrm{Law}(X_t)](X_t)\dd t
  +\sqrt2\dd B_t,\qquad \mathrm{Law}(X_0)=\rho_0,
\]
and satisfies the energy identity
\begin{equation}\label{eq:energy-identity}
  \frac12\frac{\dd}{\dd t}\int_{\Rd}|\hh^\lambda_t|^2\dd\muh
  +\int_{\Rd}|\nabla\hh^\lambda_t|^2\dd\muh
  =-\int_{\Rd}\hh^\lambda_t\,b[\rho^\lambda_t]\cdot\nabla\hh^\lambda_t\dd\muh
  \qquad\text{for a.e. }t .
\end{equation}
The same conclusions hold if Assumption~\ref{ass:data} is replaced by
\(\rho_0^\lambda=h_0w_\lambda\), where \(h_0\ge0\),
\(h_0\in L^2(\muh)\), and \(\int_{\Rd}h_0\dd\muh=1\), with this
initial datum used in Definition~\ref{def:pen-sol} and in the SDE.
\end{proposition}

The proof (Appendix~\ref{app:wp}) is a fixed-point argument for the map
sending a curve of measures to the law of the SDE with frozen drift; the
contraction is established in $\sup_t\|\cdot\|_{\TV}$ through a Girsanov
relative-entropy estimate combined with Pinsker's inequality, and the energy
framework of Definition~\ref{def:pen-sol} is obtained from standard linear
Fokker--Planck theory applied with the frozen drift.  The constants in
Appendix~\ref{app:wp} degenerate as $\lambda\downarrow0$; this is
irrelevant, since all quantitative content of the dynamical limit is in the
$\lambda$-uniform estimates of Section~\ref{sec:dynamical-proof}.

\begin{proposition}
\label{prop:ref-unique}
Under Assumptions~\ref{ass:geometry}--\ref{ass:data} there is at most one
variational solution of \eqref{eq:ref-pde}.
\end{proposition}

The energy arguments below use the Lions--Magenes chain rule
for $u\in L^2(0,T;H^1(K))$ with
$\partial_tu\in L^2(0,T;H^1(K)')$; testing and time integration are
justified by Steklov averaging.

\begin{proof}
Let $\rho_1,\rho_2$ be variational solutions and $u:=\rho_1-\rho_2$.
Subtracting \eqref{eq:var-form} and testing with $\varphi=u(t)$,
\[
  \frac12\frac{\dd}{\dd t}\|u\|_{L^2(K)}^2
  +\|\nabla u\|_{L^2(K)}^2
  =-\int_K\Bigl(u\,b[\rho_{1,t}]
  +\rho_2\bigl(b[\rho_{1,t}]-b[\rho_{2,t}]\bigr)\Bigr)\cdot\nabla u\dd x .
\]
By \eqref{eq:drift-bounds},
\[
  \|b[\rho_{1,t}]-b[\rho_{2,t}]\|_\infty
  \le L_W\|u(t)\|_{L^1(K)}
  \le L_W|K|^{1/2}\|u(t)\|_{L^2(K)}.
\]
Thus, with $m:=\sup_t\|\rho_2(t)\|_{L^2(K)}<\infty$,
\[
  \frac{\dd}{\dd t}\|u\|_{L^2}^2+\|\nabla u\|_{L^2}^2
  \le C\bigl(B_0^2+L_W^2|K|\,m^2\bigr)\|u\|_{L^2}^2 .
\]
Gr\"onwall and $u(0)=0$ give $u\equiv0$.
\end{proof}

Existence of the variational solution is obtained \emph{as the limit} of the
penalized flow in Theorem~\ref{thm:main}; no separate existence proof is
needed.

\section{Proof of Theorem~\ref{thm:main}}
\label{sec:dynamical-proof}

This section proves Theorem~\ref{thm:main}.  The argument proceeds through
four modules: a drift-uniform Gaussian barrier, weighted energy and
time-regularity estimates, compactness of the density and full flux, and
identification of the limiting boundary reaction.

\subsection{Drift-uniform Gaussian domination}\label{sec:barrier}

This subsection establishes the replacement for the missing invariant-measure
domination of the linear theory.  The statement is a comparison principle
that holds \emph{simultaneously for every drift in the class $\mathcal B$}
of \eqref{eq:drift-class}; applied with the frozen drift
$b(t,x):=b[\rho^\lambda_t](x)$, it bounds the McKean--Vlasov solution without
ever confronting the nonlinearity.

We first fix a tilt field.  Let $d_s$ denote the signed distance to
$\Gamma$ ($d_s<0$ in the interior, $d_s=\dist$ on $K^c$), which is $C^2$
on the two-sided collar $\{|d_s|<\delta_0\}$ by
Assumption~\ref{ass:geometry}.  Choose $\theta\in C^\infty(\R)$ with
\[
  \theta(\sigma)=0\ \ (\sigma\le-\delta_0),\qquad
  0\le\theta'\le1,\qquad
  \theta'(\sigma)=1\ \ (\sigma\ge-\delta_0/2),\qquad
  \|\theta''\|_\infty\le C/\delta_0,
\]
and set $\zeta:=\theta\circ d_s$.  Then $\zeta\in C^{1,1}(\Rd)$, $\zeta\ge0$,
$\zeta$ is constant ($=0$) at depth $\ge\delta_0$ inside $K$, and
\begin{equation}\label{eq:zeta-props}
  \nabla\zeta=\nabla\dist=n\quad\text{on }K^c,\qquad
  |\nabla\zeta|\le1,\qquad
  |D^2\zeta|\le C_\zeta{,\quad|\Delta\zeta|\le C_\zeta}\quad\text{a.e.},
\end{equation}
with $C_\zeta=C_\zeta(K)$: on the collar this is the $C^2$ regularity of
$d_s$, and outside the collar $\theta'\equiv1$ so $D^2\zeta=D^2\dist$, which
is bounded by \eqref{eq:lap-dist-bound}.  The essential design point is that
\emph{the transition of the tilt takes place inside $K$, at unit scale}; a
transition at scale $\lambda$ or $\sqrt\lambda$ near $\Gamma$ would
produce second derivatives of size $\lambda^{-1}$ or $\lambda^{-1/2}$ and
destroy the supersolution property below.

\begin{theorem}[Tilted Gaussian barrier]\label{thm:barrier}
Let Assumptions~\ref{ass:geometry}--\ref{ass:coefficients} hold and set
\[
  \beta:=B_0+1,\qquad
  \gamma:=\beta(\beta+B_0)+\beta\,(C_K^*\vee C_\zeta)+dB_1 .
\]
Define
\[
  \bar R(t,x):=M\,e^{\gamma t}\,e^{\beta\zeta(x)}\,\wl(x)
  =M\,e^{\gamma t}\exp\Bigl(-\frac{\dist(x)^2}{2\lambda}+\beta\zeta(x)\Bigr).
\]
Then for every measurable family $(b(t,\cdot))_{t\in[0,T]}\subset\mathcal B$,
$\bar R$ is an a.e.\ supersolution of the linear penalized equation:
\begin{equation}\label{eq:supersol}
  \partial_t\bar R
  -\divg\bigl(\nabla\bar R+\bar R\,\nabla U_\lambda+\bar R\,b(t,\cdot)\bigr)
  \ \ge\ 0
  \qquad\text{a.e.\ in }(0,T)\times\Rd,
\end{equation}
for every $\lambda>0$.  Consequently, if $\rho^\lambda$ is a penalized
solution with $0\le\rho_0\le M\one_K$, then
\begin{equation}\label{eq:domination}
  0\ \le\ \rho^\lambda(t,x)\ \le\
  M\,e^{\gamma t}\,e^{\beta\zeta(x)}\,
  e^{-\dist(x)^2/2\lambda}
  \qquad\text{for a.e.\ }(t,x)\in[0,T]\times\Rd .
\end{equation}
More generally, \eqref{eq:domination} holds with \(M\) replaced by
\(M_0\) for any initial probability density satisfying
\(0\le\rho_0^\lambda\le M_0w_\lambda\).
\end{theorem}

\begin{proof}
\textbf{Step 1: the pointwise inequality \eqref{eq:supersol}.}
Write $w_\beta:=e^{\beta\zeta}\wl$, so $\bar R=Me^{\gamma t}w_\beta$ and
$\nabla\log w_\beta=-\nabla U_\lambda+\beta\nabla\zeta$.  The flux field in
\eqref{eq:supersol} is
\[
  F:=\nabla\bar R+\bar R\nabla U_\lambda+\bar Rb
  =\bar R\bigl(\nabla\log w_\beta+\nabla U_\lambda+b\bigr)
  =\bar R\,v,
  \qquad
  v:=\beta\nabla\zeta+b,
\]
a bounded field, $|v|\le\beta+B_0$.  Since
$\partial_t\bar R=\gamma\bar R$, inequality \eqref{eq:supersol} is equivalent
to the a.e.\ pointwise bound
\begin{equation}\label{eq:div-condition}
  \frac{\divg(w_\beta v)}{w_\beta}
  =\bigl(-\nabla U_\lambda+\beta\nabla\zeta\bigr)\cdot v
  +\beta\Delta\zeta+\divg b
  \ \le\ \gamma .
\end{equation}
We verify \eqref{eq:div-condition} in three regions, using
$\nabla U_\lambda=(\dist/\lambda)\,n$ on $K^c$ and $\nabla U_\lambda=0$ on
$K$, together with \eqref{eq:zeta-props} and $|\divg b|\le dB_1$.

\emph{(a) On $K$.}  Here $\nabla U_\lambda=0$ and $|\nabla\zeta|\le1$, so
the left side of \eqref{eq:div-condition} is
$\beta\nabla\zeta\cdot(\beta\nabla\zeta+b)+\beta\Delta\zeta+\divg b
\le\beta(\beta+B_0)+\beta C_\zeta+dB_1\le\gamma$.

\emph{(b) On $K^c$ with $\dist(x)\le\beta\lambda$.}  Here $\nabla\zeta=n$,
so the first term equals
$\bigl(\beta-\dist/\lambda\bigr)\bigl(\beta+n\cdot b\bigr)$ with
$0\le\beta-\dist/\lambda\le\beta$ and $0<\beta-B_0\le\beta+n\cdot b\le
\beta+B_0$; hence the left side is at most
$\beta(\beta+B_0)+\beta C_K^*+dB_1\le\gamma$, using
$\Delta\zeta=\Delta\dist\le C_K^*$ from \eqref{eq:lap-dist-bound}.

\emph{(c) On $K^c$ with $\dist(x)\ge\beta\lambda$.}  Now
$\beta-\dist/\lambda\le0$ while
$\beta+n\cdot b\ge\beta-B_0=1>0$, so the first term is non-positive; the left side
is at most $\beta C_K^*+dB_1\le\gamma$.

All quantities above exist a.e.\ (Appendix~\ref{app:geometry}): $\dist$ is
$C^2$ in the exterior collar and $C^{1,1}_{\rm loc}$ on $K^c$ with the
a.e.\ bounds \eqref{eq:lap-dist-bound}, $\zeta\in C^{1,1}$, and $w_\beta v$
is locally Lipschitz with no jump across $\Gamma$ (indeed
$\nabla\log w_\beta\to\beta n$ from both sides at $\Gamma$, since
$\nabla U_\lambda=\dist\,n/\lambda\to0$ there and $\theta'\equiv1$ near
$\Gamma$), so the distributional divergence carries no surface part and
\eqref{eq:div-condition} holds distributionally as well.

\textbf{Step 2: comparison.}  At $t=0$, $\rho_0\le M\one_K\le Mw_\beta$
because $w_\beta\ge1$ on $K$ ($\zeta\ge0$, $\wl=1$).  Freezing
$b(t,\cdot):=b[\rho^\lambda_t]\in\mathcal B$, both $\rho^\lambda$ and
$\bar R$ are, respectively, a weak solution and an a.e.\ supersolution of
the \emph{same linear} equation with bounded measurable drift, and
$\rho^\lambda$ lies in the energy class of Definition~\ref{def:pen-sol}.
The weak comparison principle of Appendix~\ref{app:comparison}, applied to
the ratio $\rho^\lambda/w_\beta$, yields
$\rho^\lambda\le Me^{\gamma t}w_\beta$, which is \eqref{eq:domination}.
For the weighted initial datum, use
\(\rho_0^\lambda\le M_0w_\lambda\le M_0w_\beta\), since \(\zeta\ge0\),
and repeat the same comparison with \(M_0\) in place of \(M\).
Nonnegativity is part of Proposition~\ref{prop:pen-wp}.
\end{proof}

\begin{corollary}
\label{cor:moments}
Under the hypotheses of Theorem~\ref{thm:barrier} there is
$C=C(K,d,B_0,B_1,M,T)$ such that for all $t\in[0,T]$,
$\lambda\in(0,\lambda_0]$ and $q\ge0$,
\begin{equation}\label{eq:exterior-moments}
  \int_{K^c}\dist(x)^q\,\rho^\lambda(t,x)\dd x\le C_q\,\lambda^{(q+1)/2},
  \qquad
  \int_{\{\dist\ge\delta\}}\rho^\lambda(t,x)\dd x\le C\,e^{-\delta^2/8\lambda}
  \ \ (\delta\le1),
\end{equation}
and $\|\rho^\lambda\|_{L^\infty([0,T]\times K)}\le CM$.  In particular the
exterior mass is $O(\sqrt\lambda)$ uniformly on $[0,T]$.
For initial data dominated by \(M_0w_\lambda\), the same conclusions hold
with \(M\) replaced by \(M_0\).
\end{corollary}

\begin{proof}
On $K^c$, $\zeta\le\dist+c_0$ with $c_0=c_0(K)$, so
\eqref{eq:domination} gives, for $\dist\ge4\beta\lambda$,
\[
\rho^\lambda_t\le Ce^{-\dist^2/2\lambda+\beta\dist}
\le Ce^{\beta^2\lambda}e^{-(\dist-\beta\lambda)^2/2\lambda}
\le C'e^{-\dist^2/4\lambda},
\]
and $\rho^\lambda_t\le C$ elsewhere.  The tilted Gaussian layer estimate of
Lemma~\ref{lem:layer} (coarea formula plus the Steiner polynomial bound for
the parallel sets of a convex body) yields \eqref{eq:exterior-moments}.  The
interior bound is \eqref{eq:domination} restricted to $K$, where
$\wl=1$ and $e^{\beta\zeta}\le e^{\beta\theta(0)}$.
\end{proof}

\begin{remark}[Uniformity over the drift class]
The pair $(\beta,\gamma)$ depends only on $(B_0,B_1,K,d)$.  This uniformity
is used repeatedly: for the McKean--Vlasov solution itself; for any
approximating or perturbed flow one may wish to insert; and for the
\emph{lifted} solutions employed in the strong theory of the companion
paper~\cite{SeoKimChoiStrong}, whose drifts are different members of the
same class.  It is the precise sense in which the barrier decouples the
nonlinearity.
\end{remark}

\subsection{Uniform weighted energy estimates}\label{sec:energy}

To obtain uniform energy estimates, we absorb the singular penalty into
the weight and treat the full drift as a bounded perturbation.
Throughout this subsection $\rho^\lambda$ is the penalized solution,
$\hh:=\hh^\lambda=\rho^\lambda/\wl$ (so $\hh=\rho^\lambda$ on $K$), and
$b^\lambda(t,x):=b[\rho^\lambda_t](x)\in\mathcal B$.  In terms of $\hh$ the
equation reads, weakly,
\begin{equation}\label{eq:ratio-eq}
  \wl\,\partial_t\hh=\divg\bigl(\wl\nabla\hh+\wl\,\hh\,b^\lambda\bigr),
\end{equation}
because $\nabla\rho^\lambda+\rho^\lambda\nabla U_\lambda=\wl\nabla\hh$.
We write
\[
  \|v\|_\lambda^2:=\int_{\Rd}v^2\dd\muh,\qquad
  E(t):=\int_{\Rd}|\nabla\hh_t|^2\dd\muh .
\]
Recall $\muh(\Rd)\le|K|+C\sqrt\lambda\le C_K$ and
$\|\hh_0\|_\lambda=\|\rho_0\|_{L^2(K)}=M_2$ for the hard initialization.

\begin{remark}[Relation to earlier strong-confinement estimates]
The exponential conjugation underlying \eqref{eq:ratio-eq} is closely
related to the weighted $L^2$ method of Alasio, Bruna, and Carrillo
\cite{AlasioBrunaCarrillo2020}.  In the present setting only the singular
Moreau--Yosida wall is absorbed into the weight, while the fixed external
field and the law-dependent interaction remain as a nonautonomous bounded
perturbation.  The estimates below supply the compactness part of the main
theorem; {the barrier makes the exterior estimates
quantitative, while Lemma~\ref{lem:factorization} identifies the reaction;
see Remark~\ref{rem:barrier-role}.}
\end{remark}

\begin{lemma}\label{lem:basic-energy}
For all $t\in[0,T]$,
\[
  \|\hh_t\|_\lambda^2+\int_0^tE(s)\dd s\ \le\ e^{B_0^2t}\,M_2^2 .
\]
\end{lemma}

\begin{proof}
The identity \eqref{eq:energy-identity} and Cauchy--Schwarz give
$\frac12\frac{\dd}{\dd t}\|\hh\|_\lambda^2+E
\le B_0\|\hh\|_\lambda E^{1/2}
\le\frac12E+\frac{B_0^2}2\|\hh\|_\lambda^2$. The desired result follows from Gr\"onwall.
\end{proof}

{
\begin{lemma}
\label{lem:drift-time-reg}
For a.e.\ $t\in(0,T)$,
\begin{equation}\label{eq:J-L1-bound}
  \int_{\Rd}|J^\lambda(t,x)|\dd x
  \le C\bigl(1+E(t)^{1/2}\bigr),
\end{equation}
where $C$ is independent of $\lambda\le\lambda_0$.
\end{lemma}

\begin{proof}
Since $J^\lambda=\wl\nabla\hh+\rho^\lambda b^\lambda$ and
$\|b^\lambda\|_\infty\le B_0$, Cauchy--Schwarz with respect to
$\muh=\wl\dd x$ gives
\[
  \int_{\Rd}|J^\lambda(t)|\dd x
  \le \muh(\Rd)^{1/2}
       \Bigl(\int_{\Rd}|\nabla\hh_t|^2\dd\muh\Bigr)^{1/2}+B_0.
\]
The mass $\muh(\Rd)$ is uniformly bounded by Lemma~\ref{lem:layer} with
$q=0$, and the gradient integral is $E(t)$.  This proves
\eqref{eq:J-L1-bound}.
\end{proof}
}

\begin{lemma}
\label{lem:interior}
With $C=C(T,K,d,B_0,B_1,\|W\|_{C^2},M,M_2)$:
\begin{align}
  &\sup_{t\le T}\|\rho^\lambda_t\|_{L^2(K)}
  +\|\nabla\rho^\lambda\|_{L^2((0,T)\times K)}
  +\|J^\lambda\|_{L^2((0,T)\times K)}\le C,
  \label{eq:interior-bounds}\\
  &\|\partial_t\rho^\lambda\|_{L^2(0,T;H^1_0(K)')}\le C .
  \label{eq:dual-bound}
\end{align}
\end{lemma}

\begin{proof}
On $K$, $\wl\equiv1$, so \eqref{eq:interior-bounds} is
Lemma~\ref{lem:basic-energy} together with
$|J^\lambda|\le|\nabla\hh|+B_0|\hh|$ on $K$ (the penalty drift vanishes on
$K$).  For \eqref{eq:dual-bound}, test \eqref{eq:pen-weak} with
$\varphi\in C_c^\infty(K)$: $\langle\partial_t\rho^\lambda,\varphi\rangle
=-\int_KJ^\lambda\cdot\nabla\varphi$, and use
\eqref{eq:interior-bounds}.
\end{proof}

\begin{lemma}\label{lem:exterior-flux}
$\displaystyle\int_0^T\!\!\int_{K^c}|J^\lambda|\dd x\dd t\le
C\,\lambda^{1/4}$.
\end{lemma}

\begin{proof}
$\int_0^T\!\int_{K^c}|J^\lambda|
\le\int_0^T\!\int_{K^c}\wl|\nabla\hh|
+B_0\int_0^T\!\int_{K^c}\rho^\lambda$.  The second term is
$O(\sqrt\lambda)$ by Corollary~\ref{cor:moments}.  For the first,
Cauchy--Schwarz with respect to $\muh$:
\[
  \int_0^T\!\!\int_{K^c}\wl|\nabla\hh|
  \le\Bigl(\int_0^T\!\!E\Bigr)^{1/2}
  \bigl(T\,\muh(K^c)\bigr)^{1/2}
  \le C\bigl(\sqrt\lambda\bigr)^{1/2}=C\lambda^{1/4},
\]
using $\muh(K^c)=\int_{K^c}e^{-\dist^2/2\lambda}\le C\sqrt\lambda$
(Lemma~\ref{lem:layer}).
\end{proof}

{
We close this subsection with the stability estimate used for uniqueness of
Proposition~\ref{prop:pen-wp}.

\begin{lemma}
\label{lem:pen-stability}
Let $\rho^{(1)},\rho^{(2)}$ be penalized solutions with data
$\rho^{(i)}_0=h^{(i)}_0\wl$, $h^{(i)}_0\in L^2(\muh)$, and
$\rho^{(i)}_0\in\mcP(\Rd)$.  Then, with $z:=h^{(1)}-h^{(2)}$,
\[
  \sup_{t\le T}\|z_t\|_\lambda^2
  +\int_0^T\!\!\int_{\Rd}|\nabla z|^2\dd\muh\dd t
  \le C\|z_0\|_\lambda^2,
\]
where $C=C(T,B_0,L_W,K,\sup_t\|h^{(2)}_t\|_\lambda)$ is uniform in
$\lambda$.
\end{lemma}

\begin{proof}
Subtracting the two ratio equations \eqref{eq:ratio-eq} and testing by $z$
gives, with $\Delta b=b[\rho^{(1)}_t]-b[\rho^{(2)}_t]$,
\[
  \frac12\frac{\dd}{\dd t}\|z\|_\lambda^2
  +\|\nabla z\|_{L^2(\muh)}^2
  \le \bigl(B_0\|z\|_\lambda
      +\|\Delta b\|_\infty\|h^{(2)}\|_\lambda\bigr)
      \|\nabla z\|_{L^2(\muh)}.
\]
By \eqref{eq:drift-bounds} and Cauchy--Schwarz,
\[
  \|\Delta b\|_\infty
  \le L_W\int_{\Rd}|z|\dd\muh
  \le L_W\muh(\Rd)^{1/2}\|z\|_\lambda
  \le C\|z\|_\lambda.
\]
Young's inequality therefore yields
$\frac{\dd}{\dd t}\|z\|_\lambda^2+
 \|\nabla z\|_{L^2(\muh)}^2\le C\|z\|_\lambda^2$;
integration and Gr\"onwall give the claim.
\end{proof}
}

\subsection{Density and full-flux convergence}\label{sec:weak}

We now use compactness and continuity of the interaction drift to identify
the reflected limit.

\begin{proposition}\label{prop:compactness}
The family $(\rho^\lambda|_K)_{\lambda\le\lambda_0}$ is relatively compact
in $L^2((0,T)\times K)$, and every limit point $\rho$ belongs to
$L^\infty(0,T;L^2(K))\cap L^2(0,T;H^1(K))$ with
$\partial_t\rho\in L^2(0,T;H^1_0(K)')$.
\end{proposition}

\begin{proof}
Estimates \eqref{eq:interior-bounds} and \eqref{eq:dual-bound} give uniform
bounds for $\rho^\lambda|_K$ in $L^2(0,T;H^1(K))$ and for
$\partial_t\rho^\lambda$ in $L^2(0,T;H^1_0(K)')$, respectively.  Since
$H^1(K)\hookrightarrow\hookrightarrow L^2(K)\hookrightarrow H^1_0(K)'$,
the Aubin--Lions--Simon theorem gives relative compactness in
$L^2(0,T;L^2(K))$.  The stated regularity of limit points follows from weak
lower semicontinuity and by passing the uniform bound \eqref{eq:dual-bound}
to the limit; the upgrade of the dual estimate to the full space
$L^2(0,T;H^1(K)')$ is \emph{not} claimed at this stage --- it is obtained
from the limiting variational equation in the proof of
Theorem~\ref{thm:main}.
\end{proof}

The next lemma supplies the time-uniformity announced in
Theorem~\ref{thm:main}.  Recall the bounded-Lipschitz metric
$\dBL(\mu,\nu)=\sup\{\int\varphi\,\dd(\mu-\nu):
\|\varphi\|_\infty+\mathrm{Lip}(\varphi)\le1\}$, which metrizes narrow
convergence on $\mcP(\Rd)$ along tight families.

\begin{lemma}
\label{lem:equicont}
There is $C$, independent of $\lambda\le\lambda_0$, such that
\[
  \dBL\bigl(\rho^\lambda_t,\rho^\lambda_s\bigr)
  \ \le\ C\,|t-s|^{1/2},
  \qquad 0\le s\le t\le T .
\]
Moreover the family $\{\rho^\lambda_t:\lambda\le\lambda_0,\,t\in[0,T]\}$ is
tight.
\end{lemma}

\begin{proof}
Let $\varphi\in C^1_b(\Rd)$ with
$\|\varphi\|_\infty+\|\nabla\varphi\|_\infty\le1$.
Choose spatial cutoffs \(\chi_R\), equal to one on \(B_R\), supported in
\(B_{2R}\), and satisfying \(|\nabla\chi_R|\le C/R\).
Test \eqref{eq:pen-weak} with
\(\eta(r)\chi_R(x)\varphi_\eps(x)\), where
\(\eta\in C_c^\infty(0,T)\) and \(\varphi_\eps\) is a mollification
of \(\varphi\).  Lemma~\ref{lem:drift-time-reg} and the basic energy
estimate give \(J^\lambda\in L^1((0,T)\times\Rd)\); in particular,
\[
 \left|\int_0^T\!\!\int_{\Rd}
 \eta J^\lambda\cdot\varphi_\eps\nabla\chi_R\dd x\dd r\right|
 \le\frac{C\|\eta\|_\infty}{R}
 \int_0^T\!\!\int_{\Rd}|J^\lambda|\dd x\dd r\longrightarrow0.
\]
Dominated convergence removes the remaining spatial cutoff and then the
mollification.  Approximating the indicator of \((s,t)\) by smooth time
cutoffs gives, using \(\rho^\lambda\in C([0,T];L^1(\Rd))\) also at the
endpoints,
\[
  \int_{\Rd}\varphi\,\dd\rho^\lambda_t
  -\int_{\Rd}\varphi\,\dd\rho^\lambda_s
  =-\int_s^t\!\!\int_{\Rd}J^\lambda\cdot\nabla\varphi\dd x\dd r,
\]
so that, by \eqref{eq:J-L1-bound} and Cauchy--Schwarz in time,
\[
\begin{aligned}
  \Bigl|\int\varphi\,\dd\rho^\lambda_t-\int\varphi\,\dd\rho^\lambda_s\Bigr|
  &\le\int_s^t\!\!\int_{\Rd}|J^\lambda|\dd x\dd r\\
  &\le C\int_s^t\bigl(1+E(r)^{1/2}\bigr)\dd r\\
  &\le C(t-s)+C(t-s)^{1/2}\Bigl(\int_0^TE\Bigr)^{1/2}.
\end{aligned}
\]
and $\int_0^TE\le C$ by Lemma~\ref{lem:basic-energy}.  A bounded Lipschitz
$\varphi$ with $\|\varphi\|_\infty+\mathrm{Lip}(\varphi)\le1$ is the uniform
limit of mollifications $\varphi_\eps$ with
$\|\varphi_\eps\|_\infty+\|\nabla\varphi_\eps\|_\infty\le1$, so the same
bound holds in the supremum defining $\dBL$.  Tightness follows from the
domination $\rho^\lambda_t\le\bar R_t=Me^{\gamma t}e^{\beta\zeta}\wl$ of
Theorem~\ref{thm:barrier}: since $U_\lambda\ge U_{\lambda_0}$ pointwise for
$\lambda\le\lambda_0$, we have $\wl\le e^{-\dist^2/2\lambda_0}$, so
\[
  \rho^\lambda_t\ \le\ Me^{\gamma T}\,e^{\beta\zeta}\,
  e^{-\dist^2/2\lambda_0}
  \qquad\text{for all }\lambda\le\lambda_0,\ t\le T,
\]
and the right-hand side is a single fixed integrable function on $\Rd$
(Lemma~\ref{lem:layer} applied with $\lambda=\lambda_0$, using
$\zeta\le\dist+c_0$ on $K^c$).  Hence
$\sup_{\lambda\le\lambda_0,\,t\le T}\rho^\lambda_t(\{|x|\ge R\})\to0$ as
$R\to\infty$.
\end{proof}

\begin{proposition}\label{prop:interaction}
Let $\rho^{\lambda_j}|_K\to\rho$ in $L^2((0,T)\times K)$ with
$\lambda_j\downarrow0$.  Then
\[
  \int_0^T\bigl\|b[\rho^{\lambda_j}_t]-b[\rho_t]\bigr\|_{\infty}^2\dd t
  \longrightarrow0,
\]
where $\rho_t$ is extended by zero off $K$.
\end{proposition}

\begin{proof}
By \eqref{eq:drift-bounds},
$\|b[\rho^{\lambda_j}_t]-b[\rho_t]\|_\infty
\le L_W\bigl(\|\rho^{\lambda_j}_t-\rho_t\|_{L^1(K)}
+\int_{K^c}\rho^{\lambda_j}_t\bigr)$.
The second summand is bounded by $C\sqrt{\lambda_j}$ uniformly in $t$
(Corollary~\ref{cor:moments}); the first tends to $0$ in
$L^2(0,T)$ since $\|\cdot\|_{L^1(K)}\le|K|^{1/2}\|\cdot\|_{L^2(K)}$ and
$\rho^{\lambda_j}\to\rho$ in $L^2((0,T)\times K)$.  The integrand is also
uniformly bounded by $2B_0$, so dominated convergence applies.
\end{proof}

\medskip
\begin{proof}[Proof of Theorem~\ref{thm:main}, parts (i)--(ii)]
\textbf{Step 1: extraction and support of the limit.}
Take any sequence $\lambda_j\downarrow0$; by
Proposition~\ref{prop:compactness}, after a further extraction
$\rho^{\lambda_j}|_K\to\rho$ in $L^2((0,T)\times K)$ and a.e., with
$\nabla\rho^{\lambda_j}\rightharpoonup\nabla\rho$ in $L^2((0,T)\times K)$.
By Corollary~\ref{cor:moments} the exterior mass is $O(\sqrt{\lambda_j})$
uniformly in $t$, so $\int_K\rho_t\dd x=1$ for a.e.\ $t$ and
$\rho^{\lambda_j}\to\rho\one_K$ in $L^1((0,T)\times\Rd)$.

\textbf{Step 2: passage in the equation.}
Fix $\psi\in C_c^\infty([0,T)\times\Rd)$ and pass to the limit in
\eqref{eq:pen-weak}.  The exterior part of the flux integral vanishes by
Lemma~\ref{lem:exterior-flux}.  On $K$,
$J^{\lambda_j}=\nabla\rho^{\lambda_j}+\rho^{\lambda_j}b[\rho^{\lambda_j}_t]
=J^{\lambda_j}_K$;
the gradient converges weakly, while
$\rho^{\lambda_j}b[\rho^{\lambda_j}_t]\to\rho\,b[\rho_t]$ strongly in
$L^1((0,T)\times K)$ (and weakly in $L^2$) by
Proposition~\ref{prop:interaction} and the strong $L^2$ convergence of the
densities.  Hence
\begin{equation}\label{eq:limit-weak-form}
  \int_K\rho_0\psi(0)\dd x
  +\int_0^T\!\!\int_K\rho\,\partial_t\psi\dd x\dd t
  =\int_0^T\!\!\int_K
  \bigl(\nabla\rho+\rho\,b[\rho_t]\bigr)\cdot\nabla\psi\dd x\dd t .
\end{equation}
Set $J:=\nabla\rho+\rho\,b[\rho_t]\in L^2((0,T)\times K)$.
Testing \eqref{eq:limit-weak-form} with
$\psi(t,x)=\eta(t)\varphi(x)$, $\eta\in C_c^\infty(0,T)$, and using the
density of restrictions of $C_c^\infty(\Rd)$ in $H^1(K)$ gives
\[
  \langle\partial_t\rho,\varphi\rangle
  =-\int_KJ\cdot\nabla\varphi\dd x,
  \qquad \varphi\in H^1(K).
\]
The bound $|\langle\partial_t\rho,\varphi\rangle|
\le\|J\|_{L^2(K)}\|\nabla\varphi\|_{L^2(K)}$ yields
$\partial_t\rho\in L^2(0,T;H^1(K)')$ and hence
$\rho\in C([0,T];L^2(K))$.
Allowing $\psi(0)\ne0$ in \eqref{eq:limit-weak-form} then gives
$\rho(0)=\rho_0$.  Thus \eqref{eq:var-form} holds, including the
variational no-flux condition.  Nonnegativity and mass follow from Step~1.

\textbf{Step 3: uniqueness, full convergence, and time-uniformity.}
By Proposition~\ref{prop:ref-unique} the limit is unique, so the full family
converges in the topologies of Steps~1--2 (every subsequence has a further
subsequence converging to the same $\rho$).  To obtain time-uniform
convergence, set \(e_\lambda(t):=\dBL(\rho_t^\lambda,\rho_t\one_K)\).
Step~2 gives \(\rho\in C([0,T];L^2(K))\), hence the limit curve is
\(\dBL\)-continuous.  Together with Lemma~\ref{lem:equicont}, this makes
\(\{e_\lambda\}\) equicontinuous on \([0,T]\).  Moreover,
\[
 \int_0^T e_\lambda(t)\dd t
 \le\|\rho^\lambda-\rho\one_K\|_{L^1((0,T)\times\Rd)}\longrightarrow0.
\]
Nonnegativity and equicontinuity on the compact interval then imply
\(\sup_{t\le T}e_\lambda(t)\to0\).
The flux convergence
$J^{\lambda}\one_K\rightharpoonup J$ was established in Step~2, the exterior
flux bound is Lemma~\ref{lem:exterior-flux}, and the exterior mass bound is
Corollary~\ref{cor:moments}.
\end{proof}

\subsection{Boundary reaction and exact cancellation}\label{sec:reaction}

{
The penalty force is singular as a function, but its induced scalar and
vector measures have uniformly bounded mass.  We retain the exact normal
factorization and then identify the vector limit directly from the already
established full-flux convergence.

\begin{lemma}[Normal factorization and localization]
\label{lem:factorization}
Set
\[
  q^\lambda:=\lambda^{-1}\dist(x)\rho^\lambda\dd x\dd t,
  \qquad
  Q^\lambda:=\lambda^{-1}\rho^\lambda
  (x-\proj_Kx)\dd x\dd t.
\]
Then $q^\lambda\ge0$, $|Q^\lambda|=q^\lambda$, and on $K^c$
\[
  \dd Q^\lambda=n(\proj_Kx)\,\dd q^\lambda.
\]
Moreover, uniformly in $\lambda\le\lambda_0$,
\begin{align}
  q^\lambda([s,t]\times\Rd)&\le C(t-s),
  &&0\le s\le t\le T,\label{eq:q-time-slab-short}\\
  q^\lambda([0,T]\times\{\dist\ge\delta\})
  &\le CT\frac{\sqrt\lambda}{\delta},
  &&\delta>0.\label{eq:q-localization-short}
\end{align}
\end{lemma}

\begin{proof}
For a closed convex set the metric projection is single valued and, for
$x\notin K$,
\[
  x-\proj_Kx=\dist(x)\,n(\proj_Kx).
\]
This gives the factorization and $|Q^\lambda|=q^\lambda$.  The first estimate
follows from Corollary~\ref{cor:moments} with moment order $1$:
$\int\dist\,\rho^\lambda_t\le C\lambda$ uniformly in $t$.  For the second,
use $\dist\le\delta^{-1}\dist^2$ on $\{\dist\ge\delta\}$ and the
order-$2$ moment bound
$\int\dist^2\rho^\lambda_t\le C\lambda^{3/2}$.
\end{proof}

\begin{proof}[Proof of Theorem~\ref{thm:main}, parts (iii)--(iv)]
Lemma~\ref{lem:factorization} gives a uniform total-variation bound and,
through \eqref{eq:q-localization-short}, tightness of both measure families.
In spatial distributions,
\[
  Q^\lambda
  =J^\lambda-\nabla\rho^\lambda
   -\rho^\lambda b[\rho^\lambda_t].
\]
Let $V\in C_c^\infty([0,T]\times\Rd;\Rd)$.  Theorem~\ref{thm:main}(i)--(ii),
Proposition~\ref{prop:interaction}, and the exterior mass estimate imply
\begin{align*}
  \lim_{\lambda\downarrow0}\int V\cdot\dd Q^\lambda
  &=\int_0^T\!\!\int_K V\cdot
       (\nabla\rho+\rho b[\rho_t])
    +\int_0^T\!\!\int_K\rho\,\divg V
    -\int_0^T\!\!\int_K\rho b[\rho_t]\cdot V\\
  &=\int_0^T\!\!\int_K
       \bigl(V\cdot\nabla\rho+\rho\,\divg V\bigr).
\end{align*}
For a.e.\ $t$, $\rho_t\in H^1(K)$, so
$\rho_tV(t,\cdot)\in W^{1,1}(K;\Rd)$ and Gauss--Green gives
\[
  \int_K\bigl(V\cdot\nabla\rho+\rho\,\divg V\bigr)
  =\int_\Gamma\Tr(\rho)\,V\cdot n\dd\Hm.
\]
The trace theorem makes the right-hand side integrable in time.  By density
of smooth test fields, the uniform total-variation bound, and tightness,
\begin{equation}\label{eq:Q-limit-short}
  Q^\lambda\rightharpoonup^*
  \Tr(\rho)n\,\Hm\!\restriction_\Gamma\dd t.
\end{equation}
This vector test identifies every component of the limit and, in particular,
rules out a tangential defect.

It remains to recover the scalar measures.  Let $d_s$ be the signed distance
on a two-sided $C^2$ collar $\{|d_s|<\delta_0\}$ and choose
$\chi\in C_c^1(\{|d_s|<\delta_0\})$ with
$\chi=1$ on $\{|d_s|\le\delta_0/2\}$.  For
$\phi\in C_c^1([0,T]\times\Rd)$ set
$V=\phi\chi\nabla d_s$.  On the exterior part of the collar,
$\nabla d_s=n(\proj_Kx)$ and $|\nabla d_s|=1$, hence the factorization gives
\[
  \int V\cdot\dd Q^\lambda=\int\phi\chi\,\dd q^\lambda.
\]
Since $q^\lambda$ vanishes on $K$, \eqref{eq:q-localization-short} yields
\[
  \left|\int\phi\,\dd q^\lambda-
  \int V\cdot\dd Q^\lambda\right|
  \le \|\phi\|_\infty
  q^\lambda\bigl([0,T]\times\{\dist\ge\delta_0/2\}\bigr)
  \longrightarrow0.
\]
Using \eqref{eq:Q-limit-short} and $\chi=1$, $\nabla d_s=n$ on $\Gamma$,
we obtain
\[
  \lim_{\lambda\downarrow0}\int\phi\,\dd q^\lambda
  =\int_0^T\!\!\int_\Gamma\Tr(\rho)\,\phi\dd\Hm\dd t.
\]
Uniform boundedness extends this identity from $C_c^1$ to continuous compactly
supported tests, proving
$q^\lambda\rightharpoonup^*\Tr(\rho)\Hm|_\Gamma\dd t$.

{Finally, $\rho^\lambda\to\rho\one_K$ in $L^1((0,T)\times\Rd)$
by (i), hence $\nabla\rho^\lambda\to\nabla(\rho\one_K)$ in $\mathcal D'$, and
for a.e.\ $t$ the zero-extension (Gauss--Green) formula for
$\rho_t\in H^1(K)$ gives
\[
  \nabla(\rho_t\one_K)
  =\one_K\nabla\rho_t-\Tr(\rho_t)n\Hm|_\Gamma .
\]
Moreover $\rho^\lambda b[\rho^\lambda_t]\to\one_K\rho\,b[\rho_t]$ in
$L^1((0,T)\times\Rd)$ by (i), \eqref{eq:drift-bounds}, and
Proposition~\ref{prop:interaction}.  Together with
\eqref{eq:Q-limit-short} and (ii), this proves (iv).  We stress that the
relation $\one_KJ=\nabla(\rho\one_K)+\Tr(\rho)n\Hm|_\Gamma\dd t
+\one_K\rho\,b[\rho_t]$ among the three limits is itself the Gauss--Green
formula; the content of (iv) is the separate convergence of the three
components, two of which have nontrivial singular parts on $\Gamma$.}
\end{proof}

\begin{remark}
The identification is distributional: $\Tr(\rho)$ is the boundary density
of the reaction measure, whereas the limiting interior flux has zero normal
trace.  The vector test above identifies the reaction without losing
possible tangential components; the exact normal factorization is then used
to recover the scalar measure $q^\lambda$.
{The identification itself uses full-flux convergence
and Gauss--Green; the Moreau--Yosida structure supplies the quantitative
localization through the Gaussian layer and normal factorization.}
\end{remark}

\begin{remark}[Role of the barrier]\label{rem:barrier-role}
{The pointwise barrier is needed only for the
$O(\sqrt\lambda)$ exterior and moment bounds and the interior $L^\infty$
estimate.  Without it, Lemmas~\ref{lem:basic-energy} and \ref{lem:layer} give
\[
  \int_{K^c}\rho_t^\lambda\dd x
  +\int_{K^c}\frac{\dist^2}{\lambda}\rho_t^\lambda\dd x
  \le C\lambda^{1/4}.
\]
Testing \eqref{eq:pen-weak} with $\zeta$ gives, uniformly in $\lambda$,
\[
  q^\lambda([0,T]\times\Rd)
  \le \int_K\rho_0\zeta+T(C_\zeta+B_0).
\]
These estimates replace Corollary~\ref{cor:moments} in the compactness and
identification steps.  Hence parts (ii)--(iv), and the convergence assertions
in (i), remain valid for $\rho_0\in L^2(K)$; only the exterior-mass rate
becomes $O(\lambda^{1/4})$.}
\end{remark}
}

\section{Proof of Theorem~\ref{thm:stationary-main}}
\label{sec:stationary-proof}

In the interacting case stationary states solve the Gibbs fixed-point problem
\eqref{eq:fixed-point}.

For uniqueness, Theorem~\ref{thm:stationary-main} assumes $\kappa<1$ or
positive semidefiniteness of $W$; the latter includes the MMD kernel at
every amplitude. Neither condition is required for the finite-horizon
dynamical theory of Section~\ref{sec:dynamical-proof}.

\subsection{Proof of Theorem \ref{thm:stationary-main} (i)}
\label{sec:static}
We first establish the Gibbs stability estimates needed to control the
stationary fixed points.
The coefficient \eqref{eq:dob-def} is unchanged by adding terms of the form
$\varphi(x)+\psi(y)$ to $W$. In contrast, $\|W\|_\infty$ generally changes
under such additions. Moreover,
\[
  \dob(W)\le\osc(W)\le2\|W\|_\infty .
\]

\begin{lemma}\label{lem:gibbs-lip}
Let $A,B$ be measurable functions on a measure space $(\Omega,\mu)$ with
\[
  0<\int e^{-A}\dd\mu<\infty,\qquad
  0<\int e^{-B}\dd\mu<\infty,\qquad
  A-B\in L^\infty(\mu),
\]
and set
$g_A:=e^{-A}/\int e^{-A}\dd\mu$, $g_B$ likewise.  Then
\[
  \|g_A-g_B\|_{L^1(\mu)}\le\tfrac12\osc(A-B),
\]
where the oscillation is understood in the essential sense.
The individual potentials may be unbounded, so the common penalty
$U_\lambda$ can be included in both exponents.
\end{lemma}

\begin{proof}
Put $F:=A-B$ and
$g_s:=e^{-B-sF}/\int e^{-B-sF}\dd\mu$, $0\le s\le1$.
Since $F$ is bounded, differentiation under the integral is valid and
\[
  \partial_s g_s=-\Bigl(F-\int Fg_s\dd\mu\Bigr)g_s.
\]
Write $\mathbb E_s$ for expectation under $g_s\dd\mu$ and let $m$ be the
midpoint of the essential range of $F$.  Then
\[
  \|\partial_sg_s\|_{L^1(\mu)}
  =\mathbb E_s|F-\mathbb E_sF|
  \le\bigl(\mathbb E_s|F-\mathbb E_sF|^2\bigr)^{1/2}
  \le\bigl(\mathbb E_s|F-m|^2\bigr)^{1/2}
  \le\tfrac12\osc(F).
\]
Integration over $s\in[0,1]$ proves the claim.
\end{proof}

\begin{lemma}\label{lem:gibbs-contraction}
For probability densities $\rho_1,\rho_2$ on $\Rd$ and every
$\lambda\in[0,\infty)$ (with $\lambda=0$ denoting the hard-constrained map
$\mcT_0$),
\[
  \|\mcT_\lambda(\rho_1)-\mcT_\lambda(\rho_2)\|_{L^1}
  \ \le\ \kappa\,\|\rho_1-\rho_2\|_{L^1},
  \qquad
  \kappa:=\tfrac12\dob(W).
\]
\end{lemma}

\begin{proof}
For $\lambda>0$, apply Lemma~\ref{lem:gibbs-lip} with
$A=U_\lambda+\Phi+W*\rho_1$ and $B=U_\lambda+\Phi+W*\rho_2$ on $\Rd$.
For $\lambda=0$, use $A=\Phi+W*\rho_1$ and $B=\Phi+W*\rho_2$ on $K$.
In either case $A-B=W*z$, where $z:=\rho_1-\rho_2$ and $\int z\dd y=0$.  For any $x,x'$ and any constant
$\varkappa$ (using $\int z=0$),
\[
  (W*z)(x)-(W*z)(x')
  =\int\bigl(W(x,y)-W(x',y)-\varkappa\bigr)z(y)\dd y,
\]
and choosing $\varkappa$ as the midrange of
$y\mapsto W(x,y)-W(x',y)$ bounds the integrand kernel by $\dob(W)$ in
absolute value; hence $\osc(W*z)\le\dob(W)\|z\|_{L^1}$.  By
Lemma~\ref{lem:gibbs-lip},
\[
  \|\mcT_\lambda(\rho_1)-\mcT_\lambda(\rho_2)\|_{L^1}
  \le\tfrac12\osc(W*z)
  \le\tfrac12\dob(W)\|z\|_{L^1}
  =\kappa\|z\|_{L^1}. \qedhere
\]
\end{proof}

\medskip
\begin{proof}[Proof of Theorem~\ref{thm:stationary-main}(i)]
\emph{(Fixed points and upper bound{; case (a)})}
By Lemma~\ref{lem:gibbs-contraction}, $\mcT_\lambda$ ($\lambda\in[0,\infty)$)
is a $\kappa$-contraction on the complete metric space of probability
densities with the $L^1$ distance; Banach's theorem gives the unique fixed
points $\pi_\lambda$, $\pi$.  Inserting $\pi$,
\[
  \|\pi_\lambda-\pi\|_{L^1}
  \le\|\mcT_\lambda(\pi_\lambda)-\mcT_\lambda(\pi)\|_{L^1}
  +\|\mcT_\lambda(\pi)-\mcT_0(\pi)\|_{L^1}
  \le\kappa\|\pi_\lambda-\pi\|_{L^1}
  +\|\mcT_\lambda(\pi)-\mcT_0(\pi)\|_{L^1}.
\]
The second term compares Gibbs measures for the \emph{fixed} bounded
continuous potential $f_\pi:=\Phi+W*\pi$.  Set
\[
  Z_\pi:=\int_K e^{-f_\pi}\dd x,\qquad
  T_{\pi,\lambda}:=\int_{K^c}e^{-f_\pi-U_\lambda}\dd x.
\]
On $K$, $\mcT_\lambda(\pi)=Z_\pi\pi/(Z_\pi+T_{\pi,\lambda})$,
whereas $\mcT_0(\pi)=\pi$ vanishes on $K^c$.  Thus the interior deficit
and exterior mass coincide, giving
\[
  \|\mcT_\lambda(\pi)-\mcT_0(\pi)\|_{L^1}
  =\frac{2T_{\pi,\lambda}}{Z_\pi+T_{\pi,\lambda}}
  \le C\sqrt\lambda
\]
by Lemma~\ref{lem:layer} and $Z_\pi\ge |K|e^{-\|f_\pi\|_\infty}$.
Rearranging proves
$\|\pi_\lambda-\pi\|_{L^1}\le C_1(1-\kappa)^{-1}\sqrt\lambda$.

\medskip
\noindent
{\emph{(Fixed points and upper bound; case (b))}
\emph{Existence.} Put $M:=\|\Phi\|_\infty+\|W\|_\infty$.  For each fixed $\lambda>0$,
the uniform bounds on $W*\rho$ and its spatial Lipschitz constant give
local compactness of the Gibbs densities, while
$\mcT_\lambda(\rho)\le e^{2M}|K|^{-1}e^{-U_\lambda}$ controls the tails.
Thus $\mcT_\lambda$ has relatively compact range in $L^1$; for
$\lambda=0$, the same argument applies on $K$.
Continuity follows from Lemma~\ref{lem:gibbs-contraction}, so Schauder
gives fixed points for $\lambda\ge0$.

\emph{Uniqueness.} If $\rho_1,\rho_2$ are two fixed points, subtract their
logarithmic Gibbs formulas and integrate against $\rho_1-\rho_2$ on
their common support ($\Rd$ for $\lambda>0$, and $K$ for $\lambda=0$).
The normalization constants cancel because $\int(\rho_1-\rho_2)=0$, so
\[
  \int(\log\rho_1-\log\rho_2)(\rho_1-\rho_2)\dd x
  =-\iint W(x,y)(\rho_1-\rho_2)(x)(\rho_1-\rho_2)(y)\dd x\dd y\le0.
\]
The left-hand side is nonnegative, while positive semidefiniteness gives the
last inequality; hence $\rho_1=\rho_2$.

\emph{Perturbation estimate.} Set
$\beta_\lambda:=\int_{K^c}\pi_\lambda$ and
$\tilde\pi_\lambda:=\pi_\lambda\one_K/(1-\beta_\lambda)$.
This normalization puts $\tilde\pi_\lambda$ and $\pi$ on the same support
and makes their difference have zero integral, so the preceding
monotonicity argument can be applied.  The Gibbs
formula and Lemma~\ref{lem:layer} give
$\beta_\lambda\le C\sqrt\lambda$, $1-\beta_\lambda\ge c>0$, and
$\tilde\pi_\lambda,\pi\le C_0$ on $K$.
Writing \(\mathcal Z_K(\sigma):=\int_K e^{-\Phi-W*\sigma}\dd x\),
the two Gibbs formulas on \(K\), where \(U_\lambda=0\), read
\[
\begin{aligned}
 \log\tilde\pi_\lambda&=-\Phi-W*\pi_\lambda
   -\log\mathcal Z_K(\pi_\lambda),\\
 \log\pi&=-\Phi-W*\pi-\log\mathcal Z_K(\pi).
\end{aligned}
\]
Subtracting gives
\[
  (\log\tilde\pi_\lambda-\log\pi)+W*(\tilde\pi_\lambda-\pi)
  =r_\lambda+c_\lambda,
  \qquad
  \|r_\lambda\|_\infty\le2\|W\|_\infty\beta_\lambda.
\]
Here $r_\lambda:=W*(\tilde\pi_\lambda-\pi_\lambda)$ and
\(c_\lambda=\log\mathcal Z_K(\pi)-\log\mathcal Z_K(\pi_\lambda)\)
is constant in $x$; the bound follows from
$\|\tilde\pi_\lambda-\pi_\lambda\|_{L^1(\Rd)}=2\beta_\lambda$.
Multiplying by $\tilde\pi_\lambda-\pi$ and integrating, the constant drops
out and positive semidefiniteness, together with
$(\log a-\log b)(a-b)\ge(a-b)^2/\max(a,b)$, yields
\[
  C_0^{-1}\|\tilde\pi_\lambda-\pi\|_{L^2(K)}^2
  \le2\|W\|_\infty\beta_\lambda|K|^{1/2}
       \|\tilde\pi_\lambda-\pi\|_{L^2(K)}.
\]
Therefore
$\|\pi_\lambda-\pi\|_{L^1}
\le2\beta_\lambda+|K|^{1/2}\|\tilde\pi_\lambda-\pi\|_{L^2(K)}
\le C_1\sqrt\lambda$.}

\medskip
\noindent
\emph{(Lower bound)}
Extend $\pi$ by zero on $K^c$ and set
$\beta_\lambda:=\int_{K^c}\pi_\lambda\dd x$ and
$M:=\|\Phi\|_\infty+\|W\|_\infty$.
For $0<\lambda\le1$, Lemma~\ref{lem:layer} bounds the Gibbs normalizer by
\[
  Z_\lambda:=\int_{\Rd}e^{-U_\lambda-\Phi-W*\pi_\lambda}\dd x
  \le e^M(|K|+C_K\sqrt\lambda)
  \le e^M(|K|+C_K)=:C_Z.
\]
On $E_\lambda:=\{x\in K^c:0<\dist(x,K)<\sqrt\lambda\}$,
$U_\lambda\le1/2$, so $\pi_\lambda\ge C_Z^{-1}e^{-1/2-M}=:c_0>0$.
The Steiner formula for the convex outer parallel sets
$K_r:=\{x:\dist(x,K)<r\}$ gives
\begin{equation}\label{eq:exterior-collar-lower}
  |E_\lambda|=|K_{\sqrt\lambda}|-|K|
  \ge\Hm(\Gamma)\sqrt\lambda.
\end{equation}
Consequently,
\begin{equation}\label{eq:beta-lower}
  \beta_\lambda\ge c_0|E_\lambda|\ge b_0\sqrt\lambda,
  \qquad b_0:=c_0\Hm(\Gamma)>0.
\end{equation}
Since $\int_K(\pi_\lambda-\pi)=-\beta_\lambda$,
\[
  \|\pi_\lambda-\pi\|_{L^1(\Rd)}
  \ge\left|\int_K(\pi_\lambda-\pi)\dd x\right|+\beta_\lambda
  =2\beta_\lambda\ge2b_0\sqrt\lambda.
\]
This proves the lower bound with $c_1=2b_0$, with the stated dependence
on $K,d,\|\Phi\|_\infty,\|W\|_\infty$.

\medskip
\noindent
\emph{($\pi$ is stationary)}  Since $b[\pi]=\nabla(\Phi+W*\pi)$ and
$\log\pi+\Phi+W*\pi$ is constant on $K$ by the fixed-point property,
$\nabla\pi+\pi\,b[\pi]=\pi\nabla\bigl(\log\pi+\Phi+W*\pi\bigr)=0$ in $K$,
so \eqref{eq:var-form} holds with $\partial_t\pi=0$: $\pi$ is a stationary
variational solution.

\medskip
\noindent
\emph{(Uniqueness among stationary variational solutions)}
Let $\bar\rho$ be any stationary variational solution, i.e.\
$\bar\rho\in H^1(K)$, $\bar\rho\ge0$, $\int_K\bar\rho=1$, and
\begin{equation}\label{eq:stat-weak}
  \int_K\bigl(\nabla\bar\rho+\bar\rho\,\nabla f\bigr)\cdot\nabla\varphi
  \dd x=0
  \qquad\text{for all }\varphi\in H^1(K),
  \qquad f:=\Phi+W*\bar\rho\in C^2(\overline K).
\end{equation}
Equation \eqref{eq:stat-weak} is the stationary \emph{linear} reflected
Fokker--Planck equation with the frozen potential $f$, whose Gibbs density
$g:=e^{-f}/\int_Ke^{-f}$ satisfies $c\le g\le C$ on $\overline K$ and
$\nabla g+g\nabla f=0$; in particular $g$ also satisfies
\eqref{eq:stat-weak}.  Write $\bar\rho=gv$ with
$v=\bar\rho/g\in H^1(K)$ (allowed since $g\in C^1(\overline K)$ is bounded
away from zero).  Then
$\nabla\bar\rho+\bar\rho\nabla f=g\nabla v+v(\nabla g+g\nabla f)=g\nabla v$,
and \eqref{eq:stat-weak} becomes
$\int_Kg\,\nabla v\cdot\nabla\varphi=0$ for all $\varphi\in H^1(K)$.
Choosing $\varphi=v$ gives $\int_Kg|\nabla v|^2=0$, so $v$ is constant and,
by normalization, $\bar\rho=g$.  Thus every stationary variational solution
is the Gibbs measure of its own frozen potential, i.e.\ a fixed point of
$\mcT_0$; by uniqueness of the fixed point, $\bar\rho=\pi$.
\end{proof}

\subsection{Proof of Theorem \ref{thm:stationary-main} (ii)}
The upper bound has two steps: project the exterior layer onto $\Gamma$,
then redistribute mass within $K$.  The squared projection cost is
$O(\lambda^{3/2})$, giving a $W_2$ cost of $O(\lambda^{3/4})$.
The larger $O(\sqrt\lambda)$ scale comes from the redistribution step,
which we control by a negative Sobolev norm.

For the proof, set
\begin{equation}\label{eq:stationary-stability-constant}
  C_{\mathrm{st}}:=
  \begin{cases}
    C_1(1-\kappa)^{-1}, & \text{in case (a)},\\
    C_1, & \text{in case (b)},
  \end{cases}
\end{equation}
where $C_1$ is the constant from part~(i).
 
Throughout this subsection $\Gamma:=\partial K$, $\Omega:=\operatorname{int}K$,
$d_K(x):=\dist(x,K)$, and
\[
  \mu_\lambda:=\pi_\lambda\dd x,\qquad \mu:=\pi\dd x .
\]
Since $K$ is compact and convex with $C^2$ boundary, $\Omega$ is a bounded
convex $C^2$ domain, $|\Gamma|=0$, and $\proj_K$ maps $\Rd$ onto $K$ with
$\proj_K(K^c)\subseteq\Gamma$.  We write $\mcP(K)$ for the Borel probability
measures on $K$ (which may charge $\Gamma$), $\mathcal M_+(\Gamma)$ for the
finite nonnegative Borel measures on $\Gamma$, and, as elsewhere,
$H^1(K):=H^1(\Omega)$, $L^p(K):=L^p(\Omega)$.  We use repeatedly the
Poincar\'e inequality
\begin{equation}\label{eq:W2-poincare}
  \|\varphi\|_{L^2(K)}\le C_P\|\nabla\varphi\|_{L^2(K)}
  \qquad\Bigl(\varphi\in H^1(K),\ \textstyle\int_K\varphi\dd x=0\Bigr)
\end{equation}
and the trace inequality $\|\Tr\varphi\|_{L^1(\Gamma)}\le C_T\|\varphi\|_{H^1(K)}$,
both valid on the bounded convex Lipschitz domain $\Omega$.  Finally, $C$
denotes a constant depending only on $K$, $d$, $\|\Phi\|_\infty$ and
$\|W\|_\infty$, which may change from line to line.
 
The upper bound rests on two lemmas.  The first is a Benamou--Brenier
estimate in the spirit of Peyre \cite{Peyre2018}: the absolutely
continuous parts of both measures have densities bounded below, while one
measure may additionally carry a singular part on $\Gamma$.  This is exactly
what allows the exterior layer to be pushed \emph{onto} $\Gamma$ rather than
reflected into $K$.
 
\begin{lemma}
\label{lem:Hminus1-transport}
Let $\nu=f\dd x\in\mcP(K)$ with $f\in L^2(K)$ and
$f\ge c_*>0$ a.e.\ on $K$, and let
\(
  m=h\dd x+\sigma\in\mcP(K)\)
  where
\[
  h\in L^2(K),\quad h\ge c_*>0\ \text{a.e.\ on }K,
  \qquad
  \sigma\in\mathcal M_+(\Gamma)\ \text{ with }\
  \sigma\le\Lambda\,\Hm\!\restriction_\Gamma
\]
for some $\Lambda<\infty$.  For $\varphi\in H^1(K)$ put
\[
  \langle m-\nu,\varphi\rangle
  :=\int_K(h-f)\varphi\dd x+\int_\Gamma\Tr\varphi\dd\sigma ,
\]
which is well defined and finite, and set
\[
  \|m-\nu\|_{\dot H^{-1}(K)}
  :=\sup\Bigl\{\,\bigl|\langle m-\nu,\varphi\rangle\bigr|:\
  \varphi\in H^1(K),\ \int_K\varphi\dd x=0,\
  \|\nabla\varphi\|_{L^2(K)}\le1\Bigr\}
  \ \ (<\infty).
\]
Then
\[
  W_2(m,\nu)\ \le\ c_*^{-1/2}\,\|m-\nu\|_{\dot H^{-1}(K)} .
\]
\end{lemma}
 
{
\begin{proof}
For mean-zero $\varphi\in H^1(K)$, Poincar\'e and the trace inequality give
\[
  |\langle m-\nu,\varphi\rangle|
  \le C\bigl(\|h-f\|_{L^2(K)}+\Lambda\bigr)
        \|\nabla\varphi\|_{L^2(K)},
\]
so $N:=\|m-\nu\|_{\dot H^{-1}(K)}<\infty$.  On the mean-zero Hilbert space
$H^1(K)/\mathbb R$, the Riesz theorem therefore provides $u$ such that
\begin{equation}\label{eq:neumann-weak}
  \int_K\nabla u\cdot\nabla\varphi\dd x
  =-\langle m-\nu,\varphi\rangle
  \qquad(\varphi\in H^1(K)),
\end{equation}
with $\|\nabla u\|_{L^2(K)}=N$; the identity extends to constants because
$m$ and $\nu$ have the same mass.

Set $\rho_t=(1-t)m+t\nu$ and
$a_t=(1-t)h+tf\ge c_*$.  {Define
$v_t=\nabla u/a_t$ in $K$ and $v_t=0$ on $\Gamma$.  After extension by
zero outside $K$, the vector measure
$j_t:=v_t\rho_t=\nabla u\,\mathcal L^d|_K$ satisfies $j_t\ll\rho_t$
with Radon--Nikodym derivative $v_t$.}  For every
$\varphi\in C_c^\infty(\Rd)$, \eqref{eq:neumann-weak} gives
\[
  \frac{\dd}{\dd t}\int_K\varphi\dd\rho_t
  =\langle\nu-m,\varphi\rangle
  =\int_K\nabla\varphi\cdot v_t\dd\rho_t.
\]
{Thus
$\partial_t\rho_t+\nabla\!\cdot j_t=0$ in $\mathcal D'(\Rd)$; the weak
Neumann identity already accounts for the singular boundary part of
$\rho_t$.}  Moreover,
\[
  \int_K|v_t|^2\dd\rho_t
  =\int_K\frac{|\nabla u|^2}{a_t}\dd x
  \le c_*^{-1}N^2.
\]
Because $K$ is compact, $t\mapsto\rho_t$ is narrowly continuous in
$\mathcal P_2(\Rd)$ and the action above is finite.  The
Benamou--Brenier formula (equivalently, \cite[Thm.~8.3.1]{AGS}) therefore
yields
\[
  W_2^2(m,\nu)
  \le\int_0^1\!\int_K|v_t|^2\dd\rho_t\dd t
  \le c_*^{-1}N^2.
\]
\end{proof}
}
 
The second lemma states that projecting the exterior layer onto $\Gamma$
produces a surface measure of density $O(\sqrt\lambda)$; convexity makes the
normal coordinates global, so no collar radius is involved.
 
\begin{lemma}\label{lem:projected-layer}
Let $K$ satisfy Assumption~\ref{ass:geometry} and let $g:\Rd\to[0,\infty)$ be
measurable with $g(x)\le C_G\,e^{-d_K(x)^2/(2\lambda)}$ for all $x$.  Then
$(\proj_K)_\#\bigl(g\one_{K^c}\dd x\bigr)\in\mathcal M_+(\Gamma)$ and, for
$0<\lambda\le1$,
\[
  (\proj_K)_\#\bigl(g\one_{K^c}\dd x\bigr)
  \ \le\ C\,C_G\sqrt\lambda\;\Hm\!\restriction_\Gamma
  \qquad\text{as measures},
\]
with $C=C(K,d)$.
\end{lemma}
 
\begin{proof}
That the pushforward is carried by $\Gamma$ is immediate from
$\proj_K(K^c)\subseteq\Gamma$.  Since $K$ is compact and convex with $C^2$
boundary, the outward normal map
\[
  \Psi:\Gamma\times(0,\infty)\to K^c,\qquad \Psi(y,s):=y+s\,n(y),
\]
is a bijection with $\proj_K\Psi(y,s)=y$ and $d_K(\Psi(y,s))=s$: indeed every
$x\in K^c$ satisfies $x=\proj_Kx+d_K(x)\,n(\proj_Kx)$, and conversely
$n(y)$ lies in the normal cone of $K$ at $y$, so $\proj_K\Psi(y,s)=y$.  As
$n\in C^1(\Gamma;\R^d)$, the tube formula gives
\[
  \dd x=\prod_{i=1}^{d-1}\bigl(1+s\kappa_i(y)\bigr)\dd s\dd\Hm(y),
  \qquad
  0\le\kappa_i(y)\le\kappa_*:=\max_\Gamma\max_i\kappa_i<\infty,
\]
the principal curvatures being nonnegative by convexity.  Hence, for every
Borel set $A\subseteq\Gamma$,
\begin{align}
    (\proj_K)_\#\bigl(g\one_{K^c}\dd x\bigr)(A)
  &=\int_A\!\!\int_0^\infty g\bigl(\Psi(y,s)\bigr)
    \prod_{i}\bigl(1+s\kappa_i(y)\bigr)\dd s\dd\Hm(y)
     \nonumber \\
  &\le C_G\,\Hm(A)\int_0^\infty e^{-s^2/(2\lambda)}(1+\kappa_*s)^{d-1}\dd s,
  \nonumber
\end{align}
and the substitution $s=\sqrt\lambda\,t$ bounds the last integral by
$\sqrt\lambda\int_0^\infty e^{-t^2/2}(1+\kappa_*t)^{d-1}\dd t=C\sqrt\lambda$
for $0<\lambda\le1$.
\end{proof}
 
\medskip
\noindent\textit{Proof of Theorem \ref{thm:stationary-main} (ii).}
 
\medskip
\noindent\emph{(Upper bound)}
We prove the upper bound in cases (a) and (b) simultaneously,
using the constant $C_{\mathrm{st}}$ defined in
\eqref{eq:stationary-stability-constant}.
 
\medskip
\noindent\textbf{Step 1: uniform Gibbs bounds and an interior $L^\infty$
estimate.}
Set $P_\lambda:=\Phi+W*\pi_\lambda$, $P:=\Phi+W*\pi$,
$M:=\|\Phi\|_\infty+\|W\|_\infty$, and
\[
  Z_\lambda:=\int_{\Rd}e^{-U_\lambda-P_\lambda}\dd x,
  \qquad
  Z_0:=\int_Ke^{-P}\dd x .
\]
Since $\pi_\lambda,\pi$ are probability densities,
$\|P_\lambda\|_\infty,\|P\|_\infty\le M$; since $U_\lambda\equiv0$ on $K$,
\[
  |K|e^{-M}\le Z_0\le|K|e^{M},
  \qquad
  |K|e^{-M}\le Z_\lambda\le e^{M}\bigl(|K|+C\sqrt\lambda\bigr)\le C,
\]
the upper bound on $Z_\lambda$ by Lemma~\ref{lem:layer} with $q=0$.
Consequently there are $0<c_0\le C_0<\infty$, independent of
$\lambda\in(0,1]$, with
\begin{equation}\label{eq:W2-gibbs-bounds}
\begin{aligned}
c_0&\le\pi_\lambda\le C_0\quad\text{and}\quad
c_0\le\pi\le C_0\quad\text{on }K,\\
\pi_\lambda(x)&\le C_G\,e^{-d_K(x)^2/(2\lambda)}\quad\text{on }\Rd.
\end{aligned}
\end{equation}
where $C_G:=e^{2M}/|K|$.  Next, by part~(i),
\[
  \|P_\lambda-P\|_{L^\infty(\Rd)}
  \le\|W\|_\infty\|\pi_\lambda-\pi\|_{L^1(\Rd)}
  \le C\,{C_{\mathrm{st}}}\sqrt\lambda,
\]
and therefore
\[
\begin{aligned}
|Z_\lambda-Z_0|
&\le\int_K\bigl|e^{-P_\lambda}-e^{-P}\bigr|\dd x
 +\int_{K^c}e^{-U_\lambda-P_\lambda}\dd x\\
&\le C\|P_\lambda-P\|_{L^\infty(\Rd)}+C\sqrt\lambda
 \le C\,{C_{\mathrm{st}}}\sqrt\lambda.
\end{aligned}
\]
again by Lemma~\ref{lem:layer}.  Inserting these into the two Gibbs formulas
and using the uniform two-sided bounds on $Z_\lambda,Z_0$ yields
\begin{equation}\label{eq:W2-interior-Linfty}
  \|\pi_\lambda-\pi\|_{L^\infty(K)}\ \le C\,{C_{\mathrm{st}}}\sqrt\lambda.
\end{equation}
 
\medskip
\noindent\textbf{Step 2: projection of the exterior mass.}
Write $T:=\proj_K$ and define
\[
  \sigma_\lambda:=T_\#\bigl(\pi_\lambda\one_{K^c}\dd x\bigr)
  \in\mathcal M_+(\Gamma),
  \qquad
  \mu_\lambda^\sharp:=\pi_\lambda\one_K\dd x+\sigma_\lambda .
\]
Then $\mu_\lambda^\sharp\in\mcP(K)$, since its total mass is
$\mu_\lambda(K)+\mu_\lambda(K^c)=1$.  As $T$ is Borel and
$|x-Tx|=d_K(x)$, the plan $(\mathrm{id},T)_\#\mu_\lambda$ belongs to
$\Pi(\mu_\lambda,\mu_\lambda^\sharp)$, so by \eqref{eq:W2-gibbs-bounds} and
Lemma~\ref{lem:layer} with $q=2$,
\begin{equation}\label{eq:W2-projection-cost}
  W_2^2(\mu_\lambda,\mu_\lambda^\sharp)
  \le\int_{\Rd}|x-Tx|^2\pi_\lambda\dd x
  =\int_{K^c}d_K^2\,\pi_\lambda\dd x
  \le C_GC\lambda^{3/2},
\end{equation}
so that $W_2(\mu_\lambda,\mu_\lambda^\sharp)\le C\lambda^{3/4}$.  Note that no
splitting into a collar and a far exterior region is needed: the Gaussian
bound \eqref{eq:W2-gibbs-bounds} controls the whole of $K^c$ at once.
Lemma~\ref{lem:projected-layer}, applied with $g=\pi_\lambda$, gives in
addition
\begin{equation}\label{eq:W2-surface-density}
  \sigma_\lambda\ \le\ C\sqrt\lambda\;\Hm\!\restriction_\Gamma .
\end{equation}
 
\medskip
\noindent\textbf{Step 3: $\mu_\lambda^\sharp$ is $O(\sqrt\lambda)$-close to
$\mu$ in $\dot H^{-1}(K)$.}
Both $\mu_\lambda^\sharp$ and $\mu$ are probability measures on $K$.  Let
$\varphi\in H^1(K)$ with $\int_K\varphi\dd x=0$ and
$\|\nabla\varphi\|_{L^2(K)}\le1$.  In the notation of
Lemma~\ref{lem:Hminus1-transport} (with $h=\pi_\lambda|_K$, $f=\pi$,
$\sigma=\sigma_\lambda$),
\[
  \langle\mu_\lambda^\sharp-\mu,\varphi\rangle
  =\int_K(\pi_\lambda-\pi)\varphi\dd x
  +\int_\Gamma\Tr\varphi\dd\sigma_\lambda .
\]
By \eqref{eq:W2-interior-Linfty}, Cauchy--Schwarz and
\eqref{eq:W2-poincare},
\[
  \Bigl|\int_K(\pi_\lambda-\pi)\varphi\dd x\Bigr|
  \le\|\pi_\lambda-\pi\|_{L^\infty(K)}\,|K|^{1/2}\,\|\varphi\|_{L^2(K)}
  \le C\,{C_{\mathrm{st}}}\sqrt\lambda\,\|\nabla\varphi\|_{L^2(K)} ,
\]
while \eqref{eq:W2-surface-density}, the trace inequality and
\eqref{eq:W2-poincare} give
\[
  \Bigl|\int_\Gamma\Tr\varphi\dd\sigma_\lambda\Bigr|
  \le C\sqrt\lambda\,\|\Tr\varphi\|_{L^1(\Gamma)}
  \le C\sqrt\lambda\,\|\varphi\|_{H^1(K)}
  \le C\sqrt\lambda\,\|\nabla\varphi\|_{L^2(K)} .
\]
Taking the supremum over all such $\varphi$,
\begin{equation}\label{eq:W2-Hminus1}
  \|\mu_\lambda^\sharp-\mu\|_{\dot H^{-1}(K)}
  \ \le C\,{C_{\mathrm{st}}}\sqrt\lambda.
\end{equation}
 
\medskip
\noindent\textbf{Step 4: interior redistribution and conclusion.}
By \eqref{eq:W2-gibbs-bounds} we have $\pi\ge c_0$ and
$\pi_\lambda|_K\ge c_0$ on $K$, and \eqref{eq:W2-surface-density} shows that
$\mu_\lambda^\sharp$ satisfies the hypotheses of
Lemma~\ref{lem:Hminus1-transport} with $\Lambda=C\sqrt\lambda$.  Hence, by
that lemma and \eqref{eq:W2-Hminus1},
\[
  W_2(\mu_\lambda^\sharp,\mu)
  \ \le\ c_0^{-1/2}\,\|\mu_\lambda^\sharp-\mu\|_{\dot H^{-1}(K)}
  \ \le C\,{C_{\mathrm{st}}}\sqrt\lambda.
\]
Combining with \eqref{eq:W2-projection-cost} and the triangle inequality, and
using $\lambda^{3/4}\le\sqrt\lambda$ for $0<\lambda\le1$,
\[
  W_2(\mu_\lambda,\mu)
  \ \le\ W_2(\mu_\lambda,\mu_\lambda^\sharp)+W_2(\mu_\lambda^\sharp,\mu)
  \ \le C\lambda^{3/4}+C\,{C_{\mathrm{st}}}\sqrt\lambda.
\]
{In case (a), tracking $C_{\mathrm{st}}$ gives the stated
$C_2(1-\kappa)^{-1}\sqrt\lambda$ bound; in case (b) it gives
$C_2\sqrt\lambda$.}
 
\bigskip
\noindent\emph{(Lower bound)}
The exterior mass creates an equal total deficit inside $K$.  The Gibbs
ratio bound below prevents this deficit from being concentrated in an
arbitrarily thin interior strip.  A Lipschitz test across a fixed-width
strip therefore detects an $O(\sqrt\lambda)$ discrepancy already in $W_1$.
Recall $\beta_\lambda:=\int_{K^c}\pi_\lambda\dd x$ and that, for probability
densities $\rho_1,\rho_2$,
\begin{equation}\label{eq:osc-interaction}
  \osc_{x\in K}\bigl(W*(\rho_1-\rho_2)\bigr)(x)
  \le\dob(W)\,\|\rho_1-\rho_2\|_{L^1(\Rd)} .
\end{equation}
 
\medskip
\noindent\textbf{Step 1: exterior mass and $L^1$ discrepancy.}
By \eqref{eq:beta-lower} and \eqref{pi_lambda--pi--L1},
\begin{equation}\label{eq:W2-lower-inputs}
  \beta_\lambda\ \ge\ b_0\sqrt\lambda,
  \qquad
  \|\pi_\lambda-\pi\|_{L^1(\Rd)}\ \le {C_{\mathrm{st}}}\sqrt\lambda,
\end{equation}
with $b_0=b_0(K,d,\|\Phi\|_\infty,\|W\|_\infty)>0$.
 
\medskip
\noindent\textbf{Step 2: on $K$, $\pi_\lambda/\pi$ is almost constant.}
Let $q_\lambda:=\pi_\lambda/\pi$ on $K$.  Since $U_\lambda\equiv0$ on $K$, the
two Gibbs formulas give
$\log q_\lambda=c_\lambda-\bigl(W*(\pi_\lambda-\pi)\bigr)$ on $K$ for a
constant $c_\lambda$, so \eqref{eq:osc-interaction} and
\eqref{eq:W2-lower-inputs} yield
\[
  \osc_K\log q_\lambda\ \le\ A\sqrt\lambda,
  \qquad
  {A:=\dob(W)\,C_{\mathrm{st}}.}
\]
Since $\int_Kq_\lambda\dd\pi=\int_K\pi_\lambda\dd x=1-\beta_\lambda$, the
value $1-\beta_\lambda$ lies between $\inf_Kq_\lambda$ and $\sup_Kq_\lambda$,
whence
\[
  e^{-A\sqrt\lambda}(1-\beta_\lambda)\ \le\ q_\lambda\ \le\
  e^{A\sqrt\lambda}(1-\beta_\lambda)\qquad\text{on }K .
\]
Set
\[
  \lambda_*:=\min\{1,A^{-2}\}\in(0,1]
  \qquad(\lambda_*:=1\ \text{ if }A=0),
  \qquad
  {C_A:=2A} .
\]
For $0<\lambda\le\lambda_*$ we have $A\sqrt\lambda\le1$, so
$e^{A\sqrt\lambda}-1\le(e-1)A\sqrt\lambda\le {C_A}\sqrt\lambda$ by convexity,
and since $0\le1-\beta_\lambda\le1$,
\begin{equation}\label{eq:W2-q-additive}
  \bigl|q_\lambda(x)-(1-\beta_\lambda)\bigr|\ \le\ {C_A}\sqrt\lambda,
  \qquad x\in K .
\end{equation}
 
{
\medskip
\noindent\textbf{Step 3: one Lipschitz test across a fixed boundary strip.}
Let $r(x):=\dist(x,K^c)$, so $r$ is $1$-Lipschitz on $\Rd$, vanishes on
$K^c$, and equals the boundary depth on $K$.  Put
\[
  A_t:=\{x\in K:r(x)\le t\},
  \qquad F(t):=\pi(A_t).
\]
Let $\rho_K$ be the inradius of $K$.  For $0<s<\rho_K$, the inner parallel
body $K_{-s}=\{r\ge s\}$ is convex and
$\Hm(\partial K_{-s})\le\Hm(\Gamma)$.  Since $|\nabla r|=1$ a.e., the
coarea formula and the uniform upper bound $\pi\le C_0$ give
\[
  F(t)\le C_0|A_t|
  \le C_0\Hm(\Gamma)t=:M_rt,
  \qquad 0\le t<\rho_K.
\]
Choose
\[
  t_*:=\min\left\{\frac{\rho_K}{2},
  \frac{b_0}{2M_r(b_0+{C_A})}\right\}>0,
  \qquad
  F(t_*)\le\frac{b_0}{2(b_0+{C_A})}.
\]
The function $\varphi=(t_*-r)_+$ is $1$-Lipschitz.  Since
$\mu_\lambda|_K=q_\lambda\mu$ and $\varphi=t_*$ on $K^c$,
\begin{align*}
  \int\varphi\dd(\mu_\lambda-\mu)
  &=t_*\beta_\lambda
    +\int_{A_{t_*}}\varphi(q_\lambda-1)\dd\pi\\
  &=\beta_\lambda\left(t_*-
      \int_{A_{t_*}}\varphi\dd\pi\right)
    +\int_{A_{t_*}}\varphi
      \bigl(q_\lambda-(1-\beta_\lambda)\bigr)\dd\pi\\
  &\ge t_*\sqrt\lambda
      \bigl[b_0(1-F(t_*))-{C_A}F(t_*)\bigr]
   \ge \frac{b_0t_*}{2}\sqrt\lambda,
\end{align*}
where we used \eqref{eq:W2-lower-inputs},
\eqref{eq:W2-q-additive}, $0\le\varphi\le t_*$, and
$\int_{A_{t_*}}\varphi\dd\pi\le t_*F(t_*)$.
Kantorovich--Rubinstein duality and $W_1\le W_2$ therefore give
\[
  W_2(\mu_\lambda,\mu)\ge W_1(\mu_\lambda,\mu)
  \ge\frac{b_0t_*}{2}\sqrt\lambda,
  \qquad 0<\lambda\le\lambda_*.
\]
The constants have the dependencies stated in the theorem, so the claim
holds with $c_2=b_0t_*/2$. \qed
}

\begin{remark}[Constants]\label{rem:W2-constants}
{The proof uses $W$ only through $\|W\|_\infty$,
$\dob(W)\le2\|W\|_\infty$, and part~(i), which gives the dependence stated
in Theorem~\ref{thm:stationary-main}(ii).}
\end{remark}

\section{Proof of Proposition \ref{prop:kernel--free}, Corollary~\ref{cor:probabilistic}, and stochastic consequences}
\label{sec:probabilistic}\label{sec:sde-numerics}

This section proves Proposition~\ref{prop:kernel--free} and
Corollary~\ref{cor:probabilistic}.  We first treat the kernel-free Gibbs
measures, where the stationary states are explicit.  These estimates
extend the fixed-potential comparison to potentials that are merely
bounded below.  We then record the kernel-free reflected-Langevin
specialization and the particle approximation.

\subsection{Proof of Proposition \ref{prop:kernel--free}}
\label{sec:linear}

Throughout this subsection $W\equiv0$ and $\Phi=f$, so that
$b[\mu]=\nabla f$ for every $\mu$ and the penalized dynamics
\eqref{eq:pen-pde} is the linear Fokker--Planck equation
\begin{equation}\label{eq:linear-pen-pde}
  \partial_t\rho^\lambda
  =\divg\bigl(\nabla\rho^\lambda+\rho^\lambda\nabla U_\lambda
  +\rho^\lambda\nabla f\bigr)
  =\divg\bigl(\nabla\rho^\lambda+\rho^\lambda\nabla(f+U_\lambda)\bigr)
\end{equation}
of the penalized overdamped Langevin diffusion
\begin{equation}\label{eq:linear-sde}
  \dd X^\lambda_t
  =-\nabla f(X^\lambda_t)\dd t
  -\frac1\lambda\bigl(X^\lambda_t-\proj_KX^\lambda_t\bigr)\dd t
  +\sqrt2\dd B_t .
\end{equation}
This is the continuous-time model behind quadratic-penalty samplers for
distributions supported on $K$
\cite{BrosseDurmusMoulinesPereyra2018,GurbuzbalabanHuZhu2024}.
The explicit Gibbs structure yields Proposition~\ref{prop:kernel--free},
extending the fixed-potential comparison in Subsection~\ref{sec:static}
to the weaker assumptions below.

For the static theory the standing hypotheses of
Assumption~\ref{ass:coefficients} can be relaxed: no global boundedness of
$\nabla f$ is needed.

\begin{assumption}[Linear potential]\label{ass:linear}
$f\in C^2(\Rd)$ and $\inf_{\Rd}f>-\infty$.
\end{assumption}

\subsubsection{Explicit invariant measures}

Under Assumption~\ref{ass:linear}, the bound
$e^{-f-U_\lambda}\le e^{-\inf f}e^{-U_\lambda}$ gives
$Z_\lambda:=\int_{\Rd}e^{-f-U_\lambda}\dd x<\infty$;
continuity on the compact set $K$ gives
$Z:=\int_K e^{-f}\dd x\in(0,\infty)$. The Gibbs densities are

\begin{equation}\label{eq:linear-gibbs}
\begin{aligned}
\mu_{\lambda,\infty}(\dd x)&=\rli(x)\dd x,
&\rli&:=Z_\lambda^{-1}e^{-f-U_\lambda},\\
\mu_\infty(\dd x)&=\ri(x)\dd x,
&\ri&:=Z^{-1}e^{-f}\one_K.
\end{aligned}
\end{equation}
The drift of \eqref{eq:linear-sde} is the gradient field
$-\nabla(f+U_\lambda)$ (Lemma~\ref{lem:convex-geom}(i)), so the penalized
generator $\mcL_\lambda h=\Delta h-\nabla(f+U_\lambda)\cdot\nabla h$ is
symmetric in $L^2(\mu_{\lambda,\infty})$, and the sub-Markovian semigroup
$P^\lambda_t$ generated by the associated closed symmetric Dirichlet form
is symmetric with respect to $\mu_{\lambda,\infty}$ --- reversible, and
conservative whenever the penalized flow is well posed in the class of
Definition~\ref{def:pen-sol}: the linear penalized flow has an
\emph{explicit} invariant measure, in contrast with the fixed-point
characterization \eqref{eq:fixed-point} of the interacting case.

Because the penalty vanishes identically on $K$, all of the discrepancy
between $\mu_{\lambda,\infty}$ and $\mu_\infty$ is carried by the exterior
Gaussian boundary layer.  The basic layer estimate is the case $\beta=0$ of
the tilted layer bound proved in the appendix.

\begin{lemma}\label{lem:linear-layer}
For every $q\ge0$ there exists $C_{K,q}>0$ such that for all
$\lambda\in(0,1]$,
\begin{equation}\label{eq:linear-coarea}
  \int_{K^c}\dist(x)^q
  \exp\Bigl(-\frac{\dist(x)^2}{2\lambda}\Bigr)\dd x
  \ \le\ C_{K,q}\,\lambda^{(q+1)/2}.
\end{equation}
In particular
$\int_{K^c}e^{-\dist^2/2\lambda}\dd x\le C_K\sqrt\lambda$.
\end{lemma}

\begin{proof}
This is Lemma~\ref{lem:layer} with $\beta=0$: the coarea formula for the
distance function combined with the Steiner-type bound
$\Hm(\{\dist=s\})\le C_K(1+s)^{d-1}$ for the outer parallel surfaces of a
compact convex body.
\end{proof}

\begin{corollary}\label{cor:linear-moments}
For every $q\ge0$ there exists $C_q>0$ such that for all
$\lambda\in(0,1]$,
\begin{equation}
  \int_{K^c}\dist(x)^q\,\rli(x)\dd x\ \le\ C_q\,\lambda^{(q+1)/2}.
\end{equation}
\end{corollary}

\begin{proof}
Since $Z_\lambda\ge Z$ and $f$ is bounded from below,
\[
  \int_{K^c}\dist(x)^q\rli(x)\dd x
  \ \le\ \frac{e^{-\inf_{\Rd}f}}{Z}
  \int_{K^c}\dist(x)^q
  \exp\Bigl(-\frac{\dist(x)^2}{2\lambda}\Bigr)\dd x,
\]
and the claim follows from Lemma~\ref{lem:linear-layer}.
\end{proof}

\begin{lemma}\label{thm:linear-partition}
For every $\lambda>0$,
\[
  Z_\lambda=Z+T_\lambda,
  \qquad
  T_\lambda:=\int_{K^c}e^{-f-U_\lambda}\dd x,
\]
and there exist $c,C>0$ such that for all $\lambda\in(0,1]$,
\[
  c \sqrt{\lambda} \le T_\lambda\le C\sqrt\lambda .
\]
Consequently $Z_\lambda\to Z$ as $\lambda\downarrow0$.
\end{lemma}

\begin{proof}
Splitting the integral over $K$ and $K^c$ gives $Z_\lambda=Z+T_\lambda$.
Lemma~\ref{lem:linear-layer} and the lower bound on $f$ give
$T_\lambda\le e^{-\inf f}C_K\sqrt\lambda$.
For the reverse bound, put $K_1=\{\dist(x,K)<1\}$ and
$M_1=\sup_{K_1}f<\infty$, and let
$E_\lambda=\{0<\dist(x,K)<\sqrt\lambda\}$. For $0<\lambda\le1$,
$E_\lambda\subset K_1$ and $U_\lambda\le1/2$ there. Convexity and the
Steiner formula yield
\[
T_\lambda\ge e^{-M_1-1/2}|E_\lambda|
\ge e^{-M_1-1/2}\Hm(\partial K)\sqrt\lambda.
\]
The two-sided bound and convergence of $Z_\lambda$ follow.
\end{proof}

\begin{proposition}
\label{thm:linear-tv}
For every $\lambda>0$,
\begin{equation}
\label{eq:linear-tv-identity}
\|\rli-\ri\|_{L^1(\Rd)}
=
\frac{2T_\lambda}{Z_\lambda}.
\end{equation}
Consequently, there exist constants $c,C>0$ such that
\[
c\sqrt{\lambda}
\le
\|\rli-\ri\|_{L^1(\Rd)}
\le
C\sqrt{\lambda},
\qquad
0<\lambda\le1.
\]
\end{proposition}

\begin{proof}
On $K$, $\rli=(Z/Z_\lambda)\ri\le\ri$; on $K^c$, $\ri=0$.
Thus the interior deficit and exterior mass coincide:
\[
\|\rli-\ri\|_{L^1(\Rd)}
=\left(1-\frac{Z}{Z_\lambda}\right)+\frac{T_\lambda}{Z_\lambda}
=\frac{2T_\lambda}{Z_\lambda}.
\]
Lemma~\ref{thm:linear-partition} gives
$c_T\sqrt\lambda\le T_\lambda\le C_T\sqrt\lambda$.
Since $U_\lambda\ge U_1$ for $\lambda\le1$,
$Z\le Z_\lambda\le Z_1$, so
\[
\frac{2c_T}{Z_1}\sqrt\lambda
\le\|\rli-\ri\|_{L^1(\Rd)}
\le\frac{2C_T}{Z}\sqrt\lambda.\qedhere
\]
\end{proof}

The proof of Proposition~\ref{thm:linear-tv} uses only that the potential is
continuous and bounded from below, together with the layer bound
\eqref{eq:linear-coarea} and the boundary regularity of $K$.

\begin{proposition}
\label{thm:linear-w2}
There exist constants $c,C>0$ such that for all $\lambda\in(0,1]$,
\[
  c\sqrt\lambda \le W_1(\mu_{\lambda,\infty},\mu_\infty)\le 
  W_2(\mu_{\lambda,\infty},\mu_\infty)\le C\sqrt\lambda.
\]
\end{proposition}

\begin{proof}
For the upper bound, $Z_\lambda=Z+T_\lambda$ and
$T_\lambda\le C\sqrt\lambda$ give $Z\le Z_\lambda\le Z+C$ and
\[
\|\rli-\ri\|_{L^\infty(K)}
=\|e^{-f}\|_{L^\infty(K)}\frac{T_\lambda}{ZZ_\lambda}
\le C\sqrt\lambda.
\]
Also,
\begin{equation}\label{eq:linear-exterior-Gaussian}
\rli(x)\le\frac{e^{-\inf_{\Rd}f}}{Z}
e^{-\dist(x,K)^2/(2\lambda)},\qquad x\in\Rd.
\end{equation}
Continuity of $f$ on $K$ and the uniform normalizer bounds give
$\rli,\ri\ge c_0>0$ on $K$. These are the hypotheses used in
Steps~2--4 of the upper-bound proof of
Theorem~\ref{thm:stationary-main}(ii), with $\kappa=0$; the same projection
and redistribution argument yields $W_2\le C\sqrt\lambda$.

For the lower bound, set
$\beta_\lambda=T_\lambda/Z_\lambda=\mu_{\lambda,\infty}(K^c)$.
Since $\rli=(1-\beta_\lambda)\ri$ on $K$, the $1$-Lipschitz function
$h(x)=\dist(x,K^c)$ and Kantorovich--Rubinstein duality give
\[
W_1(\mu_{\lambda,\infty},\mu_\infty)
\ge\beta_\lambda\int_K h\,\dd\mu_\infty\ge c\sqrt\lambda.
\]
Here $\int_K h\,\dd\mu_\infty>0$ because $K$ has nonempty interior,
and $\beta_\lambda\ge c\sqrt\lambda$ by
Lemma~\ref{thm:linear-partition} and $Z_\lambda\le Z_1$.
Finally, $W_1\le W_2$ completes the proof.
\end{proof}

\begin{proof}[Proof of Proposition~\ref{prop:kernel--free}]
{The partition estimate is
Lemma~\ref{thm:linear-partition}, the exact $L^1$ identity and rate are
Proposition~\ref{thm:linear-tv}, and the two-sided Wasserstein bounds are
Proposition~\ref{thm:linear-w2}.}
\end{proof}

\subsection{Further reversible consequence: all-time Gibbs domination}
\label{subsec:alltime-domination}

{
This additional linear result is not used in
Theorem~\ref{thm:stationary-main}.  Reversibility replaces the finite-horizon
barrier constant by an all-time Markov-semigroup domination.
}

\begin{proposition}\label{prop:linear-domination}
Let Assumptions~\ref{ass:geometry} and \ref{ass:linear} hold, assume the
penalized flow \eqref{eq:linear-pen-pde} is well posed in the class of
Definition~\ref{def:pen-sol} (for instance under the hypotheses of
Proposition~\ref{prop:pen-wp}), and let the initial density satisfy
\[
  0\le\rho_0\le M'\ri\ \text{ a.e.\ on }K,
  \qquad
  \rho_0=0\ \text{ a.e.\ on }K^c,
  \qquad
  \int_K\rho_0\dd x=1 .
\]
Then, for every $\lambda\in(0,1]$ and every $t\ge0$,
\[
  0\ \le\ \rho^\lambda_t\ \le\ M'\,\frac{Z_\lambda}{Z}\,\rli
  \ \le\ M'\,\frac{Z_1}{Z}\,\rli
  \qquad\dd x\text{-a.e.}
\]
In particular, by Corollary~\ref{cor:linear-moments}, for every $q\ge0$,
\[
  \sup_{t\ge0}\int_{K^c}\dist(x)^q\,\rho^\lambda_t(x)\dd x
  \ \le\ C_q\,\lambda^{(q+1)/2},
\]
with $C_q$ independent of $t$ and $\lambda\in(0,1]$.
\end{proposition}

{
\begin{proof}
Let $h_t^\lambda:=\rho_t^\lambda/\rli$.  The identity
\(\nabla\rli+\rli\nabla(f+U_\lambda)=0\) transforms the equation into
\[
 \rli\,\partial_t h^\lambda=\divg(\rli\nabla h^\lambda).
\]
Detailed balance identifies this relative-density evolution with the
Markov semigroup \(P_t^\lambda\) in \(L^2(\mu_{\lambda,\infty})\);
uniqueness of the penalized evolution gives
\(h_t^\lambda=P_t^\lambda h_0^\lambda\).
Since $\rho_0=0$ on $K^c$ and, on $K$,
\[
  h_0^\lambda=\frac{\rho_0}{\ri}\frac{Z_\lambda}{Z}
  \le M'\frac{Z_\lambda}{Z},
\]
positivity and the $L^\infty$ contraction of $P_t^\lambda$ imply
$0\le h_t^\lambda\le M'Z_\lambda/Z$.  Multiplication by $\rli$ gives the
first domination bound.  Because $U_\lambda\ge U_1$ for $\lambda\le1$,
$Z_\lambda\le Z_1$, giving the second.  The exterior moment estimates then
follow directly from Corollary~\ref{cor:linear-moments}.
\end{proof}
}

{
\begin{remark}
Assumption~\ref{ass:data} implies the hypothesis above with
$M'=MZe^{\sup_K f}$, because $\ri\ge Z^{-1}e^{-\sup_Kf}$ on $K$.
Thus the linear exterior-moment constants are uniform for all $t\ge0$,
in contrast with the finite-horizon interacting barrier.
\end{remark}
}

\subsection{Proof of Corollary~\ref{cor:probabilistic}}

Existence and uniqueness for the reflected McKean--Vlasov SDE
\eqref{eq:rmv-sde} is due to Sznitman \cite{Sznitman1984} (see also
\cite{Tanaka1979,LionsSznitman1984}).  The required hypotheses --- bounded
state-Lipschitz drift and $W_1$-Lipschitz dependence on the law on a convex
domain --- follow from \eqref{eq:drift-bounds} and \eqref{eq:drift-W1}.

We use the following pathwise estimates to identify the reflected marginals
and separate the particle error from the soft-wall bias, uniformly in
$\lambda$.

{
\begin{proposition}[Pathwise penalization and $\lambda$-uniform propagation of chaos]
\label{prop:pathwise}
Assume the hypotheses of Theorem~\ref{thm:main}.  Couple the reflected
solutions $X^i$ with local times $L^i$, the penalized McKean--Vlasov solutions
$X^{i,\lambda}$, and the particles $X^{i,N,\lambda}$ by the same i.i.d.\ initial
data and Brownian motions.  Set
\[
  \varepsilon_\lambda:=\bigl(\lambda\log(1/\lambda)\bigr)^{1/4},
  \qquad 0<\lambda\le\tfrac12.
\]
There is $C$, depending only on $K,d,T$ and the coefficient bounds, but not on
$N$ or $\lambda$, such that for every $i\le N$,
\begin{align*}
  \mathbb E\sup_{t\le T}|X^{i,\lambda}_t-X^i_t|
  &\le C\varepsilon_\lambda,\\
  \mathbb E\sup_{t\le T}|X^{i,N,\lambda}_t-X^{i,\lambda}_t|
  &\le CN^{-1/2},\\
  \mathbb E\sup_{t\le T}\Bigl|
    \int_0^t\nabla U_\lambda(X^{i,\lambda}_s)\dd s
    -\int_0^t n(X^i_s)\dd L^i_s\Bigr|
  &\le C\varepsilon_\lambda.
\end{align*}
Consequently,
\[
  \sup_{t\le T}W_1(\mathrm{Law}(X^{1,\lambda}_t),\mathrm{Law}(X^1_t))
  \le C\varepsilon_\lambda,
  \qquad
  \mathbb E\sup_{t\le T}|X^{i,N,\lambda}_t-X^i_t|
  \le C(N^{-1/2}+\varepsilon_\lambda).
\]
\end{proposition}

\begin{proof}
Write $\mu_t:=\mathrm{Law}(X^1_t)$,
$\mu_t^\lambda:=\mathrm{Law}(X^{1,\lambda}_t)$ and
$L_2:=\|\nabla_y\nabla_xW\|_\infty$.  Let $Y^\lambda$ solve the penalized
SDE with frozen drift $b[\mu_t]$, driven by the same $(B^1,\xi^1)$ as $X^1$.
The prescribed-drift estimate of \cite[Thm.~4.1(ii)]{Slominski2013} gives
\[
  \mathbb E\sup_{t\le T}|Y_t^\lambda-X^1_t|\le C\varepsilon_\lambda.
\]
The hypotheses hold because $b[\mu_t]$ is uniformly bounded and Lipschitz in
space, and $t\mapsto\mu_t$ is $W_1$-continuous since $X^1$ has continuous
paths in the compact set $K$.  The cited proof applies to every real
penalty coefficient $n=\lambda^{-1}\ge2$.  Its constants for deterministic
initial points can be chosen uniformly on the compact set $K$; conditioning
on $\xi^1$ therefore gives the estimate for the present random initial data.

Set $Z:=X^{1,\lambda}-Y^\lambda$.  Synchronous coupling cancels the noise,
and monotonicity of $\nabla U_\lambda$ gives
\[
  |Z_t|\le\int_0^t\bigl(B_1|Z_s|+L_2W_1(\mu_s^\lambda,\mu_s)\bigr)\dd s,
  \qquad
  W_1(\mu_s^\lambda,\mu_s)
  \le\mathbb E|Z_s|+\mathbb E|Y_s^\lambda-X^1_s|.
\]
Thus $a(t):=\mathbb E\sup_{r\le t}|Z_r|$ satisfies
\[
  a(t)\le(B_1+L_2)\int_0^t a(s)\dd s+C\varepsilon_\lambda t.
\]
Gr\"onwall bounds $a(T)$ by $C\varepsilon_\lambda$; adding the bound for
$Y^\lambda-X^1$ yields the first estimate.

For $\Delta^i:=X^{i,N,\lambda}-X^{i,\lambda}$, the same monotonicity gives
\[
  |\Delta_t^i|\le\int_0^t\left(
  B_1|\Delta_s^i|+\frac{L_2}{N}\sum_{j=1}^N|\Delta_s^j|+|R_s^i|\right)\dd s,
\]
where
\[
  R_s^i:=\frac1N\sum_{j=1}^N\nabla_xW(X_s^{i,\lambda},X_s^{j,\lambda})
  -\int\nabla_xW(X_s^{i,\lambda},y)\,\mu_s^\lambda(\dd y).
\]
Conditionally on $X_s^{i,\lambda}$, the terms with $j\ne i$ are centered and
independent; the $j=i$ term is $O(N^{-1})$.  Hence
$\mathbb E|R_s^i|\le C\|\nabla_xW\|_\infty N^{-1/2}$ by conditional
Cauchy--Schwarz.  By exchangeability,
$a_N(t):=\mathbb E\sup_{r\le t}|\Delta_r^1|$ satisfies
\[
  a_N(t)\le(B_1+L_2)\int_0^t a_N(s)\dd s+CtN^{-1/2}.
\]
Gr\"onwall gives the second estimate, with no $\lambda^{-1}$ dependence.

Finally, subtracting the integral forms of the penalized and reflected SDEs
gives
\[
  \int_0^t\nabla U_\lambda(X^{i,\lambda}_s)\dd s
  -\int_0^t n(X^i_s)\dd L^i_s
  =X^i_t-X^{i,\lambda}_t
  -\int_0^t\bigl(b[\mu_s^\lambda](X_s^{i,\lambda})-b[\mu_s](X_s^i)\bigr)\dd s.
\]
The first estimate and \eqref{eq:drift-bounds}--\eqref{eq:drift-W1} give the
third.  The two stated consequences follow by coupling and the triangle
inequality.
\end{proof}
}

\begin{proof}[Proof of Corollary~\ref{cor:probabilistic}]
\textbf{Step 1: identification of the reflected marginals.}
{Let $\mu_t:=\mathrm{Law}(X_t)$.  By
Proposition~\ref{prop:pathwise}, $\mathbb E|X^{1,\lambda}_t-X^1_t|\le
C\varepsilon_\lambda\to0$ for every $t\in[0,T]$, so
$\mathrm{Law}(X^{\lambda}_t)\to\mu_t$ narrowly.  Since
$\mathrm{Law}(X^{\lambda}_t)=\rho^\lambda_t\dd x$
(Proposition~\ref{prop:pen-wp}) and $\rho^\lambda_t\to\rho_t\one_K$
narrowly for every $t$ (Theorem~\ref{thm:main}(i)), we obtain
$\mu_t=\rho_t\dd x$ for every $t\in[0,T]$.}  This proves part~(i).

\textbf{Step 2: convergence of the time marginals.}
By Proposition~\ref{prop:pen-wp},
$\mathrm{Law}(X_t^\lambda)=\rho_t^\lambda\dd x$.  Combining this identity
with Step~1 and Theorem~\ref{thm:main}(i) gives both the uniform
bounded-Lipschitz convergence and the stated density convergences.  This
proves part~(ii).

\textbf{Step 3: averaged penalty-force occupation measures.}
For every
$V\in C_c([0,T]\times\Rd;\Rd)$, Fubini's theorem and the marginal identity
give
\[
  \int_{[0,T]\times\Rd}V\cdot\dd\mathsf Q^\lambda
  =\mathbb E\int_0^T
  V(t,X_t^\lambda)\cdot
  \frac{X_t^\lambda-\proj_KX_t^\lambda}{\lambda}\,\dd t
  =\int_{[0,T]\times\Rd}V\cdot\dd Q^\lambda.
\]
Part~(iii) now follows directly from Theorem~\ref{thm:main}(iii).
\end{proof}

\subsection{Kernel-free reflected Langevin specialization}

We now specialize the dynamical and stationary results to $W\equiv0$.

\begin{corollary}\label{cor:linear-limit}
Let Assumptions~\ref{ass:geometry} and \ref{ass:data} hold, and let
$W\equiv0$, $\Phi=f$ with $f\in C^2(\Rd)$, $\nabla f\in C^1_b(\Rd;\Rd)$,
and $\inf_{\Rd}f>-\infty$.  Then:
\begin{enumerate}[label=(\roman*),leftmargin=2.2em]
\item Assumption~\ref{ass:coefficients} holds with
$B_0=\|\nabla f\|_\infty$, $B_1=\|\nabla^2f\|_\infty$, and $L_W=0$.
All conclusions of Theorem~\ref{thm:main}, including the density, full-flux,
and boundary-reaction limits, apply to the reflected Fokker--Planck equation
\[
\partial_t\rho=\divg(\nabla\rho+\rho\nabla f),\qquad
(\nabla\rho+\rho\nabla f)\cdot n=0\quad\text{on }\Gamma.
\]
\item $\rho$ is the time-marginal density of the normally reflected
Langevin diffusion on $K$,
\[
  \dd X_t=-\nabla f(X_t)\dd t+\sqrt2\dd B_t-n(X_t)\dd L_t,
\]
where $L$ is the boundary local time defined after \eqref{eq:rmv-sde}.
\item The Gibbs measures $\mu_{\lambda,\infty}$ and $\mu_\infty$ of
\eqref{eq:linear-gibbs} are stationary for the penalized and for the
reflected dynamics respectively.  They are the unique such states:
$\mu_\infty$ among stationary variational solutions of the reflected
equation, and $\mu_{\lambda,\infty}$ among stationary states with
finite-energy Gibbs ratio
$\bar\rho^\lambda/\rli\in H^1(\mu_{\lambda,\infty})$.
Moreover, $\mu_{\lambda,\infty}$ is the unique
invariant probability measure of the diffusion \eqref{eq:linear-sde}.
The Gibbs measures converge at the rates of
Propositions~\ref{thm:linear-tv} and \ref{thm:linear-w2}.
\end{enumerate}
\end{corollary}

{
\begin{proof}
For $W\equiv0$ one has $b[\mu]=\nabla f$, so (i) is Theorem~\ref{thm:main},
and (ii) is Corollary~\ref{cor:probabilistic}(i) with the reflected
McKean--Vlasov SDE reduced to the classical reflected Langevin equation.

For (iii), the identities
\[
  \nabla\rli+\rli\nabla(f+U_\lambda)=0\quad\text{on }\Rd,
  \qquad
  \nabla\ri+\ri\nabla f=0\quad\text{on }K
\]
show stationarity.  To prove uniqueness on $K$, let $g=Z^{-1}e^{-f}$ and
write a stationary variational density as $\bar\rho=gv$.  Since $g$ is
bounded above and below on $K$, $v\in H^1(K)$, and the stationary equation
becomes $\int_K g\nabla v\cdot\nabla\varphi\dd x=0$ for all $\varphi\in H^1(K)$.
Taking $\varphi=v$ makes $v$ constant; normalization gives
$\bar\rho=\ri$.

Likewise, if $v=\bar\rho^\lambda/\rli$ has finite Gibbs energy, stationarity
is $\int\rli\nabla v\cdot\nabla\varphi=0$.  Testing with
$\chi_R^2v$, where $\chi_R=1$ on $B_R$ and $|\nabla\chi_R|\le2/R$, gives
\[
  \int_{\Rd}\rli\chi_R^2|\nabla v|^2\dd x
  \le \frac{C}{R^2}\int_{\Rd}\rli v^2\dd x.
\]
Letting $R\to\infty$ shows that $v$ is constant and hence
$\bar\rho^\lambda=\rli$.

To obtain uniqueness among \emph{all} invariant probabilities, fix
$\lambda>0$.  The drift $-\nabla f-\nabla U_\lambda$ is globally Lipschitz,
so the diffusion is nonexplosive.  Its nondegenerate additive noise makes
the transition semigroup strong Feller and topologically irreducible
(by the local uniformly elliptic transition-kernel theory).
The standard uniqueness criterion for invariant probabilities therefore
applies; see \cite[Corollary~2.7]{Hairer2021}.
The convergence rates are exactly
Propositions~\ref{thm:linear-tv} and \ref{thm:linear-w2}.
\end{proof}
}

{
\begin{remark}[Growing potentials]\label{rem:linear-truncation}
The static Gibbs estimates and Proposition~\ref{prop:linear-domination}
require only $f\in C^2(\Rd)$ and $\inf f> -\infty$, so they include
quadratic confinement.  The finite-time limit above additionally assumes
$\nabla f\in C_b^1$, as required by Assumption~\ref{ass:coefficients}.
Truncating an unbounded drift outside a neighborhood of $K$ does not by
itself justify the same dynamical conclusion: the stationary exterior-mass
bound does not control the probability that a path exits that neighborhood.
Such an extension requires a separate stopped-process estimate.
\end{remark}
}

\section{Numerical illustration}\label{sec:numerics}

Deterministic computations examine the boundary layer, the reaction-measure
limit of Theorem~\ref{thm:main}, and the stationary Gibbs rates, without
particle-sampling fluctuations.

\subsection{Setup}\label{subsec:num-setup}

For the dynamical experiment, we use the MMD data of
Example~\ref{ex:mmd} on $K=[-1,1]^2$. The Gaussian kernel has width
$\sigma=0.25$ and amplitude $\alpha=6$. The target is the normalized
restriction to $K$ of an equal-weight mixture of four Gaussians of width
$0.12$ centered at $(\pm3/4,\pm3/4)$; the initial density is a centered
Gaussian of width $0.2$ conditioned on $K$, and $T=1$.

\begin{remark}[Scope of the box computations]\label{rem:polytopes}
The square has corners and does not satisfy
Assumption~\ref{ass:geometry}. These computations examine whether the
smooth-domain scalings persist on this nonsmooth domain; they are not
covered by the theorems as stated.
\end{remark}

The script \texttt{exp2\_dynamic.py} solves \eqref{eq:pen-pde} in
the conjugated variable $\hh^\lambda=\rho^\lambda/\wl$ on an enlarged
box with an exterior collar of at least $4\sqrt\lambda$ in each coordinate
direction. A finite-volume scheme with geometric-mean face weights treats
the weighted diffusion $\divg(\wl\nabla\hh^\lambda)$ implicitly and the
bounded drift explicitly by first-order upwind transport, with FFT
convolution refreshed every two steps. The mesh is $h=2/256$ on $K$, the time step is
$2.5\times10^{-4}$, and
\[
\lambda\in\{3\times10^{-2},10^{-2},3\times10^{-3},10^{-3},3\times10^{-4}\}.
\]
The reflected reference uses the same scheme with $\wl\equiv1$ on $K$
and zero flux through its boundary faces. Double precision is used, and
mass is conserved to $10^{-8}$.

\subsection{Density errors and the Gaussian layer}\label{subsec:dynamic-rates}

Figure~\ref{fig:part2-rates} shows the density errors and exterior
moments. The supremum-in-time exterior mass has fitted slope $0.41$;
the $q=1$ and $q=2$
exterior moments have slopes $0.91$ and $1.41$. These approach the
exponents $1/2$, $1$, and $3/2$ of Corollary~\ref{cor:moments}.
The space--time $L^2(K)$ density error and final-time whole-space $L^1$
error have slopes $0.40$ and $0.41$, respectively. The normalized exterior
profiles collapse onto $e^{-u^2/2}$ in the variable
$u=\dist(x,K)/\sqrt\lambda$, resolving the predicted Gaussian layer.

\begin{figure}[t]
\centering
\includegraphics[width=0.495\textwidth]{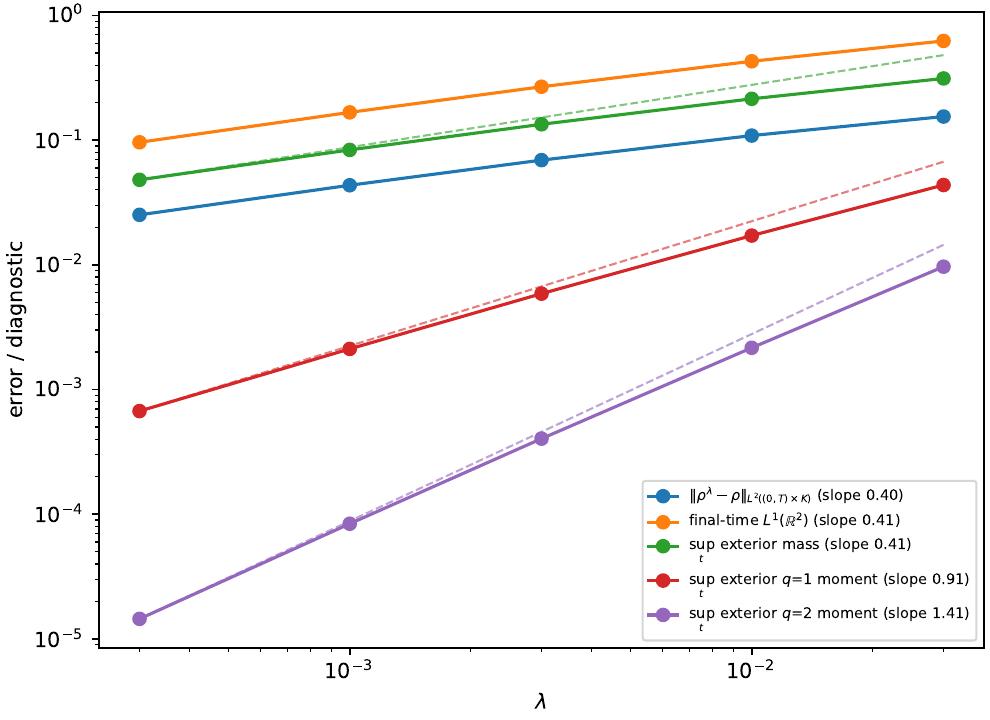}\hfill
\includegraphics[width=0.495\textwidth]{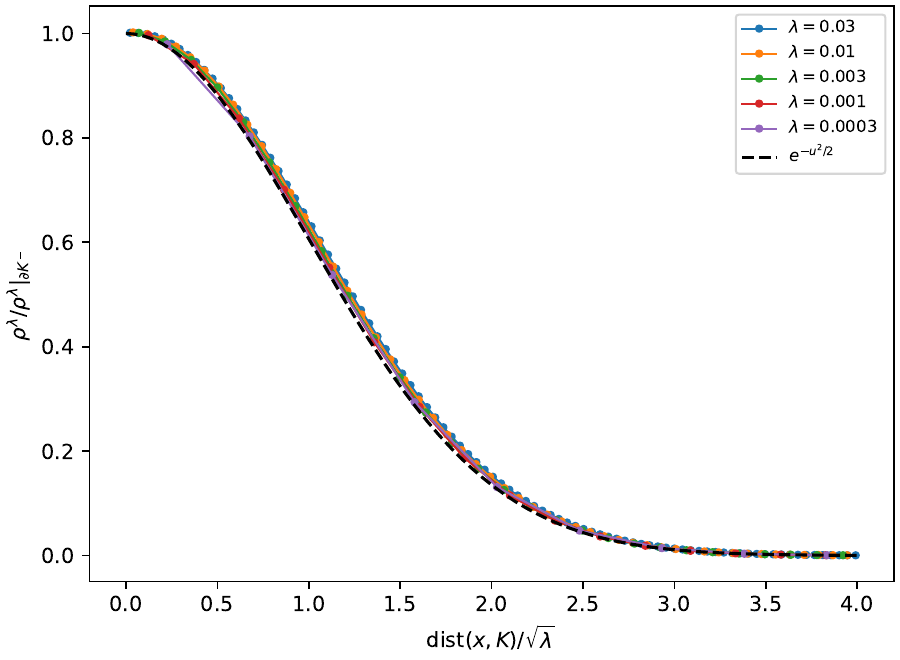}
\caption{Deterministic penalized solves. Left: density errors and exterior
moments versus $\lambda$, with the reference slopes of
Corollary~\ref{cor:moments}. Right: normalized boundary-layer profiles
against $e^{-u^2/2}$, where $u=\dist/\sqrt\lambda$.}
\label{fig:part2-rates}
\end{figure}

\subsection{The boundary reaction}\label{subsec:part2}

Figure~\ref{fig:part2-reaction} compares the surface density obtained by
integrating $q^\lambda$ across the exterior layer with
$\int_0^T\Tr(\rho)\dd t$, computed by second-order trace extrapolation
from the reflected reference. The $L^1(\Gamma)$ profile error has slope
$0.42$. A time-weighted scalar functional and a vector functional, using
a Gaussian test function centered on the right edge, have error slopes
of approximately $0.53$. The fraction of
$|Q^\lambda|$ carried by the exterior corner quadrants decays with slope
$0.45$. These diagnostics support concentration of the reaction on the
boundary with the predicted trace density.

\begin{figure}[t]
\centering
\includegraphics[width=0.55\textwidth]{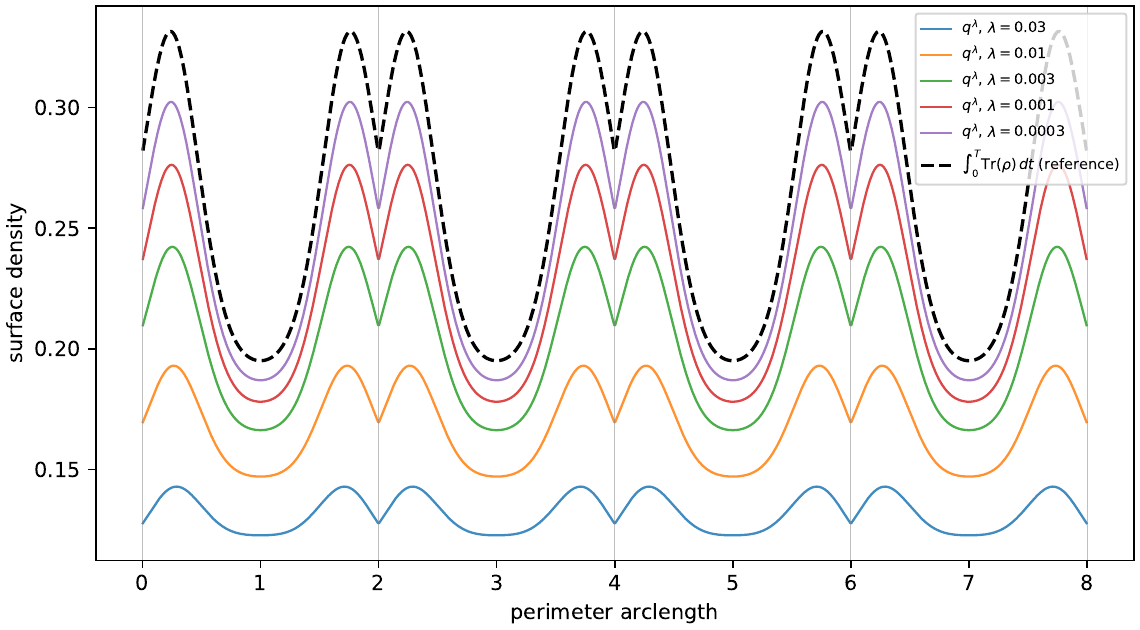}\hfill
\includegraphics[width=0.44\textwidth]{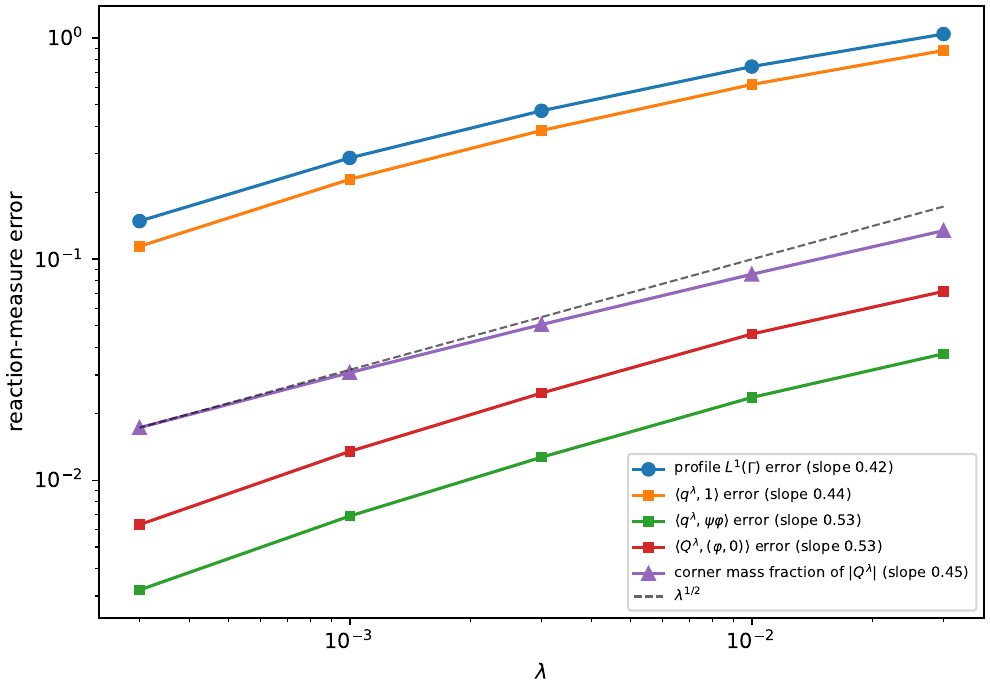}
\caption{Boundary reaction diagnostics. Left: perimeter-unrolled surface
density of $q^\lambda$ against $\int_0^T\Tr(\rho)\dd t$ from the reflected
reference; vertical lines mark corners. Right: errors in the surface
profile and smooth scalar/vector test functionals, and the corner mass
fraction of $|Q^\lambda|$.}
\label{fig:part2-reaction}
\end{figure}

\subsection{Stationary rates}\label{subsec:stationary-numerics}

The script \texttt{exp2\_stationary.py} computes the Gibbs limits in
Figure~\ref{fig:part2-stationary}. In the kernel-free one-dimensional
case, $K=[-1,1]$ and $f(x)=\cos2x+x/2$. This potential is unbounded
below and lies outside the assumptions of
Proposition~\ref{prop:kernel--free}; the Gaussian wall nevertheless makes
the normalizer finite. The exact normalization identity
\eqref{eq:linear-tv-exact} holds to machine precision over
$\lambda\in[10^{-5},10^{-1}]$.
The fitted slopes over $\lambda\le10^{-2}$ are $0.48$, $0.51$, and $0.51$
for $L^1$, $W_1$, and $W_2$, consistent with the half-power benchmark.

In the interacting test, we use the same square-domain MMD target with
$\alpha\in\{0.25,1,6\}$. All three Gaussian kernels are positive
semidefinite, and $\kappa\le\alpha/2$ also gives Dobrushin smallness for
$\alpha=0.25,1$. On the $1040^2$ grid, the fitted $L^1$ slopes
are $0.43$, $0.43$, and $0.40$, consistent with the $\sqrt\lambda$
stationary benchmark; damped iteration is used for the two larger amplitudes.

\begin{figure}[t]
\centering
\includegraphics[width=0.495\textwidth]{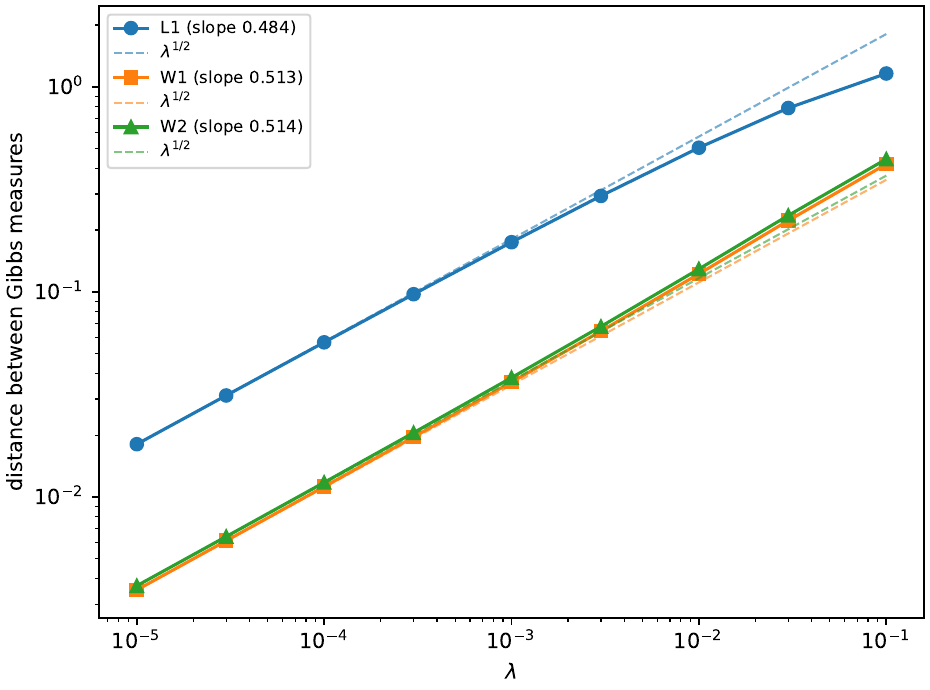}\hfill
\includegraphics[width=0.495\textwidth]{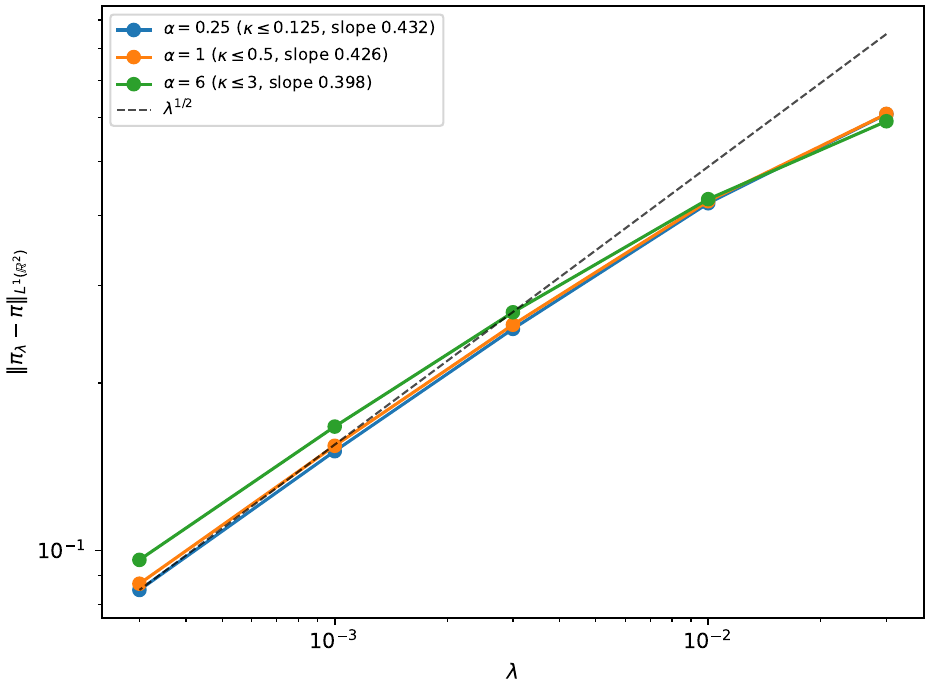}
\caption{Stationary soft-to-hard computations. Left: $L^1$, $W_1$, and
$W_2$ distances for the one-dimensional kernel-free example, with
half-power reference slopes. Right: interacting MMD Gibbs fixed points
at amplitudes $\alpha=0.25,1,6$. The potential and domain limitations
are described in the text.}
\label{fig:part2-stationary}
\end{figure}

\section{Conclusion}\label{sec:conclusion}

We have analyzed the canonical Moreau--Yosida approximation of reflected
McKean--Vlasov dynamics at the level of densities, full fluxes, and boundary
measures.  The penalized solutions converge strongly in the interior, in
whole-space $L^1$, and narrowly uniformly in time.  The singular penalty
force does not disappear: it converges to
$\Tr(\rho)n\,\Hm\!\restriction_{\partial K}\,\dd t$, and this reaction
cancels exactly the surface singularity of the zero-extended diffusive
gradient.  The resulting identity isolates the mechanism by which a smooth
whole-space dynamics generates the reflected no-flux equation.

{The proof separates three roles: a drift-uniform Gaussian
barrier resolves the exterior layer, weighted conjugation gives compactness,
and exact normal factorization with Gauss--Green identifies the reaction.
The stationary theory gives sharp Gibbs rates under Dobrushin smallness or
positive semidefiniteness, while the synchronous coupling gives path-space
penalization and $\lambda$-uniform propagation of chaos.}

{Natural next steps are uniform-in-time interacting estimates
and extensions to nonsmooth or time-dependent feasible sets.}  Quantitative
strong convergence and the vanishing interior normal trace are developed in
the companion paper~\cite{SeoKimChoiStrong}.

\appendix

\section{Convex geometry and Gaussian layer estimates}\label{app:geometry}

We collect the geometric facts used in the body; the proofs are classical,
and we indicate the argument.

\begin{lemma}\label{lem:convex-geom}
Let $K$ satisfy Assumption~\ref{ass:geometry}.  Then:
(i) $\proj_K$ is single-valued and $1$-Lipschitz on $\Rd$, and
$\dist^2\in C^{1,1}(\Rd)$ with $\nabla\dist^2=2(x-\proj_Kx)$;
(ii) $\dist\in C^2(\{0<\dist<\delta_0\})$ for some $\delta_0>0$, with
$|D^2\dist|\le C_K$ there, and
$\nabla\dist(x)=n(\proj_Kx)$ on $\{0<\dist\le\delta_0\}$;
(iii) on $\{\dist\ge\delta_0\}$, $\nabla\dist$ is Lipschitz with constant
$2/\delta_0$ and $0\le\Delta\dist\le(d-1)/\dist$ a.e.
\end{lemma}

\begin{proof}[Sketch]
(i) is the standard convex projection theorem.  (ii) is the tubular
neighborhood theorem for $C^2$ hypersurfaces, with
$D^2\dist$ expressed through the principal curvatures
$\kappa_i\ge0$ (convexity) as
$\Delta\dist=\sum_i\kappa_i/(1+s\kappa_i)$.  (iii): monotonicity of the
gradient of the convex function $\dist$ gives $\Delta\dist\ge0$; the upper
bound follows from the same curvature formula, each term being at most
$1/s$; the Lipschitz bound for $\nabla\dist=(x-\proj_Kx)/\dist$ follows from
(i).
\end{proof}

\begin{lemma}\label{lem:layer}
Let $\beta\ge0$ be fixed. For every $q\ge0$ there exists
$C_q=C_q(K,d,\beta)$ such that, for $\lambda\in(0,1]$,
\[
  \int_{K^c}\dist(x)^q
  e^{-\dist(x)^2/(2\lambda)+\beta\dist(x)}\,\dd x
  \le C_q\lambda^{(q+1)/2}.
\]
Moreover, there exists $C=C(K,d,\beta)$ such that, for every
$\delta\ge0$,
\[
  \int_{K^c\cap\{\dist\ge\delta\}}
  e^{-\dist^2/(2\lambda)+\beta\dist}\,\dd x
  \le C e^{-\delta^2/(4\lambda)}.
\]
In particular, $\muh(K^c)\le C\sqrt\lambda$.
\end{lemma}

\begin{proof}
By the coarea formula, $|\nabla\dist|=1$ a.e.\ on $K^c$, and the
Steiner formula gives, for $s>0$,
\[
  \Hm(\{\dist=s\})\le C_K(1+s)^{d-1}.
\]
Hence
\[
\begin{aligned}
\int_{K^c}\dist(x)^q
 e^{-\dist(x)^2/(2\lambda)+\beta\dist(x)}\,\dd x
&\le
 C\int_0^\infty
 s^q e^{-s^2/(2\lambda)+\beta s}(1+s)^{d-1}\,\dd s\\
&=
 C\lambda^{(q+1)/2}
 \int_0^\infty
 t^q e^{-t^2/2+\beta\sqrt\lambda\,t}
 (1+\sqrt\lambda\,t)^{d-1}\,\dd t.
\end{aligned}
\]
Since
\[
  \beta\sqrt\lambda\,t\le \frac{t^2}{4}+\beta^2\lambda
  \quad\text{and}\quad
  1+\sqrt\lambda\,t\le1+t,
\]
the last integral is bounded uniformly in $\lambda\in(0,1]$, proving
the first estimate.

For the tail, Young's inequality gives
\[
  \beta s\le\frac{s^2}{8\lambda}+2\beta^2\lambda.
\]
Thus
\[
\begin{aligned}
&\int_{K^c\cap\{\dist\ge\delta\}}
 e^{-\dist^2/(2\lambda)+\beta\dist}\,\dd x\\
&\qquad\le
 C\int_\delta^\infty
 e^{-3s^2/(8\lambda)+2\beta^2\lambda}(1+s)^{d-1}\,\dd s\\
&\qquad\le
 C e^{-\delta^2/(4\lambda)}
 \int_0^\infty
 e^{-s^2/(8\lambda)}(1+s)^{d-1}\,\dd s\\
&\qquad\le C e^{-\delta^2/(4\lambda)},
\end{aligned}
\]
where we used $\lambda\le1$ in the last step.
\end{proof}

\section{Well-posedness of the penalized dynamics}\label{app:wp}

\begin{proof}[Proof of Proposition~\ref{prop:pen-wp}]
Fix $\lambda>0$.  In both initializations of the proposition, the datum
is a probability density and \(\rho_0/w_\lambda\in L^2(\muh)\);
these are the only initial-data properties used below.  Write \(\rho_0\)
for either datum and let
$X_T:=\{\mu\in C([0,T];(\mcP(\Rd),\|\cdot\|_{\TV})):\mu_0=\rho_0\}$, a
complete metric space under
$d_T(\mu,\nu):=\sup_{t\le T}\|\mu_t-\nu_t\|_{\TV}$.

\textbf{Step 1: the frozen-measure SDE.}
Given $\mu\in X_T$, the SDE
\[
  \dd X_t=-\nabla U_\lambda(X_t)\dd t-b[\mu_t](X_t)\dd t+\sqrt2\dd B_t,
  \qquad\mathrm{Law}(X_0)=\rho_0,
\]
has drift Lipschitz in $x$ (constant $\lambda^{-1}+B_1$), bounded on
bounded sets, and measurable in $t$ ($t\mapsto b[\mu_t](x)$ is continuous
by dominated convergence from $\mu\in X_T$ and $W\in C^1_b$); hence a
unique strong solution.  Set $\mcS(\mu)_t:=\mathrm{Law}(X_t)$.  Its
marginals have densities: by Girsanov's theorem, $\mathrm{Law}(X_t)$ is
absolutely continuous with respect to the marginal of the
$U_\lambda$-Langevin diffusion (bounded drift perturbation), which has a
density for $t>0$ and, at $t=0$, by assumption.  Moreover
$t\mapsto\mcS(\mu)_t$ is $\TV$-continuous: the density solves a linear
Fokker--Planck equation with locally bounded coefficients and is continuous
in $L^1$ by parabolic theory (or by the strong continuity in $L^1$ of the
associated evolution family).  Thus $\mcS:X_T\to X_T$.

\textbf{Step 2: contraction via Girsanov and Pinsker.}
Let $\mu^1,\mu^2\in X_T$ and let $P^1,P^2$ denote the path laws on
$C([0,T];\Rd)$ of the corresponding frozen-measure diffusions, with the
same initial law.  The two drifts differ by
$\Delta b(t,x):=b[\mu^1_t](x)-b[\mu^2_t](x)$, and by
\eqref{eq:drift-bounds},
\[
  \|\Delta b(t,\cdot)\|_\infty\ \le\ L_W\,\|\mu^1_t-\mu^2_t\|_{\TV}.
\]
Both diffusions have the same (constant) diffusion matrix $2\,\mathrm{Id}$
and bounded-difference drifts, so Girsanov's theorem applies and gives the
path-space relative entropy bound
\[
  H\bigl(P^1\,\big|\,P^2\bigr)
  =\frac1{4}\,\mathbb E^{P^1}\!\int_0^T\bigl|\Delta b(t,X_t)\bigr|^2\dd t
  \ \le\ \frac{1}{4}\int_0^T\|\Delta b(t,\cdot)\|_\infty^2\dd t
  \ \le\ \frac{L_W^2\,T}{4}\,d_T(\mu^1,\mu^2)^2 .
\]
Marginals are images of path laws under evaluation, and relative entropy
decreases under images; hence by Pinsker's inequality (total-mass
convention, $\|\cdot\|_{\TV}\le\sqrt{2H}$),
\[
  d_T\bigl(\mcS(\mu^1),\mcS(\mu^2)\bigr)
  \le\sup_{t\le T}\sqrt{2H(P^1_t|P^2_t)}
  \le\sqrt{2H(P^1|P^2)}
  \le L_W\sqrt{T/2}\;d_T(\mu^1,\mu^2).
\]
For $T_0<2/L_W^2$ the map $\mcS$ is a strict contraction on $X_{T_0}$, and
Banach's fixed point theorem produces the unique law of the McKean--Vlasov
SDE on $[0,T_0]$; since $T_0$ does not depend on the initial datum, the
argument restarts at $T_0,2T_0,\dots$ and covers $[0,T]$.  (We emphasize
that the entire fixed-point argument is run in the $\TV$ metric; no claim
that $W_1$ controls $\TV$ is used at any point.)

\textbf{Step 3: PDE class and energy identity.}
Freeze the drift $b(t,x):=b[\mu_t](x)$ along the fixed point; it is $C^1_b$
in $x$ with the bounds $B_0,B_1$, and continuous in $t$ (by
$\TV$-continuity of $t\mapsto\mu_t$ and $W\in C^2_b$).  We solve the ratio
equation \eqref{eq:ratio-eq} in the Gelfand triple
\[
  V:=H^1(\muh)\ \hookrightarrow\ H:=L^2(\muh)\ \hookrightarrow\ V',
\]
where $V$ and $H$ are separable Hilbert spaces and the embeddings are
dense and continuous.  Three facts about the weight are used.  First,
$C_c^\infty(\Rd)$ is dense in $V$: for $u\in V$, the truncations
$u\chi_R$ (with $|\nabla\chi_R|\le2/R$) converge to $u$ in $V$ by
dominated convergence, since $\int(u^2+|\nabla u|^2)\dd\muh<\infty$ makes
the tail and cutoff-gradient errors vanish, and a compactly supported
$H^1$ function is approximated by mollification, the weight
$\wl=e^{-U_\lambda}$ being continuous and locally bounded above and below
by positive constants.  Second, for a.e.\ $t$ the bilinear form
\[
  a_t(u,v):=\int_{\Rd}\nabla u\cdot\nabla v\dd\muh
  +\int_{\Rd}u\,b(t,\cdot)\cdot\nabla v\dd\muh
\]
is bounded on $V\times V$, $|a_t(u,v)|\le(1+B_0)\|u\|_V\|v\|_V$, with
$t\mapsto a_t(u,v)$ continuous for fixed $u,v$.  Third, $a_t$ satisfies
the G\aa rding inequality
\[
  a_t(u,u)\ \ge\ \|\nabla u\|_{H}^2-B_0\|u\|_H\|\nabla u\|_H
  \ \ge\ \tfrac12\|\nabla u\|_{H}^2-\tfrac{B_0^2}2\|u\|_H^2 .
\]
By the classical existence theory for nonautonomous parabolic equations in
Gelfand triples (J.-L. Lions; see \cite[Ch.~III]{LionsMagenes}, whose
Galerkin construction --- approximation on finite-dimensional subspaces of
$V$, uniform energy bounds from the G\aa rding inequality, and passage to
the limit --- applies verbatim to the triple above), for every
$\hh_0\in H$ there is a unique
\[
  \hh\in L^2(0,T;V)\cap C([0,T];H),
  \qquad
  \partial_t\hh\in L^2(0,T;V'),
\]
such that $\langle\partial_t\hh(t),v\rangle_{V',V}+a_t(\hh(t),v)=0$ for
a.e.\ $t$ and all $v\in V$, and the Lions--Magenes chain rule
$\frac{\dd}{\dd t}\tfrac12\|\hh_t\|_H^2
=\langle\partial_t\hh(t),\hh(t)\rangle_{V',V}$ applied with $v=\hh(t)$
yields exactly the energy identity \eqref{eq:energy-identity}.  Undoing
the ratio, $\rho^\lambda:=\hh\,\wl$, and using the density of
$C_c^\infty(\Rd)$ in $V$, the $V'$-formulation coincides with the
distributional formulation \eqref{eq:pen-weak}.  At this stage
$\rho^\lambda$ satisfies the energy and weak-formulation requirements of
Definition~\ref{def:pen-sol}, but not yet its measure-theoretic ones:
nonnegativity, unit mass, and continuity as an $L^1$-valued probability
density are obtained only in Step 4 below, after which $\rho^\lambda$ is a
penalized solution in the full sense of Definition~\ref{def:pen-sol}.
Uniqueness within that class is Lemma~\ref{lem:pen-stability} (whose proof
uses only the definition).

\textbf{Step 4: identification with the SDE marginals.}
The marginal flow $t\mapsto\mu_t$ of Steps 1--2 also solves the frozen
linear Fokker--Planck equation weakly; we now check that it coincides with
the energy solution $\rho^\lambda$ of Step 3.  This identification is what
supplies nonnegativity, unit mass, and $\TV$-continuity of the energy
solution.  For fixed $\lambda$ the frozen drift
$-\nabla U_\lambda-b(t,\cdot)$ is globally Lipschitz in $x$ (constant
$\lambda^{-1}+B_1$), uniformly in $t$, and continuous in $t$.  For such
coefficients the Cauchy problem for the Fokker--Planck equation has
\emph{at most one} solution in the class of vaguely continuous flows
$t\mapsto\sigma_t$ of finite signed measures with
$\sup_{t\le T}|\sigma_t|(\Rd)<\infty$ and uniformly vanishing tails: this
is the {standard backward-parabolic duality argument}
with the backward Kolmogorov equation --- mollifying the drift preserves the Lipschitz constant, the
mollified terminal-value problems with data
$\varphi\in C_c^\infty(\Rd)$ have bounded classical solutions whose
spatial gradients are bounded uniformly in the mollification (via the
Feynman--Kac representation along the Lipschitz stochastic flow), and
testing the difference of two solutions with these backward solutions and
passing to the limit annihilates it against every such $\varphi$; see
\cite[Ch.~9]{BKRS} for the general uniqueness theory of
Fokker--Planck--Kolmogorov equations, and note that, probabilistically,
the statement is uniqueness for the martingale problem with globally
Lipschitz coefficients.  Both flows belong to this class.  The SDE
marginals are probability measures, $\TV$-continuous by Step 1, and
uniformly tight (a $\TV$-continuous curve on $[0,T]$ is a compact set of
probability measures).  The energy solution satisfies, by Cauchy--Schwarz
against $\muh$,
\[
  \|\rho^\lambda_t\|_{L^1(\Rd)}\le\|\hh_t\|_H\,\muh(\Rd)^{1/2},
  \qquad
  \int_{\{|x|\ge R\}}|\rho^\lambda_t|\dd x
  \le\|\hh_t\|_H\,\muh(\{|x|\ge R\})^{1/2}
  \ \xrightarrow[R\to\infty]{}\ 0
\]
uniformly in $t$ with $\sup_{t\le T}\|\hh_t\|_H<\infty$, and $t\mapsto\rho^\lambda_t$ is
$L^1$-continuous by $\hh\in C([0,T];H)$ and the same bound; moreover the
compactly supported test functions of \eqref{eq:pen-weak} extend to the
bounded backward test functions by the cutoff argument already used in
this proof, the flux being time-integrable at fixed $\lambda$,
$\int_0^T\!\!\int_{\Rd}|J^\lambda|\dd x\dd t
\le T^{1/2}\bigl(\int_0^T\!\!\int|\nabla\hh|^2\dd\muh\dd t\bigr)^{1/2}
\muh(\Rd)^{1/2}
+B_0T\sup_{t\le T}\|\rho_t^\lambda\|_{L^1(\Rd)}<\infty$.
Hence $\rho^\lambda_t=\mu_t$ for every $t\in[0,T]$; in particular the
energy solution of Step 3 is a nonnegative probability density,
continuous in $L^1$, and is therefore a penalized solution in the full
sense of Definition~\ref{def:pen-sol}.
\end{proof}

\section{Comparison principle}\label{app:comparison}

\begin{lemma}\label{lem:comparison}
Let $b\in L^\infty((0,T)\times\Rd;\Rd)$ with
{$b(t,\cdot)\in\mathcal B$ for a.e.\ $t$}, let $\rho$ be a weak solution
of $\partial_t\rho=\divg(\nabla\rho+\rho\nabla U_\lambda+\rho b)$ in the
class of Definition~\ref{def:pen-sol}, and let
$\bar R=Me^{\gamma t}w_\beta$ satisfy \eqref{eq:supersol} with
$\rho(0)\le\bar R(0)$ a.e.  Then $\rho\le\bar R$ a.e.\ on
$[0,T]\times\Rd$.
\end{lemma}

\begin{proof}
Let $h:=\rho/w_\beta$ and $\bar h:=Me^{\gamma t}$.  Since
$\nabla\rho+\rho\nabla U_\lambda+\rho b
=w_\beta\bigl(\nabla h+h\,c\bigr)$ with
$c:=\beta\nabla\zeta+b$ bounded, $h$ solves
\begin{equation}\label{eq:h-beta-eq}
  w_\beta\,\partial_th=\divg\bigl(w_\beta\nabla h+w_\beta h\,c\bigr)
\end{equation}
weakly, while \eqref{eq:supersol} states precisely that
$\divg(w_\beta c)\le\gamma w_\beta$ distributionally, i.e.\ that
$\bar h$ is a supersolution of \eqref{eq:h-beta-eq}:
$w_\beta\partial_t\bar h-\divg(w_\beta\nabla\bar h+w_\beta\bar h c)
=\bar h\,(\gamma w_\beta-\divg(w_\beta c))\ge0$.  Set
$z:=(h-\bar h)_+$.  To justify the positive-part test, note that
$h=\hh e^{-\beta\zeta}$, $\zeta\ge0$, and $|\nabla\zeta|\le1$ give
\[
\begin{aligned}
  \int h^2w_\beta&\le\int\hh^2\wl,\\
  \int|\nabla h|^2w_\beta
  &\le2\int\bigl(|\nabla\hh|^2+\beta^2\hh^2\bigr)\wl.
\end{aligned}
\]
The solution class therefore controls $h$ in
$L^\infty(0,T;L^2(w_\beta))\cap L^2(0,T;H^1(w_\beta))$.
For $R>0$ let $\chi_R$ be smooth cutoffs, equal to one on $B_R$ and zero
outside $B_{2R}$, with $|\nabla\chi_R|\le2/R$.
Testing the difference of the two relations with $z\chi_R^2\ge0$
(using time regularization) gives
\[
  \frac12\frac{\dd}{\dd t}\int w_\beta z^2\chi_R^2
  +\int w_\beta|\nabla z|^2\chi_R^2
  \le\int w_\beta|z||c||\nabla z|\chi_R^2
  +\Theta_R(t),
\]
where $\Theta_R$ collects the cutoff terms,
$\Theta_R(t)\le\frac CR\int_{B_{2R}\setminus B_R}
w_\beta\bigl(z|\nabla z|+z^2(1+|c|)\bigr)$, which is integrable in $t$ and
tends to $0$ as $R\to\infty$ by the global bounds
$\int w_\beta z^2\le C$, $\int_0^T\!\int w_\beta|\nabla z|^2<\infty$
(by the preceding estimates and
$\sup_{t\le T}\int\bar h^2w_\beta\le M^2e^{2\gamma T}\int w_\beta<\infty$).  Young's inequality and
Gr\"onwall, then $R\to\infty$ by monotone convergence, give
$\int w_\beta z^2(t)\le e^{Ct}\int w_\beta z^2(0)=0$, i.e.\ $h\le\bar h$.
\end{proof}

{\small
\bibliographystyle{plainnat}
\bibliography{ref_restructured}

\begin{thebibliography}{28}
\providecommand{\natexlab}[1]{#1}
\providecommand{\url}[1]{\texttt{#1}}
\expandafter\ifx\csname urlstyle\endcsname\relax
  \providecommand{\doi}[1]{doi: #1}\else
  \providecommand{\doi}{doi: \begingroup \urlstyle{rm}\Url}\fi

\bibitem[Alasio et~al.(2020)Alasio, Bruna, and
  Carrillo]{AlasioBrunaCarrillo2020}
Luca Alasio, Maria Bruna, and Jos{\'e}~A. Carrillo.
\newblock The role of a strong confining potential in a nonlinear
  {F}okker--{P}lanck equation.
\newblock \emph{Nonlinear Analysis}, 193:\penalty0 111480, 2020.
\newblock \doi{10.1016/j.na.2019.03.003}.

\bibitem[Alexander and Kim(2016)]{AlexanderKim2016}
Damon Alexander and Inwon Kim.
\newblock A {F}okker--{P}lanck type approximation of parabolic {PDE}s with
  oblique boundary data.
\newblock \emph{Transactions of the American Mathematical Society},
  368\penalty0 (8):\penalty0 5753--5781, 2016.
\newblock \doi{10.1090/tran/6521}.

\bibitem[Ambrosio et~al.(2005)Ambrosio, Gigli, and Savar{\'e}]{AGS}
Luigi Ambrosio, Nicola Gigli, and Giuseppe Savar{\'e}.
\newblock \emph{Gradient Flows in Metric Spaces and in the Space of Probability
  Measures}.
\newblock Birkh{\"a}user, Basel, 2005.
\newblock \doi{10.1007/b137080}.

\bibitem[Arbel et~al.(2019)Arbel, Korba, Salim, and
  Gretton]{ArbelKorbaSalimGretton2019}
Michael Arbel, Anna Korba, Adil Salim, and Arthur Gretton.
\newblock Maximum mean discrepancy gradient flow.
\newblock In \emph{Advances in Neural Information Processing Systems},
  volume~32, 2019.

\bibitem[Bogachev et~al.(2015)Bogachev, Krylov, R{\"o}ckner, and
  Shaposhnikov]{BKRS}
Vladimir~I. Bogachev, Nicolai~V. Krylov, Michael R{\"o}ckner, and Stanislav~V.
  Shaposhnikov.
\newblock \emph{Fokker--Planck--Kolmogorov Equations}, volume 207 of
  \emph{Mathematical Surveys and Monographs}.
\newblock American Mathematical Society, Providence, RI, 2015.
\newblock \doi{10.1090/surv/207}.

\bibitem[Brosse et~al.(2017)Brosse, Durmus, Moulines, and
  Pereyra]{BrosseDurmusMoulinesPereyra2018}
Nicolas Brosse, Alain Durmus, {\'E}ric Moulines, and Marcelo Pereyra.
\newblock Sampling from a log-concave distribution with compact support with
  proximal {L}angevin {M}onte {C}arlo.
\newblock In \emph{Proceedings of the 30th Conference on Learning Theory},
  volume~65 of \emph{Proceedings of Machine Learning Research}, pages 319--342.
  PMLR, 2017.

\bibitem[Bubeck et~al.(2018)Bubeck, Eldan, and Lehec]{BubeckEldanLehec2018}
S{\'e}bastien Bubeck, Ronen Eldan, and Joseph Lehec.
\newblock Sampling from a log-concave distribution with projected {L}angevin
  {M}onte {C}arlo.
\newblock \emph{Discrete \& Computational Geometry}, 59\penalty0 (4):\penalty0
  757--783, 2018.
\newblock \doi{10.1007/s00454-018-9992-1}.

\bibitem[Carrillo et~al.(2003)Carrillo, McCann, and
  Villani]{CarrilloMcCannVillani2003}
Jos{\'e}~A. Carrillo, Robert~J. McCann, and C{\'e}dric Villani.
\newblock Kinetic equilibration rates for granular media and related equations:
  entropy dissipation and mass transportation estimates.
\newblock \emph{Revista Matem{\'a}tica Iberoamericana}, 19\penalty0
  (3):\penalty0 971--1018, 2003.
\newblock \doi{10.4171/RMI/376}.

\bibitem[Choi et~al.(2026)Choi, Kim, and Seo]{SeoKimChoiStrong}
Beomjun Choi, Yeoneung Kim, and Geuntaek Seo.
\newblock Soft-confinement approximation of no-flux {M}c{K}ean--{V}lasov
  dynamics: strong convergence and long-time behavior.
\newblock Companion paper in preparation, 2026.

\bibitem[G{\"u}rb{\"u}zbalaban et~al.(2024)G{\"u}rb{\"u}zbalaban, Hu, and
  Zhu]{GurbuzbalabanHuZhu2024}
Mert G{\"u}rb{\"u}zbalaban, Yuanhan Hu, and Lingjiong Zhu.
\newblock Penalized overdamped and underdamped {L}angevin {M}onte {C}arlo
  algorithms for constrained sampling.
\newblock \emph{Journal of Machine Learning Research}, 25\penalty0
  (263):\penalty0 1--67, 2024.

\bibitem[Hairer(2021)]{Hairer2021}
Martin Hairer.
\newblock Convergence of {M}arkov processes.
\newblock Lecture notes, 2021.
\newblock May 26, 2021. \url{https://www.hairer.org/notes/Convergence.pdf}.

\bibitem[Hinds et~al.(2025)Hinds, Sharma, and
  Tretyakov]{HindsSharmaTretyakov2025}
Piers~D. Hinds, Akash Sharma, and Michael~V. Tretyakov.
\newblock Well-posedness and approximation of reflected {M}c{K}ean--{V}lasov
  {SDE}s with applications.
\newblock \emph{Mathematical Models and Methods in Applied Sciences},
  35\penalty0 (8):\penalty0 1845--1887, 2025.
\newblock \doi{10.1142/S0218202525500241}.

\bibitem[Jordan et~al.(1998)Jordan, Kinderlehrer, and Otto]{JKO}
Richard Jordan, David Kinderlehrer, and Felix Otto.
\newblock The variational formulation of the {F}okker--{P}lanck equation.
\newblock \emph{SIAM Journal on Mathematical Analysis}, 29\penalty0
  (1):\penalty0 1--17, 1998.
\newblock \doi{10.1137/S0036141096303359}.

\bibitem[Laayoun and Missaoui(2026)]{LaayounMissaoui2026}
Ayoub Laayoun and Badr Missaoui.
\newblock Mean field games with reflected dynamics: penalization and relaxed
  control approach.
\newblock arXiv:2603.15477, 2026.

\bibitem[Lamperski(2021)]{Lamperski2021}
Andrew Lamperski.
\newblock Projected stochastic gradient {L}angevin algorithms for constrained
  sampling and non-convex learning.
\newblock In \emph{Proceedings of the 34th Conference on Learning Theory},
  volume 134 of \emph{Proceedings of Machine Learning Research}, pages
  2891--2937. PMLR, 2021.

\bibitem[Lions and Magenes(1972)]{LionsMagenes}
Jacques-Louis Lions and Enrico Magenes.
\newblock \emph{Non-Homogeneous Boundary Value Problems and Applications, Vol.
  I}.
\newblock Springer, Berlin, Heidelberg, 1972.
\newblock \doi{10.1007/978-3-642-65161-8}.

\bibitem[Lions and Sznitman(1984)]{LionsSznitman1984}
Pierre-Louis Lions and Alain-Sol Sznitman.
\newblock Stochastic differential equations with reflecting boundary
  conditions.
\newblock \emph{Communications on Pure and Applied Mathematics}, 37\penalty0
  (4):\penalty0 511--537, 1984.
\newblock \doi{10.1002/cpa.3160370408}.

\bibitem[Mane and Boukardagha(2026)]{ManeBoukardagha2025}
Tarika Mane and Amine Boukardagha.
\newblock Convergence of reflected {L}angevin diffusion for constrained
  sampling.
\newblock arXiv:2512.00386v2, 2026.

\bibitem[Menaldi(1983)]{Menaldi1983}
Jos{\'e}-Luis Menaldi.
\newblock Stochastic variational inequality for reflected diffusion.
\newblock \emph{Indiana University Mathematics Journal}, 32\penalty0
  (5):\penalty0 733--744, 1983.
\newblock \doi{10.1512/iumj.1983.32.32048}.

\bibitem[Ning and Wu(2026)]{NingWu2026}
Ning Ning and Jing Wu.
\newblock Well-posedness and propagation of chaos for {M}c{K}ean--{V}lasov
  stochastic variational inequalities.
\newblock \emph{Journal of Theoretical Probability}, 39:\penalty0 Paper No.~5,
  2026.
\newblock \doi{10.1007/s10959-025-01459-0}.

\bibitem[Peyre(2018)]{Peyre2018}
R{\'e}mi Peyre.
\newblock Comparison between {$W_2$} distance and {$\dot H^{-1}$} norm, and
  localization of wasserstein distance.
\newblock \emph{ESAIM: Control, Optimisation and Calculus of Variations},
  24\penalty0 (4):\penalty0 1489--1501, 2018.
\newblock \doi{10.1051/cocv/2017050}.

\bibitem[Santambrogio(2015)]{Santambrogio}
Filippo Santambrogio.
\newblock \emph{Optimal Transport for Applied Mathematicians}.
\newblock Birkh{\"a}user, Cham, 2015.
\newblock \doi{10.1007/978-3-319-20828-2}.

\bibitem[S{\l}omi{\'n}ski(1994)]{Slominski1994}
Leszek S{\l}omi{\'n}ski.
\newblock On approximation of solutions of multidimensional {SDE}s with
  reflecting boundary conditions.
\newblock \emph{Stochastic Processes and their Applications}, 50\penalty0
  (2):\penalty0 197--219, 1994.
\newblock \doi{10.1016/0304-4149(94)90118-X}.

\bibitem[S{\l}omi{\'n}ski(2013)]{Slominski2013}
Leszek S{\l}omi{\'n}ski.
\newblock Weak and strong approximations of reflected diffusions via
  penalization methods.
\newblock \emph{Stochastic Processes and their Applications}, 123\penalty0
  (3):\penalty0 752--763, 2013.
\newblock \doi{10.1016/j.spa.2012.10.006}.

\bibitem[Sznitman(1984)]{Sznitman1984}
Alain-Sol Sznitman.
\newblock Nonlinear reflecting diffusion process, and the propagation of chaos
  and fluctuations associated.
\newblock \emph{Journal of Functional Analysis}, 56\penalty0 (3):\penalty0
  311--336, 1984.
\newblock \doi{10.1016/0022-1236(84)90080-6}.

\bibitem[Sznitman(1991)]{Sznitman1991}
Alain-Sol Sznitman.
\newblock Topics in propagation of chaos.
\newblock In \emph{\'{E}cole d'\'{E}t{\'e} de Probabilit{\'e}s de Saint-Flour
  XIX--1989}, volume 1464 of \emph{Lecture Notes in Mathematics}, pages
  165--251. Springer, Berlin, Heidelberg, 1991.
\newblock \doi{10.1007/BFb0085169}.

\bibitem[Tanaka(1979)]{Tanaka1979}
Hiroshi Tanaka.
\newblock Stochastic differential equations with reflecting boundary condition
  in convex regions.
\newblock \emph{Hiroshima Mathematical Journal}, 9\penalty0 (1):\penalty0
  163--177, 1979.
\newblock \doi{10.32917/hmj/1206135203}.

\bibitem[Yu et~al.(2026)Yu, Zhang, and Yu]{YuZhangYu2026}
Yifeng Yu, Shijie Zhang, and Lu~Yu.
\newblock Randomized midpoint method for log-concave sampling under
  constraints.
\newblock arXiv:2405.15379v3, 2026.

\end{thebibliography}
}

\section*{Statements and Declarations}
\paragraph*{Competing interests}
The authors have no relevant financial or non-financial interests to disclose.
\paragraph*{Author contributions}
Geuntaek Seo and Beomjun Choi contributed equally and share first authorship.
\paragraph*{Data availability}
The numerical experiments use synthetic distributions specified in the text.
The numerical methods and experiment parameters are described in
Section~\ref{sec:numerics}; no external datasets were used.
\end{document}